\documentclass[11pt,a4paper,oneside]{amsart}

\usepackage[dvipsnames]{xcolor}
\usepackage[pdftex,  colorlinks=true]{hyperref}
\hypersetup{urlcolor=RoyalBlue, linkcolor=RoyalBlue,  citecolor=black}
\usepackage{amssymb} 
\usepackage{amsthm} 
\usepackage[english]{babel} 
\usepackage[font=small,justification=centering]{caption} 
\usepackage[nodayofweek]{datetime}
\usepackage[T1]{fontenc}
\usepackage[a4paper]{geometry}
\usepackage[utf8]{inputenc}
\usepackage{ifthen}
\usepackage{mathtools}
\usepackage[dvipsnames]{xcolor}
\usepackage{adjustbox}
\usepackage{setspace}
\usepackage{tikz-cd}
\usepackage{tikz}
\usepackage{tikzit}
\usepackage{xfrac}
\usepackage{cleveref}
\usepackage{crossreftools}
\usepackage{stmaryrd} 
\usepackage{bm} 
\usepackage{comment} 
\usepackage{geometry}
\usepackage{enumitem} 
\usepackage{physics}
\usepackage{wasysym}
\usepackage{csquotes}
\usepackage{quoting,xparse}
\usepackage{relsize}
\usepackage[titletoc,toc,page]{appendix}

\definecolor{hotpink}{rgb}{255, 0, 136} 

\newtheorem{theorem}{Theorem}[section]
\newtheorem{proposition}[theorem]{Proposition}
\newtheorem{lemma}[theorem]{Lemma}
\newtheorem{remark}[theorem]{Remark}

\newtheorem{question}[theorem]{Question}

\newtheorem{convention}[theorem]{Convention}

\newtheorem{corollary}[theorem]{Corollary}

\newtheorem{proofclaim}{Claim}
\newtheorem{proofclaim2}{Claim}
\newtheorem{proofclaim3}{Claim}

\newtheorem*{proposition*}{Proposition}
\newtheorem*{notation*}{Notation}
\newtheorem*{lemma*}{Lemma}
\newtheorem*{fact*}{Fact}
\newtheorem*{remark*}{Remark}
\newtheorem*{bigassumption*}{Big Assumption}
\newtheorem*{smallassumption*}{Small Assumption}
\newtheorem*{question*}{Question}
\newtheorem*{theorem*}{Theorem}
\newtheorem*{propprobtrue*}{A proposition that may or may not be true}
\newtheorem*{conjecture*}{Conjecture}
\newtheorem*{claim*}{Claim}
\newtheorem*{case1*}{Case 1}
\newtheorem*{case1a*}{Case 1a}
\newtheorem*{case1b*}{Case 1b}
\newtheorem*{case2*}{Case 2}
\newtheorem*{case3*}{Case 3}
\newtheorem*{claim1*}{Claim 1}
\newtheorem*{claim2*}{Claim 2}
\newtheorem*{claim3*}{Claim 3}
\newtheorem*{claim4*}{Claim 4}
\newtheorem*{claim5*}{Claim 5}
\newtheorem*{claim6*}{Claim 6}
\newtheorem*{claim7*}{Claim 7}
\newtheorem*{claim8*}{Claim 8}
\newtheorem*{corollary*}{Corollary}
\newtheorem*{armkon*}{A remark on the proofs of the two propositions above}

\theoremstyle{definition}

\newtheorem*{example*}{Example}
\newtheorem*{definition*}{Definition}
\newtheorem{example}[theorem]{Example}

\newtheorem{definition}[theorem]{Definition}
\newtheorem{notation}[theorem]{Notation}

\tikzstyle{black dot}=[fill=black, draw=black, shape=circle, scale=0.3]

\tikzstyle{dotted?}=[-, dotted, thick]
\tikzstyle{arrow}=[draw={rgb,255: red,255; green,0; blue,136}, {|->}, fill=none]
\tikzstyle{dashed?}=[-, dashed]
\tikzstyle{blueline}=[-, draw={rgb,255: red,28; green,116; blue,222}]
\tikzstyle{pinkline}=[-, draw={rgb,255: red,255; green,0; blue,136}]
\tikzstyle{blackarrow}=[->]
\tikzstyle{greenline}=[-, draw={rgb,255: red,55; green,255; blue,0}]

\def\Z{{\mathbb Z}}
\def\N{{\mathbb N}}

\def\R{{\mathbb R}}

\def\T{{\mathcal T}}

\def\Y{{\mathbb Y}}
\def\H{{\mathbb H}}
\def\T{{\mathcal T}}
\def\r{{\mathfrak r}}
\def\q{{\mathfrak q}}
\def\F{{\mathcal F}}
\def\C{{\mathcal{C}}}

\def\A{{\mathcal{A}}}
\def\B{{\mathcal{B}}}
\def\G{{\mathbb{G}}}
\def\scal{{\mathcal{S}}}
\def\scalf{{\scal_{\textrm{finite}}}}
\def\scall{{\scal_{\textrm{linear}}}}
\def\jf{{J_{\textrm{finite}}}}
\def\jl{{J_{\textrm{linear}}}}

\def\hyp{{\mathbb H}}

\def\last{\textrm{Last}}

\def\pcg{\mathcal{P}_K(\mathbb G)}
\def\qtomsg{\mathcal{C}_K(\mathbb G)}
\def\tomsg{\mathcal{C}_K^\T(\mathbb G)}
\def\diam{\textrm{diam}}
\def\mt{\textrm{MT}}
\def\stat{\textrm{stat}}
\def\bigstat{\textrm{STAT}}
\def\trunc{\textrm{trunc}}
\def\norm{\textrm{norm}}

\def\awl{\textrm{AWL}}
\def\ad{\textrm{AD}}
\def\ld{\textrm{LD}}
\def\oop{\textrm{oop}}

\def\asdim{\textrm{asdim}}

\def\id{\textrm{id}}
\def\lca{\textrm{LCA}}
\def\df{D_\textrm{finite}}
\def\dl{D_\textrm{linear}}
\def\omegaf{\Omega_\textrm{finite}}
\def\omegal{\Omega_\textrm{linear}}
\def\mf{M_\textrm{finite}}
\def\ml{M_\textrm{linear}}
\def\asdim{\textrm{asdim}}

\def\lambdacosetsp{\lambda_1}
\def\mucosetsp{\mu_1}
\def\lambdaqg{\lambda_2}
\def\muqg{\mu_2}
\def\lambdasp{\lambda_3}
\def\musp{\mu_3}
\def\lambdaqie{\lambda_4}
\def\muqie{\mu_4}
\def\lambdaPhi{\lambda_5}

\newcommand{\length}{\textrm{length}}
\newcommand{\tb}{T_{\{0,1\}}}

\newcommand{\ol}{\overline}
\newcommand{\ola}{\overleftarrow}

\makeatletter
\newcommand{\myitem}[1]{%
\item[#1]\protected@edef\@currentlabel{#1}%
}
\makeatother

\DeclarePairedDelimiter\absval{\lvert}{\rvert}

\title{Embedding relatively hyperbolic groups into products of binary trees}

\date{}

\author{Patrick S. Nairne}

\begin{document}

\emergencystretch 3em

\raggedbottom

\setcounter{tocdepth}{1}

\begin{abstract}
We prove that if a group $G$ is relatively hyperbolic with respect to virtually abelian peripheral subgroups then $G$ quasiisometrically embeds into a product of binary trees. This extends the result of Buyalo, Dranishnikov and Schroeder in which they prove that a hyperbolic group quasiisometrically embeds into a product of binary trees. Inspired by Buyalo, Dranishnikov and Schroeder's \textit{Alice's Diary}, we develop a general theory of \textit{diaries} and \textit{linear statistics}. These notions provide a framework by which one can take a quasiisometric embedding of a metric space into a product of infinite-valence trees and upgrade it to a quasiisometric embedding into a product of \textit{binary} trees.
\end{abstract}

\maketitle

\section{Introduction}

\subsection{Results}

In this paper, we prove the following result.

\begin{theorem}[Main theorem] \label{thm.main}
Let $G$ be a finitely generated group that is relatively hyperbolic with respect to virtually abelian peripheral subgroups $H_1, H_2, \dots, H_T$. Let $R_t$ denote the rank of a finite index abelian subgroup of $H_t$, and define $R = \max(R_1, R_2, \dots, R_T)$. Then $G$ quasiisometrically embeds into a product of $R + \max(\asdim(G), R+1) + 1$ binary trees. Since $\asdim(G)$ is finite \cite{OSIN}, this is a finite product of binary trees.  
\end{theorem}

In particular, the fundamental group of a finite volume real hyperbolic manifold quasiisometrically embeds into a finite product of binary trees. 

\Cref{thm.main} should be compared with the result of Mackay and Sisto below \cite{MS}. 

\begin{theorem}[Mackay--Sisto] \label{thm.MackaySisto}
Let $G$ be a finitely generated group. Suppose $G$ is relatively hyperbolic with respect to peripheral subgroups $H_1, H_2, \dots, H_T$. Suppose also that each $H_t$ quasiisometrically embeds into a product of $m_t$ trees, and define $m = \max(m_1, m_2, \dots, m_T)$. Then $G$ quasiisometrically embeds into a product of $m + \max(\asdim(G), m+1) + 1$ trees. Since $\asdim(G)$ is finite \cite{OSIN}, this is a finite product of trees. 
\end{theorem} 

If a metric space $X$ quasiisometrically embeds into a finite product of trees then $X$ has finite asymptotic dimension. Hence the Mackay--Sisto result is itself a refinement of a result of Osin \cite{OSIN}: if a finitely generated group $G$ is relatively hyperbolic with respect to peripheral subgroups with finite asymptotic dimension, then $G$ has finite asymptotic dimension. Further, since peripheral subgroups are undistorted in a relatively hyperbolic group $G$ \cite{DS}, the reader should note that if $G$ quasiisometrically embeds into a finite product of trees (resp. binary trees) then the peripheral subgroups also quasiisometrically embed into a finite product of trees (resp. binary trees).

It is unclear to me whether the virtually abelian condition on the peripheral subgroups in \Cref{thm.main} can be weakened. The obvious conjecture is that $G$ quasiisometrically embeds into a product of binary trees if and only if the peripheral subgroups do as well. In this paper, the virtual nilpotence of the peripheral subgroups is used in order to quasiisometrically embed the Bowditch space $X(G)$ into a product of binary trees (see \Cref{cor.X(G)intobinarytrees}). 

\begin{question} \label{question.virtualnilpotence}
Can the virtually abelian condition be removed? If the peripherals quasiisometrically embed into products of binary trees, does $G$ do the same?
\end{question}

It follows from \cite{PAULS} that a virtually nilpotent but not virtually abelian finitely generated group cannot quasiisometrically embed into a product of trees. Hence, since peripheral subgroups are undistorted in a relatively hyperbolic group, if a single peripheral subgroup $H$ of a relatively hyperbolic group is virtually nilpotent but not virtually abelian, then $G$ cannot quasiisometrically embed into a product of trees. So, thanks to \cite{PAULS} and \Cref{thm.main}, the picture is fully understood when the peripheral subgroups are virtually nilpotent.

\begin{corollary}
Let $G$ be finitely generated and relatively hyperbolic with respect to virtually nilpotent peripheral subgroups. Then $G$ quasiisometrically embeds into a product of binary trees if and only if all the peripheral subgroups are virtually abelian. 
\end{corollary}

\Cref{thm.main} and \Cref{thm.MackaySisto} should themselves be compared with the two results below \cite{BL} \cite{BDS}.

\begin{theorem}[Buyalo--Lebedeva] \label{thm.BL}
A hyperbolic group $G$ admits a quasiisometric embedding into a product of $n+1$ trees where $n$ is the topological dimension of the boundary.
\end{theorem}

\begin{theorem}[Buyalo--Dranishnikov--Schroeder] \label{thm.BDS}
A hyperbolic group admits a quasiisometric embedding into a product of $n+1$ binary trees where $n$ is the topological dimension of the boundary.
\end{theorem}

Buyalo and Lebedeva \cite{BL} also prove that for a hyperbolic group $G$ we have $\asdim(G) = n+1$ where $n$ is the topological dimension of the boundary. As far as I'm aware, \Cref{thm.main}, \Cref{thm.BDS}, and the thesis of Alina Rull \cite{RULL}, in which Rull proves that an $n$-coloured right-angled Artin group or $n$-coloured right-angled Coxeter group admits a quasiisometric embedding into a product of $n$ binary trees, are the only results on the topic of quasiisometric embeddings into products of binary trees. 

\begin{remark}
The papers \cite{BDS} and \cite{RULL} both make essential use of the \textit{Morse--Thue sequence}. In my proof of \Cref{thm.main}, the use of this sequence is avoided (it is effectively replaced by \Cref{lem.nomt}).
\end{remark}

\subsection{Ideas}

There are two notable ideas that are developed in order to complete the proof of \Cref{thm.main}.

\subsubsection{The tree of metric spaces}

We use the theory of projection complexes developed by Bestvina, Bromberg, Fujiwara and Sisto, and in particular the notion of a \textit{tree of metric spaces} that is outlined in \cite[Section 4]{PART1}. I refer the reader to \cite{PART1} for the necessary background on projection complexes and the tree of metric spaces. 

\subsubsection{Diaries and statistics}

The second new idea is to use the theory of diaries and statistics described in \Cref{chap.diaries}. This work is inspired by \textit{Alice's Diary}, a function defined by Buyalo, Dranishnikov and Schroeder \cite{BDS}. In this paper, a diary is simply a height-preserving graph homomorphism from an infinite-valence rooted tree to a uniformly bounded rooted valence tree. It is outlined in \Cref{ex.diariesms} how diaries can be used to upgrade quasiisometric embeddings into products of infinite-valence trees so that they become quasiisometric embeddings into products of \textit{binary} trees. The techniques and theorems described in \Cref{chap.diaries} provide a framework for constructing quasiisometric embeddings into products of binary trees in many other contexts.   

\subsection{Questions}

It could be the case that a group quasiisometrically embeds into a finite products of trees if and only if it does so into a finite product of binary trees.

\begin{question} \label{question.embedding2}
Does there exist a finitely generated group which quasiisometrically embeds into a finite product of trees but which does not quasiisometrically embed into a finite product of binary trees?
\end{question}

One might hope to use the methods of \Cref{chap.diaries}, the proof structure of \Cref{thm.main}, and the quasiisometric embedding of the mapping class group into a product of quasi-trees of pseudo-Anosov axes \cite{BBF2} in order to prove that a mapping class group quasiisometrically embeds into a product of \textit{binary} trees, as opposed to the infinite-valence trees described in \cite{HUME} and \cite{BBF2}. One might even hope to use this to imitate \cite{PETYT} and prove that the mapping class group is quasiisometric to a uniformly locally finite CAT(0) cube complex. However, we have no equivalent of \Cref{cor.relhypreg} for the mapping class group. This regular map of a relatively hyperbolic group $G$ into a product of binary trees is used in an essential way in the proof of \Cref{thm.main}. So we ask the following question.
\begin{question}
Does there exist a regular map of the mapping class group (of, say, some compact hyperbolic surface) into a finite product of binary trees?
\end{question}
Thank you to Harry Petyt for engaging discussions on the possibility of quasiisometrically embedding mapping class groups into products of binary trees.

\subsection{Outline of the proof of the main theorem}

The reader might want to refer back to this proof outline as they read the paper if they want to understand how a section or result fits into the overall scheme of the proof of \Cref{thm.main}. 

Let $G$ be relatively hyperbolic with respect to virtually abelian peripheral subgroups $H_1, \dots H_T$. Let $R$ be the maximal rank of finite index abelian subgroups of the peripheral subgroups. Let $S$ be a finite generating set of $G$. Let $A$ be the finite alphabet $S \cup S^{-1}$, let $W$ be the set of finite words on $A$ and let $T_W$ be the sentence-tree (see \Cref{def.sentencetree}) whose vertices are in bijection with the free monoid $W^*$. Using work of Dahmani and Yaman \cite{DY} and work of Buyalo, Dranishnikov and Schroeder \cite{BDS}, we can prove there exists a regular map $\phi: G \rightarrow \prod_{q=1}^Q \tb$ where $Q \leq \max(\asdim(G),R+1) + 1$ (see \Cref{def.regular} and \Cref{cor.relhypreg}) and where $\tb$ denotes the rooted binary tree. 

\begin{enumerate}
    \item[(Step 1)] For each $1 \leq r \leq R$ and $x,z \in G$, we define a canonical standard path (which is a quasigeodesic) from $x$ to $z$ called the \textit{$r$'th ordered standard path}. This is \Cref{def.sp}.
    \item[(Step 2)] By partitioning the $r$'th ordered standard path from $e \in G$ to $x \in G$ (which corresponds to an element of $W$) into several words, we may define a map $F_r: G \rightarrow T_W$. 
    \item[(Step 3)] We prove that $F_1 \times \dots \times F_R \times \phi: G \rightarrow \prod_{r = 1}^R T_W \times \prod_{q=1}^Q \tb$ is a quasiisometric embedding.
    \item[(Step 4)] We choose several finite statistics and linear statistics, and constants $\jf$, $\jl$, $N$ and $\epsilon$ so that \Cref{cor.leotaurus} induces a diary $D: T_W \rightarrow T_\Omega$ which coarsely preserves the metric on a pair of sentences in $T_W$ whenever they satisfy $\leo(\scalf, \jf)$ or $\taurus(\scall, \jl, N, \epsilon)$. 
    \item[(Step 5)] We prove that 
    \[G \xrightarrow{F_1 \times \dots \times F_R \times \phi} \prod_{r = 1}^R T_W \times \prod_{q=1}^Q \tb \xrightarrow{D \times \dots D \times \id} \prod_{r = 1}^R T_\Omega \times \prod_{q=1}^Q \tb\]
    is a quasiisometric embedding. We follow the proof structure described in \Cref{ex.diariesms}; in other words, we reduce the problem to proving that a pair of sentences $\alpha$, $\beta \in T_W$ satisfy either the property $\leo(\scalf, \jf)$ or the property $\taurus(\scall, \jl, N, \epsilon)$.  
\end{enumerate}

Since the uniformly bounded valence word-tree $T_\Omega$ is quasiisometric to $\tb$, we are done. 

\subsection{Acknowledgements}

Thank you to Alessandro Sisto for alerting me to the result of \cite{PAULS} and for help with structuring this work. Thank you to David Hume for a useful and enlightening discussion. Finally, thank you to my supervisor Cornelia Dru\textcommabelow{t}u for suggesting that I consider relatively hyperbolic groups in this work.

\tableofcontents

\section{Definitions and conventions} \label{chap.defs}

\subsection{Numbers}

\begin{notation}
We use $\N_0$ to denote the non-negative integers and $\N$ to denote the positive integers. Notation such as $\R_{> 0}$ means the set $\{x \in \R : x > 0\}$.
\end{notation}

\subsection{Coarse geometry}

Throughout this subsection, let $X$ and $Y$ be a pair of metric spaces and let $\lambda \geq 1$, $\mu \geq 0$ be a pair of constants. 

\begin{definition}
We say that $f: X \rightarrow Y$ is \textit{$(\lambda,\mu)$-coarsely Lipschitz} if
\[d(f(x),f(x')) \leq \lambda d(x,x') + \mu\]
A $(\lambda,0)$-coarsely Lipschitz function $f:X \rightarrow Y$ is \textit{$\lambda$-Lipschitz}. 
\end{definition}

\begin{definition} \label{def.embeddings}
A \textit{$(\lambda,\mu)$-quasiisometric embedding} $f: X \rightarrow Y$ is a function which satisfies
\[\frac{1}{\lambda} d(x,x') - \mu \leq d(f(x),f(x')) \leq \lambda d(x,x') + \mu\]
for all $x,x' \in X$. It is a \textit{$(\lambda,\mu)$-quasiisometry} if in addition, for all $y \in Y$ there exists $x \in X$ such that $d(y,f(x)) \leq \mu$. A $(\lambda,0)$-quasiisometric embedding $f: X \rightarrow Y$ is a \textit{$\lambda$-bilipschitz embedding}. A $(\lambda,0)$-quasiisometry $f: X \rightarrow Y$ is a \textit{$\lambda$-bilipschitz homeomorphism}.
\end{definition}

\begin{definition}
A metric space $X$ has \textit{bounded growth at some scale} if there are constants $0 < r < R$ and $N \in \N$ such that all open balls of radius $R$ can be covered by $N$ open balls of radius $r$. 
\end{definition}

The following proposition will be used, briefly, in \Cref{sec.relhyp}.

\begin{proposition} \label{prop.bgoss}
If a metric space $X$ admits a quasiisometric embedding into a uniformly bounded degree connected graph $\Gamma$ then $X$ has bounded growth on some scale.
\end{proposition}

\begin{proof}
Suppose we have some $(\lambda,\mu)$-quasiisometric embedding $\phi: X \rightarrow \Gamma$. Suppose that $D$ is a uniform bound on the degree of vertices of $\Gamma$. We may assume that $\phi$ takes elements of $X$ to vertices of $\Gamma$. Consider an open ball $B$ of radius $R$ in $X$. Then $\phi(B)$ is contained with an open ball $B'$ of radius $\lambda R + \mu$ in $\Gamma$. $B'$ contains at most $D^{\lambda R + \mu}$ vertices. For each vertex $v \in B'$ with $\phi^{-1}(v) \cap B$ non-empty, choose an element $x_v \in B$ with $\phi(x_v) = v$. If we consider balls of radius $\mu$ around all of the $x_v$, we can see that we have covered $B$ by $D^{\lambda R + \mu}$ balls of radius $\mu$. 
\end{proof}

\begin{convention}
Let $X$ be a metric space. Sometimes, but not always, we will denote the metric on $X$ by $d_X$, with a subscript to indicate the metric space. Other times, when the metric space is clear, we will just write $d$.
\end{convention}

\subsection{Groups}

Let $G$ be a group and let $S$ be a finite generating set. We use the notation $\Gamma(G,S)$ to indicate the \textit{Cayley graph} of $G$ with respect to $S$, which is the graph with vertex set $G$ and an edge $\{g,gs\}$ for all $g \in G$ and $s \in S$. 

\begin{notation} \label{notation.cayleysubgraph}
Let $\Gamma(G,S)$ be a Cayley graph and suppose a subgroup $H \leq G$ is generated by some finite set $S_H \subset S$. We may think of the Cayley graph $\Gamma(H,S_H)$ as being a connected subgraph of $\Gamma(G,S)$. Given $g \in G$, we use the notation 
\[g\Gamma(H,S_H)\]
to refer to the image of the subgraph $\Gamma(H,S_H)$ under the isometric action of $g$. So $g\Gamma(H,S_H)$ has vertex set $gH$ and the edges correspond to the generators $S_H$. 
\end{notation}

\subsection{Trees}

\begin{definition} \label{def.hpop}
Let $T$ be a simplicial rooted tree with root vertex $v_0$.
\begin{itemize}
    \item Say that two vertices of $T$ are \textit{adjacent} if they are distinct and connected by an edge.
    \item The \textit{children} of a vertex $v \in T$ are precisely the vertices $v' \in T$ which are adjacent to $v$ and for which $d_T(v_0,v') > d_T(v_0,v)$.
    \item We say that a vertex $v' \in T$ \textit{descends} from a vertex $v \in T$ if the unique geodesic from $v_0$ to $v'$ passes through $v$. If $v'$ descends from $v$ then $v$ is an \textit{ancestor} of $v'$. 
    \item Given two vertices $u, v \in T$, the \textit{lowest common ancestor} of $u$ and $v$, denoted $\lca(u,v)$, is the common ancestor of $u$ and $v$ with the greatest distance from the basepoint $v_0$. 
    \item We say that $T$ has \textit{uniformly bounded valence} or \textit{uniformly bounded degree} or that it is \textit{uniformly locally finite} if there is a constant $M \in \N$ such that every vertex of $T$ has degree at most $M$. 
\end{itemize}
Now suppose $T$ and $T'$ are simplicial rooted trees with root vertices $v_0$ and $v_0'$ respectively.
\begin{itemize}
    \item We say that a map $f: T \rightarrow T'$ is \textit{height-preserving} if for all vertices $v \in T$ we have $d(v_0,v) = d(v_0', f(v))$.
    \item We say that $f: T \rightarrow T'$ is \textit{order-preserving} if the following implication holds: if $v'$ descends from $v$ in $T$ then $f(v')$ descends from $f(v)$ in $T'$. 
\end{itemize}
\end{definition}

\subsection{Words and sentences}

\subsubsection{Words}

\begin{definition} \label{def.alphabet}
An alphabet is just a set of \textit{letters}. We will often denote alphabets by $A$, $B$ or $C$ and we will often denote letters by $a$ and $b$. We say that the alphabet is finite if there are finitely many letters.
\end{definition}

\begin{definition} \label{def.word}
A \textit{word} is an empty, finite, or infinite string of letters from an alphabet $A$. The \textit{length} of a word $u$ is the number of letters in it. We denote this by $\length(u)$.
\end{definition}

For example, if $A = \{a,b\}$ then $u = abab$ is a word of length $4$ on the alphabet $A$. We will often denote by $W$ the set of all finite words on an alphabet $A$. We will often denote words themselves by the letters $u$, $v$ or $w$. To indicate the concatenation of a pair of words into a longer word, we simply write the two words next to each other. For example, if $u = abab$ and $v = aaa$ then
\[uv = ababaaa\] 

\begin{notation} \label{not.ola}
We write $\ola{w}$ to denote the word $w$ written backwards. So $\ola{abcd} = dcba$ and so on.
\end{notation}

Further, we sometimes need to refer to the position of a letter $a$ within a word $u$. To do this, we might refer to the \textit{distance} of $a$ from the start or end of $u$. If we say $a$ has distance $d$ from the start of $u$ that means $a$ is the $d$'th letter in the word $u$. If we say that $a$ has distance $d$ from the end of $u$ that means $a$ is the $(\length(u) - d + 1)$'th letter of $u$, or, equivalently, the $d$'th letter of $\ola{u}$. For example, in the word $u = abaca$, the letter $c$ has distance $2$ from the end of $u$ and distance $4$ from the start of $u$. 

\subsubsection{Sentences}

\begin{definition} \label{def.sentence}
A \textit{sentence} is a string of words. We will often denote sentences by the Greek letters $\alpha$ and $\beta$. In order to describe the words which make up a sentence, we will use \textit{overlines} to indicate where one word ends and another begins. For example, if $A = \{a,b\}$ then 
\[\alpha = \ol{abab} \ol{aaa} \ol{ba}\]
is a sentence that consists of the three words $abab$, $aaa$ and $ba$. If $u = abab$, $v = aaa$ and $w = ba$ we can also write
\[\alpha = \ol{u} \ \ol{v} \ \ol{w}\]
The purpose of the overlines is so that we can distinguish between the single word $uvw$ and the three word sentence $\ol{u} \ \ol{v} \ \ol{w}$. The \textit{length} of a sentence is the number of words in it. We denote this by $\length(\alpha)$.
\end{definition}

Sentences are, of course, just words on an alphabet of words. To indicate the concatenation of a pair of sentences into a longer sentence, we simply write the two sentences next to each other. For example, if $\alpha = \ol{abab} \ol{aaa} \ol{ba}$ and $\beta = \ol{ab} \ol{a}$ then
\[\alpha \beta  = \ol{abab} \ \ol{aaa} \ \ol{ba} \ \ol{ab} \ \ol{a}\]

\begin{notation}
Given a sentence $\alpha$ and $m \in \N$, we can write $\alpha^m$ to indicate the sentence $\alpha$ written $m$ times. For example, $(\ol{a} \ol{b})^3 = \ol{a} \ol{b} \ol{a} \ol{b} \ol{a} \ol{b}$.
\end{notation}

The following notation will be convenient in \Cref{sec.proofpart1}.

\begin{notation} \label{not.sentencenegative}
Suppose we have a group $G$ generated by a finite set $S$ and suppose $A = S \cup S^{-1}$ is our alphabet. If $q \in \Z$ and $q < 0$, and if $s \in S$, then $(\ol{s})^q$ indicates the sentence $(\ol{s^{-1}})^{-q}$. For example, $(\ol{s})^{-3} = \ol{s^{-1}} \ \ol{s^{-1}} \ \ol{s^{-1}}$.
\end{notation}

\subsubsection{Word-trees and sentence-trees}

\begin{definition} \label{def.wordtree}
Given an alphabet $A$, we will use the notation $T_A$ to indicate \textit{the word-tree} on $A$. The tree $T_A$ is defined precisely so that its vertices are in bijection with finite words on the alphabet $A$. Every edge of $T_A$ will be labelled by an element of $A$. The word-tree $T_A$ is defined as follows. As a set, $T_A$ is precisely the rooted tree where every vertex has $\absval{A}$ children. Further, given a vertex $v \in T_A$, the $\absval{A}$ edges which go between $v$ and its children are labelled by the elements of $A$, with every such edge having a label in $A$, and no two such edges having the same label. We will now describe the natural bijection between the vertices of $T_A$ and finite words on the alphabet $A$. First, we associate the empty word $\emptyset$ to the root vertex of $T_A$. Now let $v$ be a vertex in $T_A$. The geodesic from $\emptyset$ to $v$ in $T_A$ travels along a series of labelled edges. These labels, read in order, make up the word that we associate to $v$. Thus every vertex in $T_A$ has a corresponding word on $A$. \textit{Indeed, we will often identify them.} The vertices at distance $n$ from the root of $T_A$ correspond to words of length $n$ on the alphabet $A$. 
\end{definition}

The following notation for the rooted binary tree will be used frequently in this paper. 

\begin{example} \label{ex.binarytree}
If $A = \{0,1\}$ then $T_{A} = \tb$ is a rooted binary tree whose vertices correspond to finite strings of zeros and ones.
\end{example}

\begin{definition} \label{def.sentencetree}
Let $A$ be a finite alphabet and let $W$ be the set of non-empty finite words on $A$. Analogous to the word-tree on $A$, we have a \textit{sentence-tree} on $A$ which is exactly the tree $T_W$, i.e. the word-tree on $W$. So the vertices of $T_W$ are in bijection with sentences on $A$, and the edges of $T_W$ are labelled by finite words on $A$. The vertices at distance $n$ from the root vertex of $T_W$ correspond to sentences with $n$ words.
\end{definition}

Note that if $A$ is a finite alphabet then $T_A$ is a uniformly bounded valence tree, whereas $T_W$ is a tree with countable valence at every vertex.

The following proposition will be used in the proof of \Cref{thm.main}. 

\begin{lemma} \label{lem.nomt}
Let $A$ be a finite alphabet and let $T_A$ be the word-tree on $A$ (whose vertices are in bijection with the free monoid $A^*$). Suppose $w,w' \in A^*$ are distinct words on $A$ that satisfy $d_{T_A}(w,w') \leq k$. Then one of the two following conditions hold:
\begin{enumerate}
    \item the final $k$ letters of $w$ and $w'$ are distinct;
    \item the final $k$ letters of the base $10$ expansion of $\length(w)$ and the final $k'$ letters of the base $10$ expansion of $\length(w')$ are distinct. 
\end{enumerate}
\end{lemma}

\begin{proof}
Suppose first that $\length(w) \neq \length(w')$. If the final $k$ letters of the base $10$ expansion of $\length(w)$ and the final $k'$ letters of the base $10$ expansion of $\length(w')$ are the same, then one can easily show that $\absval{\length(w) - \length(w')} \geq 10^k$. But this would imply that $d_{T_A}(w,w') \geq 10^k$ which is a contradiction. So we may assume that $\length(w) = \length(w')$. Since $d_{T_A}(w,w') \leq k$, it follows that there exists letters $a \in w$ and $a' \in w'$ which are distinct and at the same distance $d \leq k/2$ from the end of their words. It follows that the final $k$ letters of $w$ and $w'$ are distinct.
\end{proof}

Buyalo, Dranishnikov and Schroeder \cite{BDS} use the 
\textit{Morse--Thue sequence} in their proof for similar purposes as the proposition above.

\subsubsection{Notation for pairs of sentences}

Let $T_W$ be a sentence-tree and suppose we have a pair of sentences $\alpha$, $\beta \in T_W$. In this paper, this situation will occur so frequently that it is useful to have some standard notation for the words that make up these sentences. 

\begin{notation} \label{notation.sentence}
Given sentences $\alpha, \beta \in T_W$, we can always write
\begin{equation} \label{eq.alpha}
\alpha = \ol{u_1} \ \ol{u_2} \dots \ol{u_p} \ \ol{u_{p+1}} \dots \ol{u_{p+m}}
\end{equation}
and 
\begin{equation} \label{eq.beta}
\beta = \ol{u_1} \ \ol{u_2} \dots \ol{u_p} \ \ol{u_{p+1}'} \dots \ol{u_{p+n}'}
\end{equation}
where $u_{p+1} \neq u_{p+1}'$ and $m,n \geq 0$. In \eqref{eq.alpha}, $u_1$, $u_2$, \dots, $u_p$, $u_{p+1}$, \dots, $u_{p+m-1}$, $u_{p+m}$ are the words that make up the sentence $\alpha$. In \eqref{eq.beta}, $u_1$, $u_2$, \dots, $u_p$, $u_{p+1}'$, \dots, $u_{p+n-1}'$, $u_{p+n}'$ are the words that make up the sentence $\beta$. Therefore, the sentence $\ol{u_1} \ \ol{u_2} \dots \ol{u_p}$ is the lowest common ancestor of the sentences $\alpha$ and $\beta$ in the tree $T_W$ and $d_{T_W}(\alpha,\beta) = m+n$. 
\end{notation}

\subsection{Regular maps} \label{subsec.reg}

The notion of a \textit{regular map} notably appears in the work of Benjamini, Schramm and Tim\'{a}r \cite{BST} on separation profiles since separation profiles are monotone under regular maps. 

\begin{definition} \label{def.regular}
Let $\phi: V(\Gamma) \rightarrow V(\Gamma')$ be a map between the vertex sets of connected bounded degree graphs $\Gamma$, $\Gamma'$. We say that $\phi$ is \textit{regular} if $\phi$ is Lipschitz and if there is a uniform bound on the cardinality of $\phi^{-1}(v)$.
\end{definition}

Two important examples of regular maps are coarse embeddings and subgroup inclusion. It is possible to define a regular map between arbitrary metric spaces, but we will not do so here. In the lemma below, the maps $\phi$ and $\Phi$ should be interpreted as having vertex sets as domain and codomain, although this is dropped from the notation. 

\begin{lemma} \label{lem.reglem}
Suppose $\phi: \Gamma \rightarrow \prod_{q=1}^Q \tb$ is a regular map of a bounded degree graph (with base vertex $v_0 \in \Gamma$) into the product of $Q$ binary trees. Then there exists a finite alphabet $C$, with associated word-tree $T_C$, and a Lipschitz embedding $\Phi: \Gamma \rightarrow \prod_{q = 1}^{Q} T_C$ such that 
\begin{itemize}
    \item $\Phi(v)$ is a vertex for all vertices $v \in \Gamma$;
    \item $\Phi$ is injective on the vertices of $\Gamma$;
    \item the basepoint $v_0 \in \Gamma$ is mapped by $\Phi$ to the product of the roots $\mathbf{\emptyset} = (\emptyset,\emptyset, \dots, \emptyset)$.
\end{itemize}
\end{lemma}

\begin{proof}
Let $m \in \N$ be a uniform upper bound on the amount of vertices in $\phi^{-1}(v)$. Let $n \in \N$ and let $T$ denote the $n$-ary rooted tree. That is, $T$ is the rooted tree such that every vertex has $n$ children. I think it is clear that if $n \in \N$ is much larger than $m$ then by adjusting $\phi$ by a uniformly bounded amount, we can define an \textit{injective} and coarsely Lipschitz map $\Phi: \Gamma \rightarrow \prod_{q=1}^Q T$ which takes $v_0 \in \Gamma$ to the basepoint of $\prod_{q=1}^Q T$. Further, a coarsely Lipschitz map whose domain is the vertex set of a graph is in fact Lipschitz. So $\Phi$ is Lipschitz on the vertices of $\Gamma$.

Finally, by setting $C = \{1,2,3, \dots, n\}$, we can identify $T$ with the word-tree $T_C$. This completes the proof of the lemma. 
\end{proof}

\section{Background on relatively hyperbolic groups} \label{sec.relhyp}

\subsection{Relatively hyperbolic groups}

Let $G$ be a group with finite generating set $S$ and let $\Gamma(G,S)$ denote the Cayley graph of $G$ with respect to $S$. When we speak of the metric on $G$, we are actually referring to the metric induced by $\Gamma(G,S)$, but we will generally denote it by $d_G$ nonetheless. Suppose that $\H$ is some collection of subgroups of $G$. Elements of $\H$ are called \textit{peripheral subgroups}. Let $\G$ be the collection of all cosets of peripheral subgroups. In other words, $\G = \{gH : H \in \H\}$. Elements of $\G$ are called \textit{peripheral cosets}. Let us suppose as well that for each $H \in \H$, $S \cap H$ generates the subgroup $H$. 

\begin{definition} \label{def.conedoff}
The \textit{coned-off graph} of $G$ with respect to $\H$ is the graph formed by taking the Cayley graph $\Gamma(G,S)$ and then adding an edge between $g$ and $g'$ (if it's not already there) whenever $g$ and $g'$ are contained in the same peripheral coset. Although the coned-off graph depends on $\H$, this is omitted from the notation and the coned-off graph is simply denoted by $\hat{G}$.
\end{definition}

The coned-off graph is (coarsely) what you get if you crush every peripheral coset to a point. We will now define the \textit{Bowditch space} $X(G)$, a metric space closely related to $\hat{G}$.

\begin{definition} \label{def.horoball}
Suppose that $\Gamma$ is a connected graph with vertex set $V_\Gamma$ and edge set $E_\Gamma$. Suppose also that every edge has length $1$. The \textit{horoball} $\mathcal{H}(\Gamma)$ is the graph with vertex set 
\[V_{\mathcal{H}(\Gamma)} = V_{\Gamma} \times \N_0\]
and edge set
\[E_{\mathcal{H}(\Gamma)} = \{\{(v,n), (v',n)\} : \{v,v'\} \in E_{\Gamma}, n \in \N_0\} \sqcup \{\{(v,n),(v,n+1)\} : v \in V_{\Gamma}, n \in \N_0\} \]
We insist that edges of the form $\{(v,n), (v',n)\}$ have length $e^{-n}$ and edges of the form $\{(v,n),(v,n+1)\}$ have length $1$. 
\end{definition}

So the horoball $\mathcal{H}(\Gamma)$ consists of countably many copies of $\Gamma$ stacked on top of each other, getting smaller and smaller as one travels higher. 

\begin{definition} \label{def.bowditchspace}
The \textit{Bowditch space} of $G$ with respect to $\H$ is formed by taking $\Gamma(G,S)$ and then, for every $gH \in \G$, attaching a copy of $\mathcal{H}(\Gamma(H,S \cap H))$ to the subgraph $g\Gamma(H,S \cap H)$. The Bowditch space is denoted by $X(G)$. 
\end{definition}

The definition we use below is equivalent to all the other definitions of relative hyperbolicity. It is due to Groves and Manning \cite{GM}.

\begin{definition} \label{def.relhyp}
We say that $G$ is \textit{relatively hyperbolic} with respect to $\H$ if $X(G)$ is hyperbolic.
\end{definition}

\begin{remark}
Due to \cite[Theorem 1.1]{OSIN2}, if $G$ is relatively hyperbolic we may always assume that $\H$ is finite and that every $H \in \H$ is finitely generated.
\end{remark}

Strangely, I cannot find it recorded in the literature that \cite[Theorem 1.2]{BDS} is really an \textit{if and only if}. I'll record it here, since we have it almost for free.

\begin{corollary}[Buyalo--Dranishnikov--Schroeder, Bonk--Schramm] \label{cor.bdsbs}
If $X$ is a visual hyperbolic metric space then $X$ admits a quasiisometric embedding into a finite product of binary trees if and only if the boundary of $X$ is a doubling metric space. Further, if the boundary of $X$ is a doubling metric space, then $X$ admits a quasiisometric embedding into a finite product of $n+1$ binary trees, where $n$ is the linearly controlled metric dimension of the boundary. 
\end{corollary}

\begin{proof}
The backwards implication of the if and only if, and the \textit{Further} part of the corollary, are exactly \cite[Theorem 1.2]{BDS}. 

So suppose $X$ is a visual hyperbolic metric space that admits a quasiisometric embedding into a finite product of binary trees. Then obviously $X$ admits a quasiisometric embedding into a uniformly bounded degree graph. So \Cref{prop.bgoss} tells us that $X$ has bounded growth on some scale. Hence, \cite[Theorem 9.2]{BS} tells us that the boundary of $X$ is doubling.
\end{proof}

\begin{corollary}[Buyalo--Dranishnikov--Schroeder, Dahmani--Yaman, Bonk--Schramm] \label{cor.X(G)intobinarytrees}
Suppose $H$ is relatively hyperbolic with respect to $\H$. Then $X(G)$ quasiisometrically embeds into a finite product of binary trees if and only if all the subgroups in $\H$ are virtually nilpotent. Further, if all the subgroups in $\H$ are virtually nilpotent then $X(G)$ quasiisometrically embeds into a product of $n+1$ binary trees where $n$ is the linearly controlled metric dimension of the boundary of $X(G)$. 
\end{corollary}

\begin{proof}
Mackay and Sisto \cite[Proposition 4.5]{MS2} use the main result of Dahmani and Yaman's paper \cite{DY} to prove that all the peripheral subgroups are virtually nilpotent if and only if the boundary of $X(G)$ is doubling. It then follows from \Cref{cor.bdsbs} that all the peripheral subgroups are virtually nilpotent if and only if $X(G)$ quasiisometrically embeds into a finite product of binary trees. 

The \textit{Further} part of the corollary follows from \Cref{cor.bdsbs}. 
\end{proof}

Recall from \Cref{ex.binarytree} that $\tb$ denotes the rooted binary tree. 

\begin{corollary} \label{cor.VAbinary}
Suppose $\H = \{H_1, H_2, \dots, H_T\}$, suppose each peripheral subgroup $H_t$ is virtually abelian, and define $R = \max_t R_t$ where $R_t$ is the rank of a finite index abelian subgroup of $H_t$. Then there exists a quasiisometric embedding $\phi': X(G) \rightarrow \prod_{q=1}^Q \tb$ where $Q$ is at most $\max(\asdim(G), R+1) + 1$. We may assume that the image of $\phi'$ is contained in the vertex set of $\prod_{q=1}^Q \tb$. 
\end{corollary}

\begin{proof}
\Cref{cor.X(G)intobinarytrees} implies that $X(G)$ quasiisometrically embeds into a product of $n+1$ binary trees where $n$ is the linearly controlled metric dimension of the boundary of $X(G)$. Proposition 3.6 and Proposition 3.4 of \cite{MS} combine to prove that $n \leq \max(\asdim(G),R+1)$ as desired.
\end{proof}

\begin{corollary} \label{cor.relhypreg}
Let $G$ be as in \Cref{cor.VAbinary}. Then there exists a regular map $\phi: G \rightarrow \prod_{q=1}^Q \tb$ where $Q \leq \max(\asdim(G),R+1) + 1$. 
\end{corollary}

\begin{proof}
Let $\phi$ be the composition $G \hookrightarrow X(G) \rightarrow \prod_{q=1}^Q \tb$ where the second map is $\phi'$ from \Cref{cor.VAbinary}. By the assumption on $\phi'$, the image of $\phi$ is in the vertex set of $\prod_{q=1}^Q \tb$. 

Since $G \rightarrow X(G)$ is Lipschitz and $\phi': X(G) \rightarrow \prod_{q=1}^Q \tb$ is coarsely Lipschitz we conclude that $\phi$ is Lipschitz. 

Let $v$ be a vertex of $\prod_{q=1}^Q \tb$. Since $\phi': X(G) \rightarrow \prod_{q=1}^Q \tb$ is coarsely Lipschitz, the preimage of $v$ in $X(G)$ is contained in a ball of uniformly bounded radius in $X(G)$. But a ball of uniformly bounded radius in $X(G)$ can contain only a uniformly bounded amount of elements of $G \subset X(G)$. So $\phi^{-1}(v) \subset G$ has a uniformly bounded cardinality. 
\end{proof}

We now state some results on the metric structure of relatively hyperbolic groups. These are \cite[Lemma 4.15]{DS} and \cite[Lemma 1.13, Proposition 1.14 and Lemma 1.15]{SISTO}. 

\begin{proposition}[Dru\textcommabelow{t}u--Sapir] \label{prop.ds}
Peripheral cosets are undistorted in $G$. That is, the natural map $gH \subseteq G \rightarrow g\Gamma(H,S \cap H)$ is a quasiisometry. 
\end{proposition}

Given $gH \in \G$ and $x \in G$, we define $\pi_{gH}(x) \subseteq gH$ to be the nearest point projection of $x$ onto $gH \subseteq \Gamma(G,S)$, i.e.
\[\pi_{gH}(x) = \{y \in gH : d_G(x,y) = \min_{z \in gH} d_G(x,z)\}\]

\begin{lemma}[Sisto] \label{lem.sisto1}
There exists $M$ such that if $\hat{\gamma}$ is a geodesic in the coned-off graph $\hat{G}$ connecting $x \in G$ to $y \in gH$ then the first point in $\hat{\gamma} \cap gH$ is at distance (taken with respect to the metric on $\Gamma(G,S)$) at most $M$ from $\pi_{gH}(x)$. 
\end{lemma}

The \textit{lift} of a geodesic $\hat{\gamma}$ in $\hat{G}$ is a path $\gamma$ in $G$ obtained by substituting edges of $\hat{\gamma}$ that go between $g, g' \in kH$ with a geodesic in the subgraph $k\Gamma(H, S \cap H)$ from $g$ to $g'$. 

\begin{proposition}[Sisto] \label{prop.sisto}
There exist $\lambda \geq 1$, $\mu \geq 0$ so that if $\hat{\gamma}$ is a geodesic in $\hat{G}$ then its lifts are $(\lambda,\mu)$-quasigeodesics. 
\end{proposition}

Let $\bar{B}_r(x)$ denote the closed $r$-ball around some point $x \in G$.

\begin{lemma}[Sisto] \label{lem.sisto2}
There exists $L$ so that if $d_G(\pi_{gH}(x),\pi_{gH}(z)) \geq L$ for some ${gH} \in \G$ and $x,z \in G$ then 
\begin{enumerate}
    \item all $(\lambda,\mu)$-quasigeodesics in $\Gamma(G,S)$ connecting $x$ to $z$ intersect $\bar{B}_R(\pi_{gH}(x))$ and $\bar{B}_R(\pi_{gH}(z))$ where $R = R(\lambda,\mu)$;
    \item all geodesics in $\hat{G}$ connecting $x$ to $z$ contain an edge between two elements of $gH$.
\end{enumerate}
\end{lemma}

\subsection{Relatively hyperbolic groups as projection complexes} \label{subsec.relhyppc}

For the necessary background on projection complexes, consult \cite[Section 2]{PART1} or \cite{BBFS}. 

As before, let $G$ be finitely generated with respect to a finite generating set $S$ and relatively hyperbolic with respect to a set of peripheral subgroups $\H$ (where $S \cap H$ generates each $H \in \H$) and let $\G$ be the set of peripheral cosets. 

The set of peripheral cosets $\G$ will be our indexing set, taking the place of $\Y$ in the general projection complex theory. I have chosen to use the notation $\G$ in this context so that there is a clear demarcation between the general theory of projection complexes given in \cite{PART1} and the specific case of relatively hyperbolic groups given here. However, for ease of notation, we will frequently refer to cosets by the notation $X,Y,Z,W$ as in \cite{PART1} or \cite{BBFS}. 

We associate to each $gH \in \G$ the metric space $\C(gH) = g\Gamma(H,S \cap H)$. Given $Y = gH \in \G$ and $x \in \C(X)$, we define, as in the previous subsection, $\pi_Y(x) \subseteq \C(Y)$ to be the nearest point projection of $x$ onto $gH \subseteq \Gamma(G,S)$. 
Then we have the projection data given by
\[\pi_Y(X) = \{\pi_Y(x) : x \in \C(X)\} \]
There is a slight subtlety in the above definitions; the nearest-point projections $\pi_Y$ are determined by the metric in $\Gamma(G,S)$, but we think of them as being subsets of $\C(Y) = g\Gamma(H,S \cap H)$.

Mackay and Sisto \cite{MS} notice that, by applying the work of Sisto \cite{SISTO}, the following result holds.

\begin{theorem}[Mackay--Sisto] \label{thm.p3-p5}
The metric spaces $\C(Y) = g\Gamma(H,S \cap H)$ together with the projection data
\[ \pi_Y: \G \setminus Y \rightarrow \textrm{non-empty subsets of } \C(Y)\]
satisfy axioms (P3) - (P5) for some projection constant $\xi$. 
\end{theorem}

We can now use the general theory of projection complexes to upgrade our projections $\{\pi_Y\}$ so that they now satisfy the \textit{strong} projection axioms with respect to a constant $\theta$, and let us define $\pcg$, $\qtomsg$ and $\tomsg$ for some $K \geq 4\theta$ accordingly.

Suppose $\H = \{H_1, H_2, \dots, H_T\}$. A group element $g \in G$ has $T$ possible images in the quasi-tree of metric spaces $\qtomsg$. Therefore, to properly specify a vertex of $\qtomsg$ we need specify the group element \textit{and} the coset. We write $g \in gH_t$ to mean the copy of $g$ that sits within the coset $gH_t \subset \qtomsg$. Let's define $\iota_t: G \rightarrow \qtomsg$ by $\iota_t(g) = g \in gH_t$. It is not difficult to prove that the embeddings $\iota_t: G \rightarrow \qtomsg$ are the same up to bounded distance; for lack of a better choice, we will always consider the embedding $\iota_1: G \rightarrow \qtomsg$.

\begin{notation} \label{notation.images}
A group element $g \in G$ has a natural image in $\hat{G}$ and $X(G)$. We will denote both of these images in $\hat{G}$ and $X(G)$ by $g$. 
\end{notation}

Sisto \cite{SISTO} proves that the distance in a relatively hyperbolic group can be understood in terms of projections onto peripheral cosets and the distance in the coned-off graph. In order to state Sisto's theorem, we need some notation. 

\begin{notation} \label{not.approx}
Let $A,B,L \in \R$, let $\lambda \geq 1$ and let $\mu \geq 0$. 
\begin{itemize}
    \item We write $A \approx_{\lambda,\mu} B$ if $A/\lambda - \mu \leq B \leq \lambda A + \mu$. If $A$ and $B$ are quantities which depend on a pair of group elements $g$ and $h$ (for example if $A = d_G(g,h)$ and $B = d_X(f(x),g(x))$ for some $f: G \rightarrow X$ and metric space $X$) then we might write $A \approx B$ if $A \approx_{\lambda,\mu} B$ for some $\lambda \geq 1$ and $\mu \geq 0$ which don't depend on $g$ and $h$ (in other words, $\lambda$, $\mu$ are \textit{uniform} constants). 
    \item We define $\{\{A\}\}_L$ to be $0$ if $A \leq L$ and to be $A$ if $A > L$.
\end{itemize}
\end{notation}

Let $Y = gH$ and let $\rho_Y: G \rightarrow Y$ be \textit{some} choice of nearest-point projection onto $Y$.

\begin{theorem}[Sisto] \label{thm.relhypdistanceformula}
There exists $L_0 \in \R$ such that for each $L \geq L_0$ there exist $\lambda, \mu$ such that the following holds. If $g,h \in G$ then 
\[d_G(g,h) \approx_{\lambda,\mu} \sum_{Y \in \G} \{\{ d_G(\rho_Y(g),\rho_Y(h)) \}\}_L + d_{\hat{G}}(g,h) \]
\end{theorem}

It follows from \Cref{thm.relhypdistanceformula} and \cite[Theorem 6.4]{BBFS} that you may approximate distances in $G$ via the images of group elements in $\qtomsg$, $\hat{G}$ and $X(G)$. See \cite[Theorem 4.1]{MS} for the proofs of these corollaries.   

\begin{corollary} \label{cor.embedding1}
For sufficiently large $K$, the map $G \rightarrow \qtomsg \times \hat{G}$ given by
\[g \mapsto (\iota_1(g), g)\]
is a quasiisometric embedding. 
\end{corollary}

\begin{corollary} \label{cor.embedding2}
For sufficiently large $K$, the map $G \rightarrow \qtomsg \times X(G)$ given by
\[g \mapsto (\iota_1(g), g)\]
is a quasiisometric embedding.
\end{corollary}

\section{Diaries} \label{chap.diaries}

In \textit{100 Years of Solitude} a plague of insomnia afflicts the town of Macondo. The most troubling effect of the insomnia plague is that it ultimately leads to memory loss. 

\begin{displayquote}
One day he was looking for the small anvil that he used for laminating metals and he could not remember its name. His father told him: “Stake.” Aureliano wrote the name on a piece of paper that he pasted to the base of the small anvil: \textit{stake}. In that way he was sure of not forgetting it in the future. It did not occur to him that this was the first manifestation of a loss of memory, because the object had a difficult name to remember. But a few days later he discovered that he had trouble remembering almost every object in the laboratory. Then he marked them with their respective names so that all he had to do was read the inscription in order to identify them. When his father told him about his alarm at having forgotten even the most impressive happenings of his childhood, Aureliano explained his method to him, and José Arcadio Buendía put it into practice all through the house and later on imposed it on the whole village. With an inked brush he marked everything with its name: \textit{table, chair, clock, door, wall, bed, pan.} He went to the corral and marked the animals and plants: \textit{cow, goat, pig, hen, cassava, caladium, banana.} Little by little, studying the infinite possibilities of a loss of memory, he realized that the day might come when things would be recognized by their inscriptions but that no one would remember their use. Then he was more explicit. The sign that he hung on the neck of the cow was an exemplary proof of the way in which the inhabitants of Macondo were prepared to fight against loss of memory: \textit{This is the cow. She must be milked every morning so that she will produce milk, and the milk must be boiled in order to be mixed with coffee to make coffee and milk.} Thus they went on living in a reality that was slipping away, momentarily captured by words, but which would escape irremediably when they forgot the values of the written letters. \\
\textit{100 Years of Solitude, Gabriel Garc\'{i}a M\'{a}rquez}
\end{displayquote}

The human effort to record daily events in journals, databases and diaries is an attempt to record as effectively as possible the infinite detail of day-to-day life in finite packets of information. To record everything is a lost cause: you can only hope to record the key features of an event. For example, it is a waste of time to write in your diary \textit{The sun rose today.} because this always happens. Far better to take note of the events that distinguish the day from every other.

\subsection{Diaries}

A \textit{diary}, as defined in this paper, and perhaps in its normal sense as well, is some system of converting infinite quantities of data that occur on days $1,2,3, \dots, i$ into finite diary entries at the end of each day. 

\begin{definition}
A \textit{diary} is a height-preserving and order-preserving map (with respect to some chosen basepoints) from an infinite-valence tree to a uniformly locally finite tree. 
\end{definition}

\begin{remark}
Since a diary is a graph homomorphism, it is always $1$-Lipschitz.
\end{remark}

Throughout this section, we will assume that the infinite-valence trees have some extra structure; we will assume that they are \textit{sentence-trees} as defined in \Cref{def.sentencetree}. So, for the remainder of this section, we suppose that $A$ is a finite alphabet of letters, we suppose that $W$ is the set of finite but non-empty words on $A$ and we suppose that $T_W$ is the associated sentence-tree. We will also assume that the codomain of a diary is a word-tree $T_\Omega$ for some finite set $\Omega$. 

\begin{example}
Let $\last: W \rightarrow A$ be the function which outputs the last letter of a word $w \in W$. Then there exists a diary $\ld: T_W \rightarrow T_{A}$ which just outputs the last letter of every word in $\alpha \in T_W$. More precisely, 
\[\ld(\ol{w_1} \dots \ol{w_i}) = \last(w_1) \dots \last(w_i)\]
For example, $\ld(\ol{abc} \ \ol{bc} \ \ol{aa}) = c c a$.
\end{example}

The term \textit{diary} comes from the following metaphor. Every day, our protagonist Alice receives some sort of input data (it could be an image, or a voice note, or a text message) which may be very long, in the form of a finite word in $W$. Every day, she has only a finite amount of time to record the data in her diary, and there are only finite many letters in the English alphabet, and she can only write at a certain pace, and therefore the daily entry in her diary may be summed up as an element of some finite set $\Omega$. Suppose that on days $1,2, \dots, i$ Alice receives the input data $w_1$, $w_2$, \dots, $w_i \in W$ respectively and in response to this data Alice records $v_1$, $v_2$, \dots, $v_i \in \Omega$ in her diary. We have thus defined a map $D: T_W \rightarrow T_\Omega$ by $D(\ol{w_1} \dots \ol{w_i}) = \ol{v_1} \dots \ol{v_i}$ and it is a diary: it is height-preserving because Alice writes exactly one diary entry per day and it is order-preserving because Alice does not "re-write" diary entries.

We will be interested in understanding whether a given diary $D: T_W \rightarrow T_\Omega$ can successfully distinguish between two sentences $\alpha, \beta \in T_W$ (meaning that $D(\alpha) \neq D(\beta)$). As the following example demonstrates, this theory can be applied to understanding when we can upgrade quasiisometric embeddings into products of infinite-valence trees to quasiisometric embeddings into products of binary trees. The following example explains how the theorems in this section are used to prove \Cref{thm.main}.

\begin{example} \label{ex.diariesms}
Suppose $X$ is a metric space and suppose we have a $(\lambda,\mu)$-quasiisometric embedding $F: X \rightarrow \prod_{q=1}^Q T_W$. Let $P$ be a property of \textit{pairs} of sentences in $T_W$ (or, what is the same thing, $P$ is a subset of the product $T_W \times T_W$) and suppose we have a diary $D: T_W \rightarrow T_\Omega$ for which there exists a constant $M \geq 1$ such that 
\begin{equation} \label{eq.ablb}
d_{T_\Omega}(D(\alpha), D(\beta)) \geq \frac{1}{M} d_{T_W}(\alpha,\beta)
\end{equation}
for all pairs $\alpha$ and $\beta$ satisfying $P$. We are interested in whether the composition
\begin{equation} \label{eq.composition}
X \xrightarrow{F} \prod_{q=1}^Q T_W \xrightarrow{D \times D \times \dots \times D} \prod_{q=1}^Q T_\Omega
\end{equation}
is a quasiisometric embedding. Suppose products are given the $L^1$ metric.

Let $x,z \in X$ and write $F(x) = (\alpha_1, \alpha_2, \dots, \alpha_Q)$ and $F(z) = (\beta_1, \beta_2, \dots, \beta_Q)$. Then $d(F(x),F(z)) = \sum_{q=1}^Q d_{T_W}(\alpha_q,\beta_q)$. Suppose $\mathfrak{q} \in \{1,2,\dots, Q\}$ is such that $d_{T_W}(\alpha_\q,\beta_\q) = \max_q d_{T_W}(\alpha_q,\beta_q)$. Then $d_{T_W}(\alpha_\q,\beta_\q) \geq \frac{1}{Q} d(F(x),F(z))$. 

Suppose we manage to prove that the pair $\alpha_\q, \beta_\q \in T_W$ arising in this manner satisfy the property $P$. Then I claim that \eqref{eq.composition} would be a quasiisometric embedding. The diary $D$ is $1$-Lipschitz and the map $F$ is a quasiisometric embedding so the composition \eqref{eq.composition} is coarsely Lipschitz. For the lower bound of the quasiisometric inequality we have
\[d_{T_\Omega}(D(\alpha_\q), D(\beta_\q)) \geq \frac{1}{M} d_{T_W}(\alpha_\q,\beta_\q) \geq \frac{1}{QM} d(F(x),F(z)) \geq \frac{1}{\lambda QM} d_X(x,z) - \mu \]
where the first lower bound follows from the fact that $\alpha_\q$ and $\beta_\q$ satisfy $P$. As $\Omega$ is finite, $T_\Omega$ is quasiisometric to a binary tree (assuming $\absval{\Omega} > 1$). 
\end{example}

In the remainder of this section, we will provide four criteria $(\aries)$, $(\leo)$, $(\virgo)$ and $(\taurus)$ on pairs of sentences, analogous to the property $P$ in \Cref{ex.diariesms}, which imply that a certain diary $D$ coarsely preserves the metric on those sentences as in \eqref{eq.ablb}.

The ideas behind these criteria stem from the work of Buyalo, Dranishnikov and Schroeder \cite{BDS} on \textit{Alice's Diary}. Alice's Diary allows one to quasiisometrically embed subsets of products of infinite-valence trees into products of uniformly bounded valence trees in certain circumstances. However, the results that Buyalo, Dranishikov and Schroeder prove on Alice's Diary (in particular \cite[Proposition 5.5]{BDS}) are phrased in an \textit{ad hoc} manner. It was my aim to unravel the ideas and theorems which are latent in their work, and to provide easy to use and flexible criteria so that one may find quasiisometric embeddings into products of binary trees in many contexts.

I would also like to express my appreciation of Buyalo, Dranishnikov and Schroeder's paper. They chose to use metaphor to help the reader understand the Alice's Diary function, and it was this metaphor and the questions it raised that sparked my curiosity on the topic. 

\subsection{Finite statistics}

\begin{definition}
A \textit{finite statistic} is a function $\stat: T_W \rightarrow \Omega$ where $\Omega$ is some finite set. If we wish to specify the domain of $\stat$ we say that it is a finite statistic \textit{on $T_W$}. 
\end{definition}

You might think of $\stat(\alpha)$ as being a fact about $\alpha$ that contains a uniformly bounded amount of information. It is a \textit{statistic} in the sense that it is a value which can be calculated from some input data $\alpha$. We can use finite statistics to create diaries as follows. 

\begin{definition}
Suppose we have a finite statistic $\stat: T_W \rightarrow \Omega$ and suppose $T_\Omega$ is the word-tree on $\Omega$. Then the diary $D: T_W \rightarrow T_\Omega$ defined by
\[D(\ol{w_1} \ \ol{w_2} \dots \ol{w_i}) = \stat(\ol{w_1}) \stat(\ol{w_1} \ \ol{w_2}) \dots \stat(\ol{w_1} \ \ol{w_2} \dots \ol{w_i}) \]
is the \textit{diary associated to $\stat: T_W \rightarrow \Omega$}. 
\end{definition}

\begin{example} \label{example.truncdiary}
Let $\kappa \in \N$ be a constant and let $\Omega_\kappa$ denote the set of all words in $W$ of length at most $\kappa$. Consider the map $\trunc_\kappa: T_W \rightarrow \Omega_\kappa$ which is such that $\trunc_\kappa(\ol{w_1} \ \ol{w_2} \dots \ol{w_i})$ is equal to the final $\kappa$ letters of the word $w_i$. Then $\trunc_\kappa$ is a finite statistic on $T_W$. The associated diary simply returns the final $\kappa$ letters of each word in the sentence. If the letters in $A$ are thought of as \textit{events}, and the words in $W$ as \textit{sequences of events that occur on a given day}, then this diary is the one in which Alice records at the end of the day the $\kappa$ most recent events. 
\end{example}

We may now provide our first criterion for when there exists a diary that coarsely preserves a pair of sentences. 

Let $\scalf$ be a finite collection of finite statistics on the same sentence-tree $T_W$ (but the finite statistics are allowed to have different codomains). Let $0 \leq \delta < 1$, $J \in \N$ be constants. Consider the following property, pronounced \textit{Aries}, that a pair of sentences $\alpha$ and $\beta$ in the forms described by \Cref{notation.sentence} might possess.

\begin{definition} \label{def.aries}
We say that $\alpha$, $\beta$ satisfy $\aries(\scalf,\delta,J)$ if there exists some $1 \leq j \leq \delta \min(m,n) + J$ and $\stat \in \scalf$ such that
\[\stat(\ol{u_1} \dots \ol{u_p} \ \ol{u_{p+1}} \dots \ol{u_{p+j}}) \neq \stat(\ol{u_1} \dots \ol{u_p} \ \ol{u_{p+1}'} \dots \ol{u_{p+j}'})\]
\end{definition}

Property ($\aries$) is saying that there exists a \textit{fact} about $\alpha$ and $\beta$ which distinguishes them and which only requires a finite amount of information to state. The purpose of the constants $\delta$ and $J$ is to ensure that this distinguishing fact separates $\alpha$ and $\beta$ sufficiently close to their lowest common ancestor $\ol{u_1} \dots \ol{u_p}$. 

\begin{theorem} \label{thm.aries}
There exists a finite set $\Omega = \Omega(\scalf)$, a diary $D = D(\scalf)$ which is a function $D: T_W \rightarrow T_\Omega$, and a constant $M = M(\delta,J) \geq 1$, such that   
\[\frac{1}{M} d_{T_W}(\alpha,\beta) \leq d_{T_{\Omega}}(D\alpha,D\beta) \leq d_{T_W}(\alpha,\beta)\]
for all $\alpha,\beta$ satisfying $\aries(\scalf,\delta,J)$.
\end{theorem} 

\begin{proof}
Write $\scalf = \{\stat^k: T_W \rightarrow \Omega^k : 1 \leq k \leq K \}$. We begin by combining all the finite statistics in $\scalf$ into a single finite statistic $\bigstat: T_W \rightarrow \Omega$. Set $\Omega = \prod_{k = 1}^K \Omega^k$. We define $\bigstat: T_W \rightarrow \Omega$ by 
\[
\bigstat(\ol{w_1} \dots \ol{w_i}) = 
\begin{pmatrix}
\stat^1(\ol{w_1} \dots \ol{w_i}) \\
\stat^2(\ol{w_1} \dots \ol{w_i}) \\
\vdots \\
\stat^K(\ol{w_1} \dots \ol{w_i})
\end{pmatrix}
\]
Let $D: T_W \rightarrow T_\Omega$ be the diary associated to $\bigstat$. More explicitly 
\[D(\ol{w_1} \dots \ol{w_i}) = \bigstat(\ol{w_1}) \bigstat(\ol{w_1} \  \ol{w_2}) \dots \bigstat(\ol{w_1} \ \ol{w_2} \dots \ol{w_i})\]
Suppose $\alpha,\beta$ satisfy $\aries(\scalf, \delta,J)$. Then there exists $1 \leq j \leq \delta \min(m,n) + J$ and $1 \leq k \leq K$ such that $\stat^k(\ol{u_1} \dots \ol{u_p} \dots \ol{u_{p+j}}) \neq \stat^k(\ol{u_1} \dots \ol{u_p} \dots \ol{u_{p+j}'})$. It follows that $\bigstat(\ol{u_1} \dots \ol{u_p} \dots \ol{u_{p+j}}) \neq \bigstat(\ol{u_1} \dots \ol{u_p} \dots \ol{u_{p+j}'})$ and so 
\[d_{T_\Omega}(D\alpha,D\beta) \geq (p+m) - (p+j-1) + (p+n) - (p+j-1) = m+n - 2j + 2\] 
We have two cases: either $m+n < 4J / (1 - \delta)$ or $m+n \geq 4J / (1 - \delta)$. In the first case, we have
\[d_{T_\Omega}(D\alpha,D\beta) \geq m+n - 2j + 2 \geq 2 = \frac{1 - \delta}{2J} \frac{4J}{1 - \delta} \geq \frac{1 - \delta}{2J} (m+n)\]
In the second case, we have
\[d_{T_\Omega}(D\alpha,D\beta) \geq m+n - 2j + 2 \geq m+n - \delta (m+n) - 2J \geq \frac{1-\delta}{2} (m+n)\]
So if we set $M = 2J / (1 - \delta)$ then we are done. 
\end{proof}

So we may now address the situation described in \Cref{ex.diariesms}: if the sentences $\alpha_\q, \beta_\q \in T_W$ satisfy $\aries(\scalf,\delta,J)$ for some uniform choices of $\scalf$, $\delta$ and $J$, then the diary $D = D(\scalf)$ given by \Cref{thm.aries} will be such that \eqref{eq.composition} is a quasiisometric embedding.  

\begin{definition} \label{def.leo}
We say that $\alpha,\beta \in T_W$ satisfy $\leo(\scalf,J)$ if they satisfy $\aries(\scalf,0,J)$. This property is pronounced \textit{Leo}.
\end{definition}

In the next section, we are going to introduce Buyalo, Dranishnikov and Schroeder's \textit{Alice's Diary}. Afterwards, in \Cref{subsec.ls}, Alice's Diary will be used to create diaries that are more powerful than the ones which are made via finite statistics. 

\subsection{Alice's Diary}

Let us add some more texture to our diary metaphor. We refer to a letter chosen from the finite alphabet $A$ as an \textit{event}. On each day, several events occur in some order (thereby forming a word in $W$), and at the end of each day Alice dedicates $\kappa \in \N$ minutes to diary writing. Alice can write down one letter per minute and so records at most $\kappa$ letters in her diary each day. On each page she writes down only one letter, and so on each day she fills in at most $\kappa$ pages of her diary. The pages of her diary corresponding to a single day form a \textit{chapter}. Alice wants to record the events on days $1,2,3,4,5,6,7, \dots$ to the best of her ability. She has perfect memory of past events but seems to always be paranoid that one day she might forget everything and only her diary will be left behind to remind her of what occurred. 

If $\Omega_\kappa$ is the set of all words of length at most $\kappa$ on the alphabet $A$, then Alice's diary writing process defines a diary $D: T_W \rightarrow T_{\Omega_\kappa}$ (assuming she responds deterministically to possible sequences of events). A vertex in $T_W$ corresponds to some possible sequence of events over the course of several days, and a vertex in $T_{\Omega_\kappa}$ corresponds to a possible sequence of chapters in the diary.

The diary described in \Cref{example.truncdiary} in which Alice simply records the final $\kappa$ events of the day provides an example of this sort of diary. However this truncation diary might be very inefficient if fewer than $\kappa$ events take place on some days. The inefficiency of the truncation diary motivates the definition of \textit{Alice's Diary}. Alice's Diary was invented by Buyalo, Dranishnikov and Schroeder \cite{BDS}. 

\begin{definition}
Alice's Diary $\ad_\kappa: T_W \rightarrow T_{\Omega_\kappa}$ is characterised by the following rule: \textit{Alice always records the most recent unrecorded event first}. 
\end{definition}

So Alice's Diary is a greedy algorithm for creating diaries. An example will help.

\begin{example} \label{diaryexample}
Let $A = \{a,b,c\}$ and $\kappa = 3$. Suppose that on day one, at first event $a$ occurs, then event $b$, then event $a$ again and then event $c$. So the associated word is $abac$. Since $\kappa = 3$, Alice only has time to write down $3$ letters at the end of the day, and so she records $cab$. That's because event $c$ occured most recently, then event $a$, then event $b$. Suppose that on day two, the associated word is $cb$. Then Alice writes down in her diary $bca$. The final $a$ she writes down was the first event of day one. Now suppose that on days three, four and five, the associated events are $accc$, $bcbc$ and $a$ respectively. Then the chapters that Alice records in her diary at the end of days three, four and five are $ccc$, $cbc$, $aba$. In the chapter corresponding to day five, the first page records an event from day five, the second page records an event from day four, and the third page records an event from day three. So we have shown that $\ad_3: T_W \rightarrow T_\Omega$ satisfies $\ad_3(\ol{abac} \ \ol{cb} \ \ol{accc} \ \ol{bcbc} \ \ol{a}) = \ol{cab} \ \ol{bca} \ \ol{ccc} \ \ol{cbc} \ \ol{aba}$.
\end{example}

We would now like to understand how much information $\ad_\kappa$ captures about the original sentence. But first, we want to assume that our sentences $\alpha$ have some further structure. 

\begin{definition} \label{def.starred}
Suppose that our alphabet $A$ contains a special character $\star$. We say that a word $u \in W$ is \textit{starred} if $u = \star w$ for some word $w \in W$ that does not contain the letter $\star$. We say that a sentence $\alpha \in T_W$ is \textit{starred} if all the constituent words of $\alpha$ are starred. Maybe $\star$ corresponds to the event of sunrise. 
\end{definition}

The primary reason to work with starred words is so that you can say the following: if $u$ and $u'$ are distinct starred words then there exist distinct letters $a \in u$ and $a' \in u'$ with the same distance to the end of their words.

In \Cref{lem.diary1}, \Cref{lem.diary2}, \Cref{lem.diary3} and \Cref{thm.diary} that follow, we will always assume that we have two starred sentences $\alpha = \ol{\star u_1} \ \ol{\star u_2} \dots \ol{\star u_m}$ and $\beta = \ol{\star u_1'} \ \ol{\star u_2'} \dots \ol{\star u_n'}$ in $T_W$ and that $\ad_\kappa(\alpha) = \ol{v_1} \ \ol{v_2} \dots \ol{v_m}$ and $\ad_\kappa(\beta) = \ol{v_1'} \ \ol{v_2'} \dots \ol{v_n'}$.

\begin{lemma} \label{lem.diary1}
Suppose the diary chapter $v_i$ has the form $v_i = u \star v$ where $u$ does not contain the letter $\star$. Then $u_i = \ola{u}$. 
\end{lemma}

\begin{proof}
At the end of day $i$, Alice sets out writing her diary. She starts by recording the events of day $i$. If she does not fully record the events of day $i$, then there can be no $\star$ in that day's chapter. Since $v_i = u \star v$, it follows that Alice records day $i$ in its entirety. Since Alice always records the most recent events first, her recording of day $i$ must come at the start of the word $v_i$ and the $\star$ in $\star u_i$ must be the first $\star$ of chapter $i$. It follows that $u_i = \ola{u}$. 
\end{proof}

\begin{lemma} \label{lem.diary2}
Suppose that $\ad_\kappa(\alpha) = \ad_\kappa(\beta)$. If a word $\star u_i$ has length at most $\kappa$ then $u_i' = u_i$. Similarly, if a word $\star u_j'$ has length at most $\kappa$ then $u_j = u_j'$.
\end{lemma}

\begin{proof}
Suppose $\star u_i$ has length at most $\kappa$. It follows that $v_i$ is of the form $\ola{u_i} \star v$ for some word $v$. Since $v_i' = v_i$, by \Cref{lem.diary1} we deduce that $u_i' = u_i$. The case when $\star u_j'$ has length at most $\kappa$ follows similarly. 
\end{proof}

\begin{lemma} \label{lem.diary3}
Let $i \leq \min(m,n)$ and suppose $a \in \star u_i$ and $a' \in \star u_i'$ are letters such that 
\begin{enumerate}
    \item $a,a'$ have the same distance from the end of their words;
    \item there exists $i \leq j \leq \min(m,n)$ and $1 \leq k \leq \kappa$ such that event $a$ is recorded on page $k$ of diary chapter $v_j$;
    \item $u_l$ has the same length as $u_l'$ for $i+1 \leq l \leq j$.
\end{enumerate} 
Then $a'$ is recorded on page $k$ of diary chapter $v_j'$ 
\end{lemma}

\begin{proof}
Write $\star u_i = vaw$ and $\star u_i' = v'a'w'$. The content of the first condition is that $w$ has the same length as $w'$. The position of $a$ within the diary words $\ol{v_i} \ \ol{v_{i+1}} \dots \ol{v_{j}}$ depends only on the lengths of the words $w, u_{i+1}, \dots, u_j$. Similarly, the position of $a'$ within the diary words $\ol{v_i'} \ \ol{v_{i+1}'} \dots \ol{v_{j}'}$ depends only on the lengths of the words $w', u_{i+1}', \dots, u_j'$. It follows that $a$ and $a'$ appear on the same page of $v_j$ and $v_j'$ respectively. 
\end{proof}

\begin{theorem} \label{thm.diary}
Suppose that $\ad_\kappa(\alpha) = \ad_\kappa(\beta)$. Let $a$ be a letter in the word $\star u_i$ and let $a'$ be a letter in the word $\star u_i'$ and suppose 
\begin{enumerate}
    \item $a,a'$ have the same distance from the end of their words;
    \item The events $a$ and $a'$ are recorded in the diaries of $\alpha$ and $\beta$.
\end{enumerate} 
Then there exists $i \leq j \leq m$ and $1 \leq k \leq \kappa$ such that $a,a'$ appear on page $k$ of chapters $v_j$ and $v_j'$ respectively. Consequently, $a,a'$ correspond to the same element of $A$.
\end{theorem}

\begin{proof}
Write $\star u_i = vaw$ and $\star u_i' = v'a'w'$. The content of the first condition on $a,a'$ is that $w$ has the same length as $w'$. If $a$ is recorded within $v_i$ then it follows that $aw$ and $a'w'$ both have length at most $\kappa$ and the conclusion clearly holds. So assume that $a$ is recorded within $v_j$ for some $i+1 \leq j \leq m$. Then clearly $\star u_j$ has length at most $\kappa$ and so \Cref{lem.diary2} implies that $u_j = u_j'$. We will prove by induction that $u_l = u_l'$ for $i+1 \leq l \leq j$. So fix $i+1 \leq l \leq j$ and suppose $u_{l+1} = u_{l+1}', u_{l+2} = u_{l+2}', \dots, u_j = u_j'$. Looking for a contradiction, suppose that $u_l \neq u_l'$. Then we can find distinct letters $b \in \star u_l$ and $b' \in \star u_l'$ with the same distance to the end of their words. We also know that $b$ must appear in one of the diary words $v_l$, $v_{l+1}, \dots, v_{j}$ because for Alice to record event $a$ in chapter $v_j$, she must already have recorded event $b$ on some earlier page of the diary. It then follows from \Cref{lem.diary3}, and the fact that the diaries are equal, that $b$ and $b'$ correspond to the same element of $A$, which is a contradiction. Thus $u_l = u_l'$.

So we have shown that $u_l = u_l'$ for all $i+1 \leq l \leq j$. We can now apply \Cref{lem.diary3} again to conclude that $a$ and $a'$ appear on the same page of $v_j = v_j'$ and thus are equal. 
\end{proof}

In my view, the above theorem is the fundamental observation on Alice's Diary. To paraphrase it: if you know two diaries are equal, and if you have two events that are 
\begin{itemize}
    \item recorded in the diary;
    \item occur on the same day;
    \item are in the same position (with respect to the ends of their words);
\end{itemize}
then those two events correspond to the same element of $A$. Taking the contrapositive of \Cref{thm.diary}, we see that distinct, recorded events $a \in u_i$, $a' \in u_i'$ in the same position imply distinct diaries. 

It leads us naturally to trying to understand when we can guarantee that an event is recorded in the diary. We need some further definitions. 

\begin{definition} \label{def.tailsentence}
Suppose we have a sentence $\alpha = \ol{u_1} \ \ol{u_2} \dots \ol{u_m} \in T_W$ and suppose $a$ is some letter in $u_i$. We can write $u_i = vaw$ where $v$ and $w$ are some words on the alphabet $A$. Then the sentence 
\[ \ol{aw} \ \ol{u_{i+1}} \dots \ol{u_m} \]
is the \textit{tail-sentence} of $a$. Similarly, the sentence
\[\ol{u_1} \dots \ol{u_{i-1}} \ \ol{va}\]
is the \textit{head-sentence} of $a$. 
\end{definition}

In terms of our analogy, the tail-sentence of $a$ consists of $a$ and every event that happens after $a$. The head-sentence of $a$ consists of $a$ and every event that happens before $a$. 

\begin{definition} \label{def.AWL}
The \textit{average word length (AWL)} of a sentence $\alpha = \ol{u_1} \ \ol{u_2} \dots \ol{u_m} \in T_W$ is the amount of letters in $u_1 u_2 \dots u_m$ divided by $m$. We will often write $\awl(\alpha)$ to denote the average word length of $\alpha$. 
\end{definition}

\begin{proposition} \label{prop.AWL}
Suppose $\alpha = \ol{u_1} \ \ol{u_2} \dots \ol{u_m} \in T_W$ and let $a$ be some letter in $u_i$. Suppose that the tail-sentence of $a$ has average word length $N \geq 1$. If $\kappa \geq N$ then $a$ is recorded in $\ad_\kappa(\alpha)$.
\end{proposition}

\begin{proof}
We will prove the contrapositive. Suppose that $a$ is not recorded in $\ad_\kappa(\alpha)$. Write $\ad_\kappa(\alpha) = \ol{v_1} \ \ol{v_2} \dots \ol{v_m}$. A chapter only has fewer than $\kappa$ pages when Alice has successfully recorded all past events already. It follows that $v_i$, $v_{i+1}$, \dots, $v_m$ all have $\kappa$ pages. Since Alice would only record an event that happened prior to $a$ if $a$ had already been recorded, we know that every page of the chapters $v_i$, $v_{i+1}$, \dots, $v_m$ corresponds to an event that happened after $a$. It follows that the tail-sentence of $a$ has at least $\kappa(m - i + 1) + 1$ letters. It consists of $m-i+1$ words and hence the average word length of the tail-sentence is greater than $\kappa$.
\end{proof}

\begin{corollary} \label{cor.AWL}
Suppose $\alpha$, $\beta$ are starred sentences in $T_W$. Write
\[\alpha = \ol{u_1} \ \ol{u_2} \dots \ol{u_p} \ \ol{u_{p+1}} \dots \ol{u_{p+m}} \] 
and 
\[\beta = \ol{u_1} \ \ol{u_2} \dots \ol{u_p} \ \ol{u_{p+1}'} \dots \ol{u_{p+n}'}\]
where $u_{p+1} \neq u_{p+1}'$. Write $\ad_\kappa(\alpha) = \ol{v_1} \ \ol{v_2} \ \dots \ \ol{v_{p+m}}$ and $\ad_\kappa(\beta) = \ol{v_1'} \ \ol{v_2'} \ \dots \ \ol{v_{p+n}'}$. Suppose we have letters $a \in u_{p+j}$ and $a' \in u_{p+j}'$ which satisfy the following
\begin{itemize}
    \item $a$ and $a'$ correspond to different elements of the alphabet $A$;
    \item $a$ and $a'$ are at equal distance from the ends of their words;
    \item the AWL of the tail-sentence of $a$ is $N$ and the AWL of the tail-sentence of $a'$ is $N'$;
\end{itemize}
Let $i \in \N_0$ be such that $p+j+i \leq p+\min(m,n)$. If $\kappa \geq N \frac{m-j+1}{i+1}$ and $\kappa \geq N' \frac{n-j+1}{i+1}$ then 
\[d(\ad_\kappa(\alpha),\ad_\kappa(\beta)) \geq d(\alpha,\beta) -2j - 2i \]
\end{corollary}

\begin{proof}
Suppose the amount of letters in the tail-sentence of $a$ is $l$ and the amount of letters in the tail-sentence of $a'$ is $l'$. Then $N = \frac{l}{m-j+1}$ and $N' = \frac{l'}{n-j+1}$. Consider the truncated sentences $\alpha_{p+j+i} = \ol{u_1} \dots \ol{u_{p+j+i}}$ and $\beta_{p+j+i} = \ol{u_1} \dots \ol{u_{p+j+i}'}$. The AWL of the tail-sentence of $a \in \alpha_{p+j+i}$ is at most $\frac{l}{i+1} = N \frac{m-j+1}{i+1}$. So, by \Cref{prop.AWL}, if $\kappa \geq N \frac{m-j+1}{i+1}$ then $a$ is recorded in $\ad_\kappa(\alpha_{p+j+i})$. Similarly, if $\kappa \geq N' \frac{n-j+1}{i+1}$ then $a'$ is recorded in $\ad_\kappa(\beta_{p+j+i})$. It then follows from \Cref{thm.diary} that $\ad_\kappa(\alpha_{p+j+i}) \neq \ad_\kappa(\beta_{p+j+i})$ and the corollary follows.
\end{proof}

\subsection{Linear statistics} \label{subsec.ls}

Intuitively, where a finite statistic is some uniformly bounded quantity that can be calculated from a sentence $\ol{w_1} \ \ol{w_2} \dots \ol{w_i}$, a linear statistic is a quantity that can be calculated from $\ol{w_1} \ \ol{w_2} \dots \ol{w_i}$ that contains a linear amount of information with respect to some variable $c \in \N_0$. 

\begin{definition} \label{def.lstat}
Suppose, as ever, we have a finite alphabet $A$ and an associated sentence-tree $T_W$. Suppose we have another, potentially distinct, finite alphabet $B$ and let $T_B$ be the word-tree on $B$. Given a number $c \in [0,\infty]$ let $\Omega_c$ denote all words on the alphabet $B$ of length at most $c$. So we can identify the sets $\Omega_c$ with subsets of the vertices of $T_B$. A \textit{linear statistic} on $T_W$ consists of a constant $\tau \geq 1$ and a collection of finite statistics $\{\stat_c: T_W \rightarrow \Omega_{\tau c} : c \in \N_0\}$ such that if $c \leq c'$ then $\stat_{c'}(\ol{w_1} \dots \ol{w_i})$ descends from $\stat_{c}(\ol{w_1} \dots \ol{w_i})$ in $T_B$. We will typically denote a linear statistic just by $\stat_c: T_W \rightarrow \Omega_{\tau c}$ as opposed to $\{\stat_c : T_W \rightarrow \Omega_{\tau c} : c \in \N_0\}$. Associated to every linear statistic $\stat_c: T_W \rightarrow \Omega_{\tau c}$ is an \textit{infinite statistic} $\stat_{\infty}: T_W \rightarrow \Omega_\infty$ defined by $\stat_\infty(\ol{w_1} \dots \ol{w_i}) = \lim_{c \rightarrow \infty} \stat_c(\ol{w_1} \dots \ol{w_i})$. 
\end{definition}

So $\stat_{c}(\ol{w_1} \dots \ol{w_i})$ is a collection of facts about the sentence $\ol{w_1} \dots \ol{w_i}$ that contain steadily more information as $c$ increases. Since the image of $\stat_c$ is in $\Omega_{\tau c}$, the amount of information contained in the word $\stat_c(\ol{w_1} \dots \ol{w_i})$ is comparable in size to $c$. If $\stat_{c'}(\ol{w_1} \dots \ol{w_i})$ descends from $\stat_{c}(\ol{w_1} \dots \ol{w_i})$ that means that the former contains at least as much information as the latter. This happens when when $c \leq c'$. 

Let's now consider some examples of linear statistics. 

\begin{example} \label{ex.ltrunc}
Let $B = A$. If we define $\trunc_c(\ol{w_1} \dots \ol{w_i})$ to be the final $\tau c$ letters of the word $u_i$ written backwards (i.e. starting from the final letter) then $\trunc_c: T_W \rightarrow \Omega_{\tau c}$ is a linear statistic. 
\end{example}

\begin{example} \label{ex.oop}
Let $B = A$. Let $l = \length(w_1 w_2 \dots w_i)$. Let $a_1, a_2, \dots, a_l$ denote all the letters in the word $w_1 w_2 \dots w_i$. Let $\sigma \in S_l$ be a permutation which we call the \textit{order of priority}. We can define $\oop(\sigma)_c: T_W \rightarrow \Omega_{\tau c}$ by letting $\oop(\sigma)_c(\ol{w_1} \dots \ol{w_i})$ be the first $\tau c$ letters of the word $a_{\sigma(1)} a_{\sigma(2)} \dots a_{\sigma(l)}$. Then $\oop(\sigma)$ is a linear statistic.
\end{example}

We now come to our second criterion. Let $\scall$ be a finite collection of linear statistics on the same tree $T_W$ (but the linear statistics are allowed to have different codomains). Let $0 \leq \delta < 1$, $J \in \N$, $N > 0$, $\epsilon > 0$ be constants. Consider the property, pronounced \textit{Virgo}, that a pair of vertices $\alpha, \beta \in T_W$ of the form described in \Cref{notation.sentence} might possess. It is somewhat analogous to ($\aries$).

\begin{definition} \label{def.virgo}
We say that $\alpha$, $\beta$ satisfy $\virgo(\scall,\delta,J,N,\epsilon)$ if there exists some $1 \leq j \leq \delta \min(m,n) + J$ and $\stat_c \in \scall$ such that 
    \begin{enumerate}
        \myitem{($\virgo 1$)} \label{condition.virgo1} $\awl(\ol{u_{p+j+1}} \dots \ol{u_{p+m}})$ and $\awl(\ol{u_{p+j+1}'} \dots \ol{u_{p+n}'})$ are both at most $N$;
        \myitem{($\virgo 2$)} \label{condition.virgo2} $\length(u_{p+j}) \geq \epsilon (m+n)$ or $\length(u_{p+j}) \geq \epsilon (m+n)$ or $u_{p+j} \neq u_{p+j}'$;
        \myitem{($\virgo 3$)}\label{condition.virgo3} $\stat_{m+n}(\ol{u_1} \dots \ol{u_p} \ \ol{u_{p+1}} \dots \ol{u_{p+j}}) \neq \stat_{m+n}(\ol{u_1} \dots \ol{u_p} \ \ol{u_{p+1}'} \dots \ol{u_{p+j}'})$.
    \end{enumerate}
\end{definition}

($\virgo$) is saying that there exists some fact about $\alpha$ and $\beta$ which distinguishes them. However, unlike in the case of ($\aries$) and ($\leo$), this fact can now contain an unbounded amount of information, as long as the amount of information is at most linear with respect to $m+n = d_{T_W}(\alpha,\beta)$. The cost of this upgrade is that you very much need the $\awl$ of the tail-sentences to be uniformly bounded.

\begin{theorem} \label{thm.virgo}
There exists a finite set $\Omega = \Omega(\scall,\delta,J,N,\epsilon)$, a diary $D = D(\scall,\delta,J,N,\epsilon)$ which is a function $D: T_W \rightarrow T_\Omega$, and a constant $M = M(\delta,J) \geq 1$, such that 
\[\frac{1}{M} d_{T_W}(\alpha,\beta) \leq d_{T_{\Omega}}(D\alpha,D\beta) \leq d_{T_W}(\alpha,\beta)\]
for all $\alpha,\beta$ satisfying $\virgo(\scall,\delta,J,N,\epsilon)$.
\end{theorem}

The theorem is powerful for this reason: when $m+n = d_{T_W}(\alpha,\beta)$ is very large, $\stat_{m+n}(\ol{u_1} \dots \ol{u_p} \ \ol{u_{p+1}} \dots \ol{u_{p+j}})$ and $\stat_{m+n}(\ol{u_1} \dots \ol{u_p} \ \ol{u_{p+1}'} \dots \ol{u_{p+j}'})$ contain vast amounts of information about our sentences and so \ref{condition.virgo3} is more likely to hold. 

\begin{proof}[Proof of \Cref{thm.virgo}]
Now, 
\begin{itemize}
\item write $\scall = \{\stat_c^k : 1 \leq k \leq K\}$;
    \item let $\tau^k \geq 1$ be the constant associated to the linear statistic $\stat_c^k$;
    \item let $B^k$ denote the finite alphabet associated to the linear statistic $\stat^k_c$;
    \item let $\stat^k_\infty$ denote the infinite statistic associated to $\stat^k_c$;
    \item let $\tau = \max_k \tau^k$;
    \item let $\omega$ be the smallest natural number which is at least $\tau / \epsilon$;
    \item let $\star$ and $\bigstar$ be two special letters.
\end{itemize} 
Let $u = b_1 b_2 b_3 \dots$ be a word on some alphabet, and let $l = \length(u)$. For $r \in \N$, we define
\[\norm_r(u) = 
\begin{cases}
b_1 b_2 \dots b_r & r \leq l \\ 
b_1 b_2 \dots b_l (\star)^{r-l} & r > l
\end{cases}
\]
So $\norm_r(u)$ is the same as $u$ except it has been either truncated or otherwise extended by $\star$'s so that it has exactly $r$ letters (the word has been "normalised"). Recall that $\ola{u}$ is the word $u$ written backwards.

If $\ol{w_1} \dots \ol{w_i}$ is an element of $T_W$ then we define a function $\widehat{\stat_\infty^k}$ with domain $T_W$ by
\[\widehat{\stat_\infty^k}(\ol{w_1} \dots \ol{w_i}) = \ola{\norm_{\omega \length(w_i)}(\stat_\infty^k(\ol{w_1} \dots \ol{w_i}))}\]
So $\widehat{\stat_\infty^k}(\ol{w_1} \dots \ol{w_i})$ is a word on the alphabet $B^k \cup \star$ of length exactly $\omega \length(w_i)$. Let $W'$ denote the set of finite words on the finite alphabet 
\[A' = \bigstar \cup \left(A \times \prod_{k = 1}^K (B^k \cup \star) \right)\]
Then we have a map $I: T_W \rightarrow T_{W'}$ given by 
\[I(\ol{w_1}\dots\ol{w_i}) = 
\ol{
\bigstar
\begin{pmatrix}
(w_1)^\omega \\ 
\widehat{\stat_\infty^1}(\ol{w_1}) \\
\widehat{\stat_\infty^2}(\ol{w_1}) \\
\vdots \\
\widehat{\stat_\infty^K}(\ol{w_1})
\end{pmatrix}
}
\
\ol{
\bigstar
\begin{pmatrix}
(w_2)^\omega \\ 
\widehat{\stat_\infty^1}(\ol{w_1} \ \ol{w_2}) \\
\widehat{\stat_\infty^2}(\ol{w_1} \ \ol{w_2}) \\
\vdots \\
\widehat{\stat_\infty^K}(\ol{w_1} \ \ol{w_2})
\end{pmatrix}
}
\dots
\ol{
\bigstar
\begin{pmatrix}
(w_i)^\omega \\ 
\widehat{\stat_\infty^1}(\ol{w_1} \ \ol{w_2} \dots \ol{w_i}) \\
\widehat{\stat_\infty^2}(\ol{w_1} \ \ol{w_2} \dots \ol{w_i}) \\
\vdots \\
\widehat{\stat_\infty^K}(\ol{w_1} \ \ol{w_2} \dots \ol{w_i})
\end{pmatrix}
}
\]
It can be checked that $I$ is an isometric embedding: evidently it is height-preserving and order-preserving, and it is injective since if $I(\ol{w_1} \dots \ol{w_i}) = I(\ol{w_1'} \dots \ol{w_i'})$ then $(w_l)^\omega = (w_l')^\omega$ for all $1 \leq l \leq i$ and so $w_l = w_l'$ for all $1 \leq l \leq i$. 

Let
\[U = \frac{12 \tau J}{1 - \delta} + \omega N + 1 \quad \textrm{and} \quad V = \frac{12(\tau + \epsilon)J}{1 - \delta} + \omega N + 1\]
and let $\kappa$ be the smallest natural number that is larger than $\frac{16U}{1 - \delta}$ and $\frac{64J\tau}{1-\delta}$ and $\frac{16V}{1 - \delta}$ and $\frac{64J(\tau+\epsilon)}{1-\delta}$. Let $\ad_\kappa: T_{W'} \rightarrow T_{\Omega_\kappa}$ be Alice's Diary for the value of $\kappa$ given above. Let $D = \ad_\kappa \circ I$. The map $D: T_W \rightarrow T_{\Omega_\kappa}$ is a diary since it is height-preserving and order-preserving. Suppose $\alpha, \beta \in T_W$ satisfy $\virgo(\scall,\delta,J,N,\epsilon)$, and suppose, as ever, that they have the form given by \Cref{notation.sentence}. We can assume that 
\begin{equation} \label{eq.mnub}
m \leq 2 \min(m,n) \quad \textrm{and} \quad n \leq 2 \min(m,n) 
\end{equation}
otherwise, we would have $\absval{m - n} \geq (m+n)/3$ and so $d(D\alpha, D\beta) \geq (m+n)/3$ and we would be done by choosing $M(\delta,J) \geq 3$. Let $1 \leq j \leq \delta \min(m,n) + J$ and $\stat^k_c \in \scall$ be those given by condition $\virgo(\scall,\delta,J,N,\epsilon)$. We have three cases: $\length(u_{p+j}) \geq \epsilon (m+n)$ or $\length(u_{p+j}') \geq \epsilon (m+n)$ or $u_{p+j} \neq u_{p+j}'$. \\

\noindent \textbf{Case 1:} We suppose first that $\length(u_{p+j}) \geq \epsilon (m+n)$. 

\begin{claim*}
Either (1) or (2) below occurs. 
\begin{enumerate}
    \item There exist letters 
    \[ b \in \widehat{\textup{stat}_\infty^k}(\ol{u_1} \dots \ol{u_p} \dots \ol{u_{p+j}}) \]
    and
    \[b' \in \widehat{\textup{stat}_\infty^k}(\ol{u_1} \dots \ol{u_p} \dots \ol{u_{p+j}'}) \]
    which are distinct and at the same distance $d \leq \tau(m+n)$ from the \textit{end} of their words; \\
    \item $\length(\widehat{\textup{stat}_\infty^k}(\ol{u_1} \dots \ol{u_p} \dots \ol{u_{p+j}})) \geq \tau(m+n)$ but $\length(\widehat{\textup{stat}_\infty^k}(\ol{u_1} \dots \ol{u_p} \dots \ol{u_{p+j}'})) < \tau(m+n)$.
\end{enumerate}
\end{claim*}

\begin{proof}[Proof of claim]
First, note that $\length(\widehat{\textup{stat}_\infty^k}(\ol{u_1} \dots \ol{u_p} \dots \ol{u_{p+j}})) \geq \tau (m+n)$ since $\length((u_{p+j})^\omega) = \omega \length(u_{p+j}) \geq \omega \epsilon (m+n) \geq \tau (m+n)$.

Now suppose that (2) doesen't occur, i.e. $\length(\widehat{\textup{stat}_\infty^k}(\ol{u_1} \dots \ol{u_p} \dots \ol{u_{p+j}'})) \geq \tau(m+n)$. So $\tau(m+n) \leq \omega \length(u_{p+j}')$. For a contradiction, suppose that (1) also doesen't occur. Then the final $\tau(m+n)$ letters of $\widehat{\textup{stat}_\infty^k}(\ol{u_1} \dots \ol{u_p} \dots \ol{u_{p+j}})$ and $\widehat{\textup{stat}_\infty^k}(\ol{u_1} \dots \ol{u_p} \dots \ol{u_{p+j}'})$ must be the same. It follows that the first $\tau(m+n)$ letters of $\norm_{\omega \length(u_{p+j})}(\stat_\infty^k(\ol{u_1} \dots \ol{u_{p+j}}))$ and $\norm_{\omega \length(u_{p+j}')}(\stat_\infty^k(\ol{u_1} \dots \ol{u_{p+j}'}))$ must be the same. But $\norm_{\omega \length(u_{p+j})}(\stat_\infty^k(\ol{u_1} \dots \ol{u_{p+j}}))$ contains $\stat_{m+n}^k(\ol{u_1} \dots \ol{u_{p+j}}))$ as an initial subword since $\tau(m+n) \leq \omega \length(u_{p+j})$. Similarly, $\norm_{\omega \length(u_{p+j}')}(\stat_\infty^k(\ol{u_1} \dots \ol{u_{p+j}'}))$ contains $\stat_{m+n}^k(\ol{u_1} \dots \ol{u_{p+j}'}))$ as an initial subword since $\tau(m+n) \leq \omega \length(u_{p+j}')$. We have a contradiction since $\stat_{m+n}^k(\ol{u_1} \dots \ol{u_{p+j}})) \neq \stat_{m+n}^k(\ol{u_1} \dots \ol{u_{p+j}'}))$. 
\end{proof}

It follows from the claim that there exist letters
\[
b
\in
\bigstar
\begin{pmatrix}
(u_{p+j})^\omega \\ 
\widehat{\stat_\infty^1}(\ol{u_1} \dots \ol{u_p} \dots \ol{u_{p+j}}) \\
\widehat{\stat_\infty^2}(\ol{u_1} \dots \ol{u_p} \dots \ol{u_{p+j}}) \\
\vdots \\
\widehat{\stat_\infty^K}(\ol{u_1} \dots \ol{u_p} \dots \ol{u_{p+j}})
\end{pmatrix}
\quad \textrm{and} \quad 
b'
\in
\bigstar
\begin{pmatrix}
(u_{p+j}')^\omega \\ 
\widehat{\stat_\infty^1}(\ol{u_1} \dots \ol{u_p} \dots \ol{u_{p+j}'}) \\
\widehat{\stat_\infty^2}(\ol{u_1} \dots \ol{u_p} \dots \ol{u_{p+j}'}) \\
\vdots \\
\widehat{\stat_\infty^K}(\ol{u_1} \dots \ol{u_p} \dots \ol{u_{p+j}'})
\end{pmatrix}
\]
which are distinct and both at distance $d \leq \tau (m+n)$ from the end of their words (in case (2) you choose $b'$ to be the $\bigstar$ at the beginning of the word). Now, suppose that the sentence $\ol{u_{p+j+1}}\dots\ol{u_{p+m}}$ has exactly $l$ letters. This implies that the tail-sentence of
\[
b
\in
I(\alpha)
\]
has at most $\tau(m+n) + \omega l + (m-j)$ letters. So it has average word length at most
\[\frac{\tau(m+n) + \omega l + (m-j)}{m-j+1} \leq \tau\frac{m+n}{m-j} + \omega N + 1 \leq \tau\frac{3m}{m-j} + \omega N + 1\]
where we have used that $\awl(\ol{u_{p+j+1}} \dots \ol{u_{p+m}}) = l / (m-j) \leq N$ together with \eqref{eq.mnub}. We now need to divide into two subcases: either $J \leq \frac{1 - \delta}{16} \min(m,n)$ or $J > \frac{1 - \delta}{16} \min(m,n)$. \\

\noindent \textbf{Case 1a:} We suppose $J \leq \frac{1 - \delta}{16} \min(m,n)$. In particular, we have $J \leq \frac{1 - \delta}{2} \min(m,n)$ and $2 \leq \frac{1 - \delta}{8}(m+n)$. Then the average word length of $b \in I(\alpha)$ is at most
\[\tau\frac{3m}{m-j} + \omega N + 1 \leq \tau \frac{3m}{m - \delta m - J} + \omega N + 1 \leq \tau \frac{6}{1 - \delta} + \omega N + 1 \leq U\]
By similar arguments, the average word length of the tail-sentence of $b' \in I(\beta)$ is also at most $U$. In order to apply \Cref{cor.AWL} to the starred sentences $I(\alpha)$ and $I(\beta)$, we choose 
\[i = \left\lceil \frac{\min(m,n) - j}{2} \right\rceil\]
from which it follows that
\[U \frac{m-j+1}{i+1} \leq 4U \frac{m - j}{\min(m,n) - j} \leq 4U \frac{2 \min(m,n)}{\min(m,n) - \delta \min(m,n) - J} \leq \frac{16U}{1 - \delta} \leq \kappa\]
and similarly $U \frac{n-j+1}{i+1} \leq \kappa$. Hence, \Cref{cor.AWL} tells us that
\begin{align*}
d(\ad_\kappa \circ I(\alpha), \ad_\kappa \circ I(\alpha)) &\geq d_{T_{W'}}(I(\alpha),I(\beta)) - 2j - 2i \\
&= (m - j - i) + (n - j - i) \\
&\geq ((m-j)/2 - 1) + ((n-j)/2 - 1) \\
&\geq ((m-\delta m - J)/2) + ((n-\delta n - J)/2) - 2 \\
&\geq \frac{(1 - \delta)m}{4} + \frac{(1 - \delta)n}{4} - 2 \\
&\geq \frac{(1 - \delta)(m+n)}{4} - \frac{(1 - \delta)(m+n)}{8} \\
&= \frac{1 - \delta}{8} d_{T_W}(\alpha,\beta)
\end{align*}
If we choose $M(\delta,J) \geq 8 / (1 - \delta)$ then we are done in this case.

\noindent \textbf{Case 1b:} We suppose $J > \frac{1 - \delta}{16} \min(m,n)$. Then, using \eqref{eq.mnub}, we have $m,n \leq \frac{32J}{1-\delta}$ and so $m+n \leq \frac{64J}{1-\delta}$. Recall that $b \in I(\alpha)$ and $b' \in I(\beta)$ are both at distance $d \leq \tau (m+n)$ from the ends of their words. So, since $\kappa \geq \frac{64J\tau}{1-\delta} \geq \tau(m+n) \geq d$, they are distance at most $\kappa$ from the ends of their words. It follows that Alice records the events $b$ and $b'$ on page $d$ of chapter $p+j$ of the diary. It follows that chapter $p+j$ of $\ad_\kappa \circ I(\alpha)$ is distinct from chapter $p+j$ of $\ad_\kappa \circ I(\beta)$. Hence, 
\begin{align*}
d(\ad_\kappa \circ I(\alpha), \ad_\kappa \circ I(\alpha)) &\geq (p+m) - (p+j-1) + (p+n) - (p+j-1) \\
&= m+n - 2j + 2 \\
&\geq 2 \\
&= \frac{1-\delta}{32J} \frac{64J}{1-\delta} \\
&\geq \frac{1-\delta}{32J} (m+n)
\end{align*}
So if we choose $M(\delta,J) \geq \frac{32J}{1 - \delta}$ then we are done in this case too. \\

\noindent \textbf{Case 2:} We are done for symmetrical reasons in the case when $\length(u_{p+j}') \geq \epsilon (m+n)$. \\

\noindent \textbf{Case 3:} Finally, we suppose that $\length(u_{p+j}) \leq \epsilon (m+n)$, $\length(u_{p+j}') \leq \epsilon (m+n)$ and $u_{p+j} \neq u_{p+j}'$. Two possibilities can occur: (1) there exist distinct letters $a \in (u_{p+j})^\omega$ and $a' \in (u_{p+j}')^\omega$ which are at the same distance from the end of their words and this distance is at most $\epsilon \omega (m+n) \leq (\tau + \epsilon)(m+n)$ or (2) $\length((u_{p+j})^\omega) \neq \length((u_{p+j}')^\omega)$. In either case, it follows that there exist distinct letters $b \in I(\alpha)$ and $b' \in I(\beta)$ which occur on the $(p+j)$'th day and which have equal distance $d \leq (\tau + \epsilon)(m+n)$ from the ends of their words. The proof now proceeds along the same lines as Case 1 and Case 2; a similar argument, splitting into cases depending on the magnitude of $J$, will give us that $d(\ad_\kappa \circ I(\alpha), \ad_\kappa \circ I(\alpha)) \geq \frac{1 - \delta}{32J} d_{T_W}(\alpha, \beta)$.
\end{proof}

We will now define another useful criterion, pronounced \textit{Taurus}, which makes use of the fact that we always know that $u_{p+1} \neq u_{p+1}'$ when $\alpha$ and $\beta$ have the form given by \Cref{notation.sentence}. It is somewhat analogous to condition $\leo(\scalf,J)$. As before, suppose that $\scall$ is some finite collection of linear statistics on $T_W$ (with possibly different codomains) and that $J \in \N$, $N > 0$ and $\epsilon > 0$ are constants. 

\begin{definition} \label{def.taurus}
We say that $\alpha$, $\beta$ satisfy $\taurus(\scall, J, N, \epsilon)$ if there exists some $1 \leq j \leq J$ such that
    \begin{enumerate}
        \myitem{($\taurus 1$)} \label{condition.taurus1} $\awl(\ol{u_{p+j+1}} \dots \ol{u_{p+m}})$ and $\awl(\ol{u_{p+j+1}'} \dots \ol{u_{p+n}'})$ are both at most $N$;
        \myitem{($\taurus 2$)} \label{condition.taurus2}  if $1 \leq j' \leq j$ is such that we have $\length(u_{p+j''}) \leq \epsilon (m+n)$ and $\length(u_{p+j''}') \leq \epsilon (m+n)$ for all $j' < j'' \leq j$  then there exists $\stat_c \in \scall$ such that $\stat_{m+n}(\ol{u_1} \dots \ol{u_p} \ \ol{u_{p+1}} \dots \ol{u_{p+j'}}) \neq \stat_{m+n}(\ol{u_1} \dots \ol{u_p} \ \ol{u_{p+1}'} \dots \ol{u_{p+j'}'})$.
    \end{enumerate}
\end{definition}

It is worth noting that \ref{condition.taurus2} implies that there exists a linear statistic $\stat_c \in \scall$ such that $\stat_{m+n}(\ol{u_1} \dots \ol{u_p} \ \ol{u_{p+1}} \dots \ol{u_{p+j}}) \neq \stat_{m+n}(\ol{u_1} \dots \ol{u_p} \ \ol{u_{p+1}'} \dots \ol{u_{p+j}'})$.

\begin{theorem} \label{thm.taurus}
There exists a finite set $\Omega = \Omega(\scall,J,N,\epsilon)$, a diary $D = D(\scall,J,N,\epsilon)$ which is a function $D: T_W \rightarrow T_\Omega$, and a constant $M = M(J) \geq 1$, such that 
\[\frac{1}{M} d_{T_W}(\alpha,\beta) \leq d_{T_{\Omega}}(D\alpha,D\beta) \leq d_{T_W}(\alpha,\beta)\]
for all $\alpha,\beta$ satisfying $\taurus(\scall,J,N,\epsilon)$.
\end{theorem}

\begin{proof}
Suppose that $\alpha,\beta$ satisfy $\taurus(\scall, J, N, \epsilon)$. We let the map $D: T_W \rightarrow T_{\Omega}$ be the diary $D(\scall,0,J,N + 6J^2\epsilon,\epsilon)$ given by \Cref{thm.virgo}.

Note that if $m < (m+n) / 3$ then $n > 2 (m+n) / 3$ and so $\absval{m - n} > (m+n) / 3$, which implies $d_{T_\Omega}(D\alpha,D\beta) \geq (m+n) / 3$. So we are done when $m < (m+n) / 3$ by choosing $M = M(J) \geq 3$. Similarly we are done if $n < (m+n) /3$. So we can assume that $m,n \geq (m+n) / 3$. 

\begin{claim*}
$\alpha$ and $\beta$ satisfy $\virgo(\scall,0,J,N+6J^2\epsilon,\epsilon)$
\end{claim*}

\begin{proof}
Since $\alpha, \beta$ satisfy $\taurus(\scall, J, N, \epsilon)$, we know that there exists $1 \leq j \leq J$ such that \ref{condition.taurus1} and \ref{condition.taurus2} hold. In particular, $\awl(\ol{u_{p+j+1}}\dots\ol{u_{p+m}}) \leq N \leq N + 6J^2\epsilon$, $\awl(\ol{u_{p+j+1}'}\dots\ol{u_{p+n}'}) \leq N \leq N + 6J^2\epsilon$ and there exists $\stat_c \in \scal$ such that $\stat_{m+n}(\ol{u_1} \dots \ol{u_p} \dots \ol{u_{p+j}}) \neq \stat_{m+n}(\ol{u_1} \dots \ol{u_p} \dots \ol{u_{p+j}'})$. Hence, if $\length(u_{p+j}) \geq \epsilon (m+n)$ or $\length(u_{p+j}) \geq \epsilon (m+n)$ then $\alpha, \beta$ satisfy $\virgo(\scall, 0, J, N + 6J^2\epsilon,\epsilon)$ and the claim is proved. So we can assume both lengths are at most $\epsilon (m+n)$. But this implies that 
\begin{align*}
\awl(\ol{u_{p+j}} \dots \ol{u_{p+m}}) &= \frac{ \length(u_{p+j+1} \dots u_{p+m}) + \length(u_{p+j})}{m-j+1} \\
&\leq \awl(\ol{u_{p+j+1}}\dots\ol{u_{p+m}}) + \frac{\epsilon(m+n)}{m-j+1} \\
&\leq N + \frac{3\epsilon m}{m-J+1} \\
&\leq 
\begin{cases}
N + 6\epsilon & \textrm{when $m \geq 2J$} \\
N + 6J\epsilon & \textrm{when $m \leq 2J$}
\end{cases} \\
&\leq N + 6J^2 \epsilon
\end{align*}
and similarly $\awl(\ol{u_{p+j}'} \dots \ol{u_{p+n}'}) \leq N + 6J^2$. It also implies, since $\alpha$ and $\beta$ satisfy $\taurus(\scall, J, N, \epsilon)$, that there exists $\stat_c \in \scall$ such that $\stat_{m+n}(\ol{u_1} \dots \ol{u_p} \dots \ol{u_{p+j-1}}) \neq \stat_{m+n}(\ol{u_1} \dots \ol{u_p} \dots \ol{u_{p+j-1}'})$. So if $\length(u_{p+j-1}) \geq \epsilon (m+n)$ or $\length(u_{p+j-1}') \geq \epsilon (m+n)$ then the claim is proved. So we can assume that they are both at most $\epsilon (m+n)$. Continuing in this way, we can assume that $\length(u_{p+2})$, $\length(u_{p+3})$, \dots, $\length(u_{p+j})$ and $\length(u_{p+2}')$, $\length(u_{p+3}')$, \dots, $\length(u_{p+j}')$ are all at most $\epsilon (m+n)$. It follows that $\awl(\ol{u_{p+2}} \dots \ol{u_{p+m}})$ and $\awl(\ol{u_{p+2}'} \dots \ol{u_{p+n}'})$ are both at most $N + 6J^2\epsilon$. Since $u_{p+1} \neq u_{p+1}'$, $\alpha$ and $\beta$ satisfy $\virgo(\scall, 0, J, N + 6J^2\epsilon,\epsilon)$ and the claim is proved.
\end{proof}

It follows that if $M = M(0,J)$ is the constant given by \Cref{thm.virgo} then we have
\[\frac{1}{M} d_{T_W}(\alpha,\beta) \leq d_{T_{\Omega}}(D\alpha,D\beta) \leq d_{T_W}(\alpha,\beta)\]
and we are done.
\end{proof}

Recall the definition of $\leo(\scalf,J)$ from \Cref{def.leo}. The following corollary combines conditions $\leo(\scalf,J)$ and conditions $\taurus(\scall,J,N,\epsilon)$ so that we can use them simultaneously. 

\begin{corollary} \label{cor.leotaurus}
Let $\scalf$ be a finite collection of finite statistics, let $\scall$ be a finite collection of linear statistics, let $\jf \in \N$, let $\jl \in \N$, let $N \geq 1$, and finally let $\epsilon > 0$. 

Then there exists a finite set $\Omega = \Omega(\scalf,\scall, \jl, N,\epsilon)$, a diary $D = \\ D(\scalf,\scall,\jl, N,\epsilon)$ which is a function $D: T_W \rightarrow T_\Omega$, and a constant $M = M(\jf, \jl)$, such that 
\[\frac{1}{M} d_{T_W}(\alpha,\beta) \leq d_{T_{\Omega}}(D\alpha,D\beta) \leq d_{T_W}(\alpha,\beta)\]
for all $\alpha,\beta$ satisfying $\leo(\scalf, \jf)$ or $\taurus(\scall, \jl, N, \epsilon)$.
\end{corollary} 

\begin{proof}
Let $\df: T_W \rightarrow T_{\omegaf}$ be the diary $D(\scalf)$ given by \Cref{thm.aries}, and let $\mf = \mf(0,\jf)$ be the associated constant. Let $\dl: T_W \rightarrow T_{\omegal}$ be the diary $D(\scall,\jl,N,\epsilon)$ given by \Cref{thm.taurus}, and let $\ml = \ml(\jl)$ be the associated constant. Let $\Omega = \omegaf \times \omegal$; we think of $\Omega$ as being a finite alphabet. Define $M = \max(\mf,\ml)$. Consideration of the diary $D: T_W \rightarrow T_{\Omega}$ defined by
\[
D(\ol{w_1} \dots \ol{w_i}) = 
\begin{pmatrix}
\df(\ol{w_1}) \\ 
\dl(\ol{w_1}) \\ 
\end{pmatrix}
\begin{pmatrix}
\df(\ol{w_1} \ \ol{w_2}) \\ 
\dl(\ol{w_1} \ \ol{w_2}) \\ 
\end{pmatrix}
\dots
\begin{pmatrix}
\df(\ol{w_1} \ \ol{w_2} \dots \ol{w_i}) \\ 
\dl(\ol{w_1} \ \ol{w_2} \dots \ol{w_i}) \\ 
\end{pmatrix}
\]
will convince the reader of the corollary.
\end{proof}

\begin{remark} \label{remark.criteria}
It should hopefully be clear that the above argument can be repeated in order to combine any of the criteria $(\aries)$, $(\leo)$, $(\virgo)$, $(\taurus)$. 
\end{remark}

\Cref{cor.leotaurus} is used repeatedly in the proof of \Cref{thm.main}. \Cref{ex.diariesms} gives an idea as to how \Cref{cor.leotaurus} is used in the proof of \Cref{thm.main}. First, we need to thoughtfully choose the data $\scalf,\scall,\jl, N, \epsilon$ at the start of the proof in order to define the map $D(\scalf,\scall,\jl, N, \epsilon)$ given by \Cref{cor.leotaurus} which is a function $D: T_W \rightarrow T_\Omega$. We then analyse all the possible pairs $\alpha_\q$, $\beta_\q$ to ensure that they satisfy $\leo(\scalf, \jf)$ or $\taurus(\scall, \jl, N, \epsilon)$. You might think of this as \textit{finding a finite or linear (with respect to $m+n = d_{T_W}(\alpha_\q, \beta_\q)$) fact which distinguishes $\alpha_\q$ from $\beta_\q$}. Once this is done, we have proved that there exists a quasiisometric embedding $X \rightarrow \prod_{q=1}^Q T_\Omega$.

If a reader would like to the see the criteria $(\leo)$ and $(\virgo)$ in action in a relatively simple setting then they should read \Cref{appchap.3} in which these ideas are applied in order to quasiisometrically embed the hyperbolic plane into a product of binary trees. The theorems in this current section and also those in the next section (which is short) suffice for the understanding of this example. 

\section{Standard paths in relatively hyperbolic groups with virtually abelian peripherals} \label{sec.standardpaths}

In this section, we will use the description of relatively hyperbolic groups in terms of projection complexes, as described in \Cref{subsec.relhyppc}, to define a collection of canonical quasigeodesic paths between two group elements $x,z \in G$ when $G$ is relatively hyperbolic with respect to virtually abelian peripheral subgroups. The reader should note that this section, and indeed the rest of this paper, relies heavily on the theory of projection complexes and the tree of metric spaces outlined in \cite{PART1}.

So suppose $G$ is relatively hyperbolic with respect to virtually abelian peripheral subgroups $H_1, H_2, \dots, H_T$. Let $R_t$ denote the rank of a finite index abelian subgroup of $H_t$, and define $R = \max(R_1, R_2, \dots, R_T)$. Let $A_t \cong \Z^{R_t}$ be a finite index subgroup of $H_t$. A virtually abelian group always contains a normal, finite index abelian subgroup. Therefore without loss of generality we may assume that $A_t$ is normal in $H_t$. Suppose $G$ is generated by a finite set $S$. We may assume that $S_{H_t} := S \cap H_t$ generates $H_t$ and that $S_{A_t} := S_{H_t} \cap A_t$ generates $A_t$. We write $S_{H_t} = S_{A_t} \cup \{b_{t,1}, b_{t,2}, \dots, b_{t,R'_t}\}$ where $b_{t,r} \not \in A_t$ for $1 \leq r \leq R'_t$. Further, we may assume the following: that $S_{A_t} = \{a_{t,1}, a_{t,2}, \dots, a_{t,R_t}\}$ and that identifying 
\[a_{t,1}^{q_1} a_{t,2}^{q_2} \dots a_{t,R_t}^{q_{R_t}} \in \Gamma(A_t, S_{A_t})\]
with
\[(q_1, q_2, \dots, q_{R_t}) \in \Gamma(\Z^{R_t},\{(1,0,\dots,0), (0,1,\dots,0), \dots, (0,0,\dots,1)\})\]
gives a graph isomorphism. In other words, we assume that the subgraph $g\Gamma(A_t, S_{A_t})$ is a \textit{grid} for all $g \in G$. Let $I_t$ denote the cardinality of the quotient group $H_t / A_t$ and write $I = \max_t I_t$.

As in \Cref{subsec.relhyppc}, let $\G = \{gH_t : g \in G, 1 \leq t \leq T\}$ and define the projections $\{\pi_Y\}$ using the nearest point projections. We then modify $\{\pi_Y\}$ by a uniformly bounded amount using \cite[Theorem 2.2]{PART1} (which is the same as \cite[Theorem 4.1]{BBFS}) so that they now satisfy axioms (P3) - (P5) and (P4') with respect to the projection constant $\theta$. Since the original projections are collections of vertices, the modified projections will be too, as described in \cite[Theorem 2.2]{PART1}. The new projection distances will be denoted by $d_Y$. We choose $K$ to be at least $4 \theta + 2I$, and large enough that \Cref{cor.embedding2} holds. In particular, since $K \geq 4\theta$, all the results in \cite[Section 2]{PART1} hold and we can define $\pcg$ and $\qtomsg$. By setting $B = H_1$, we can also define $\T = \T_K(\G,B)$ as in \cite[Definition 3.1]{PART1}. In order to define $\tomsg$ as in \cite[Definition 4.1]{PART1}, we arbitrarily choose vertices $p(X,Y) \in \pi_X(Y)$ and $p(Y,X) \in \pi_Y(X)$ for every distinct pair $X,Y \in \G$ that are connected by an edge in $\pcg$. 

\subsection{Standard non-abelian paths between cosets of $A_t$ in $H_t$} 
Let us write $H_t / A_t = \{h_{t,1} A_t, h_{t,2} A_t, \dots, h_{t,I_t} A_t\}$. 

\begin{definition} \label{def.nonabelianpath}
For $1 \leq i \leq I_t$, let $\psi_{t,i}$ be some fixed path in $\Gamma(G,S)$ of length at most $I_t$, which is a word on the generators $\{ b_{t,1}, b_{t,2}, \dots, b_{t,R'_t}\} \subset H_t$ and their inverses, and that travels from $e$ to the coset $h_{t,i} A_t$. This is possible since $H_t / A_t$ is a group of cardinality $I_t$ that is generated by the set $\{b_{t,r}A_t : 1 \leq r \leq R_t' \}$.

Given two group elements $k, k'$ in the same coset of $H_t$, say $k,k' \in gH_t$, there exists a unique $1 \leq i \leq I_t$ such that for all group elements $l \in kA_t$ we have that $l\psi_{t,i} \in k'A_t$. We refer to $\psi_{t,i}$ as \textit{the standard non-abelian path} from the coset $kA_t$ to the coset $k'A_t$.
\end{definition} 

\subsection{Standard paths in $g\Gamma(H_t,S_{H_t})$} \label{subsec.cosetsp}

\begin{definition} \label{def.bp}
Given a coset $Y \in \G$, write
\[\G_K[H_1, Y] = \{H_1 = Y_1 < Y_2 < \dots < Y_{L-1} < Y_L = Y\}\]
The \textit{basepoint} of $Y$, which we denote by $p_0(Y)$, is defined to be $p(Y,Y_{L-1})$ if $Y \neq H_1$ and defined to be the identity $e$ if $Y = H_1$.
\end{definition}

\begin{definition} \label{def.spcoset}
Suppose $k$ and $k'$ are vertices of $\C(Y) = g\Gamma(H_t,S_{H_t})$. Let $\psi_{t,i}$ be the standard non-abelian path from the coset $kA_t$ to the coset $p_0(Y)A_t$. Let $\psi_{t,j}$ be the standard non-abelian path from the coset $p_0(Y)A_t$ to the coset $k'A_t$. A \textit{standard path} from $k$ to $k'$ in $\C(Y)$ is any path which is formed by concatenating the following three paths
\begin{itemize}
    \item The standard non-abelian path $\psi_{t,i}$;
    \item A geodesic in $p_0(Y)\Gamma(A_t,S_{A_t})$ formed of the abelian generators $\{a_{t,1}, \dots, a_{t,R_t}\}$ that takes us from $k\psi_{t,i}$ to $k'\psi_{t,j}^{-1}$;
    \item The standard non-abelian path $\psi_{t,j}$.
\end{itemize}
\end{definition}
So the standard path from $k$ to $k'$ in $\C(Y)$ has the label $\psi_{t,i} \nu \psi_{t,j}$ where $\nu$ is a geodesic in the generators $\{a_{t,1}, \dots, a_{t,R_t}\}$ from $k\psi_{t,i}$ to $k'\psi_{t,j}^{-1}$. So we have forced the standard path from $k$ to $k'$ to go via the coset of $A_t$ that contains $p_0(Y)$. 

Sometimes, we would like to make choices of standard path which reflect the ordering of the abelian generators $\{a_{t,r} : 1 \leq r \leq R_t\}$. We will refer to these special choices of standard path as \textit{ordered} standard paths. See \Cref{fig.cosetsp}.

\begin{figure}
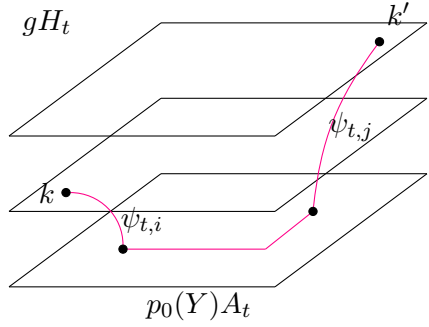

    \ctikzfig{cosetsp}
    \caption{An ordered standard path from $k$ to $k'$ in the coset $gH_t$. }
    \label{fig.cosetsp}
\end{figure}

\begin{definition} \label{def.spcosetordered}
Recall that $R = \max_t R_t$ and suppose $1 \leq r \leq R$. As before, suppose $k,k'$ are vertices of $\C(Y) = g\Gamma(H_t, S_{H_t})$. The \textit{$r$'th ordered standard path} from $k$ to $k'$ in $\C(Y)$ is the unique standard path from $k$ to $k'$ which travels the abelian generators of $A_t$ in the cyclic order 
\begin{align*}
&a_{t,r} < a_{t,r+1} < \dots < a_{t,R_t} < a_{t,1} < a_{t,2} < \dots < a_{t,r-1} & \textrm{if } r \leq R_t \\
&a_{t,1} < a_{t,2} < \dots < a_{t,R_t} & \textrm{if } r > R_t
\end{align*}
So the $r$'th ordered standard path from $k$ to $k'$ has a label of the form
\begin{align*}
&\psi_{t,i} a_{t,r}^{q_r}a_{t,r+1}^{q_{r+1}} \dots a_{t,R_t}^{q_{R_t}} a_{t,1}^{q_1} a_{t,2}^{q_2} \dots a_{t,r-1}^{q_{r-1}} \psi_{t,j} & \textrm{if } r \leq R_t \\
&\psi_{t,i} a_{t,1}^{q_1}a_{t,2}^{q_{2}} \dots a_{t,R_t}^{q_{R_t}} \psi_{t,j} & \textrm{if } r > R_t
\end{align*}
\end{definition}

\begin{proposition} \label{prop.cosetsp}
There exist $\lambdacosetsp \geq 1$ and $\mucosetsp \geq 0$ such that for all $1 \leq t \leq T$ and $g \in G$, every standard path in $g\Gamma(H_t, S_{H_t})$ is a $(\lambdacosetsp,\mucosetsp)$-quasigeodesic with respect to the metric on $g\Gamma(H_t, S_{H_t})$.  
\end{proposition}

\begin{proof}
This follows from the fact that $A_t \hookrightarrow H_t$ is a quasiisometry (since $A_t$ has finite index in $H_t$) and the fact that $\length(\psi_{t,i}) \leq I$ for all $t$ and $i$. 
\end{proof}

\begin{lemma} \label{lem.existsatrgenerator}
Suppose $Y,Y',Y''$ are distinct cosets in $\G$ and that $\{Y,Y'\}$ and $\{Y',Y''\}$ are edges in $\pcg$. Suppose also that $d_{Y'}(Y,Y'') > K$. Then a standard path from $p(Y',Y)$ to $p(Y',Y'')$ contains at least one abelian generator $a_{t,r}^{\pm 1}$. 
\end{lemma}

\begin{proof}
Let $\zeta = \psi_{t,i} \nu \psi_{t,j}$ denote a standard path from $p(Y',Y)$ to $p(Y',Y'')$. We have that 
\begin{align*}
\length(\zeta) &\geq d(p(Y',Y), p(Y',Y'')) \\
&\geq \diam(\pi_{Y'}(Y) \cup \pi_{Y'}(Y'')) - 2\theta \\
&= d_{Y'}(Y,Y'') - 2\theta \\
&> K - 2\theta \\
&> 2I
\end{align*}
where we have used $K \geq 4\theta + 2I$ in the final inequality. Since $\length(\psi_{t,i}) \leq I$ and $\length(\psi_{t,j}) \leq I$ we deduce that $\length(\nu) \geq 1$. 
\end{proof}

\subsection{Quasigeodesics between cosets of $H_t$} \label{subsec.transversegeo}

Let $\lambdaqg \geq 1$ and $\muqg \geq 0$ be constants of our choosing.

\begin{definition} \label{def.transverseqgeos}
Given $X,Y \in \G$ that are connected by an edge in $\pcg$, let $\gamma(X,Y)$ be a fixed choice of $(\lambdaqg,\muqg)$-quasigeodesic in $\Gamma(G,S)$ from $p(X,Y) \in G$ to $p(Y,X) \in G$. We sometimes refer to $\gamma(X,Y)$ as a \textit{transverse quasigeodesic}. 
\end{definition}

It does not exactly matter what $\lambdaqg$ and $\muqg$ are as long as they are uniform over all the choices of quasigeodesic. For example, you could always choose the transverse quasigeodesics to simply be geodesics from $p(X,Y) \in G$ to $p(Y,X) \in G$. 

\subsection{Standard paths in $\Gamma(G,S)$}

Let $x,z \in G$ and write $X = xH_1$ and $Z = zH_1$. Let $L = L[X,Z]$ and suppose 
\[\G_K[X,Z] = \{X = Y_1 < Y_2 < \dots < Y_{L - 1} < Y_{L} = Z\}\]

\begin{definition} \label{def.sp}
A \textit{standard path} from $x$ to $z$ is formed by concatenating standard paths within cosets $Y_l$ (as defined in \Cref{def.spcoset}) together with the transverse quasigeodesics $\gamma(Y_l,Y_{l+1})$ between these cosets (as defined in \Cref{def.transverseqgeos}). More formally, a standard path $\zeta$ from $x$ to $z$ is defined to be any concatenation
\[ \zeta = \zeta_1 \ \gamma(Y_1, Y_2) \ \zeta_2 \ \gamma(Y_2, Y_3) \ \zeta_3 \ \dots \ \zeta_{L-1} \ \gamma(Y_{L-1}, Y_L) \ \zeta_L\]
where
\begin{itemize}
    \item $\zeta_1$ is a standard path from $x \in Y_1$ to $p(Y_1, Y_2) \in Y_1$;
    \item for $2 \leq l \leq L-1$, $\zeta_{l}$ is a standard path from $p(Y_l, Y_{l-1})$ to $p(Y_l,Y_{l+1})$;
    \item $\zeta_L$ is a standard path from $p(Y_L, Y_{L-1})$ to $z$.
\end{itemize}

Given $1 \leq r \leq R$, the \textit{$r$'th ordered standard path from $x$ to $z$} is the unique standard path such that for $1 \leq l \leq L$, we have that $\zeta_l$ is the $r$'th ordered standard path between its initial and terminal vertices. See \Cref{fig.sp}.
\end{definition}

\begin{figure}
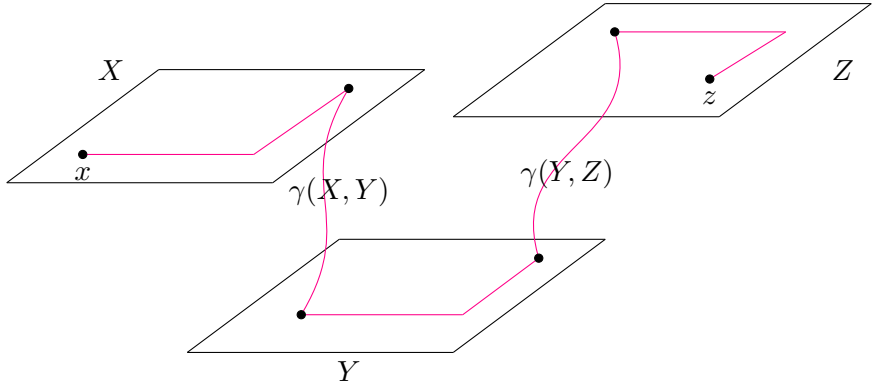

    \ctikzfig{groupsp}
    \caption{An ordered standard path from $x$ to $z$ in $\Gamma(G,S)$. It is drawn as if $\G_K[X,Z] = \{X < Y < Z\}$ and as if $A_t = H_t = \Z^2$ for all $1 \leq t \leq T$.}
    \label{fig.sp}
\end{figure}

\begin{proposition} \label{prop.sp}
For sufficiently large $K$, there exist constants $\lambdasp \geq 1$ and $\musp \geq 0$, neither of which depend on $x$ or $z$, such that a standard path from $x$ to $z$ is a $(\lambdasp, \musp)$-quasigeodesic in $G$.
\end{proposition}

\begin{proof}
We use the notation for a standard path from $x$ to $z$ that is given in \Cref{def.sp}.

Let $\hat{\gamma}$ denote a geodesic from $x$ to $z$ in the coned-off graph $\hat{G}$. By \Cref{lem.sisto2} (2), the geodesic $\hat{\gamma}$ must pass through the cosets $Y_1$, $Y_2$, \dots, $Y_{L}$ and the intersection of $\hat{\gamma}$ with one of these cosets is either an edge or possibly a single vertex (in the case of $Y_1$ and $Y_L$). For $1 \leq l \leq L$, let $p_{l}$ be the vertex where $\hat{\gamma}$ enters $Y_{l}$ (or let $p_l = x$ if $l=1$). For $1 \leq l \leq L$, let $q_{l}$ be the vertex where $\hat{\gamma}$ leaves $Y_{l}$ (or let $q_l = z$ if $l = L$). Now the vertices $V := \{p_l : 1 \leq l \leq L \} \cup \{q_l : 1 \leq l \leq L\}$ admit a natural ordering of the following form: $p_1 \leq q_1 \leq p_2 \leq q_2 \leq p_3 \leq \dots \leq p_{L} \leq z$ and given $v,v' \in V$ we say that $v < v'$ if $v \leq v'$ and $v \neq v'$. Another possible ordering on $V$, which we denote by $<<$, is to say that $v << v'$ if $\hat{\gamma}$ hits $v$ before $v'$ on its journey from $x$ to $z$. We write $v \leq \leq v'$ if $v << v'$ or $v = v'$.

\begin{claim*}
The total orders $<$ and $<<$ are the same. 
\end{claim*}

\begin{proof}
First, note that by definition $p_l \leq \leq q_l$. Second, recall that the intersection of $\hat{\gamma}$ with a coset $Y_l$ is either an edge or a vertex. It follows that there cannot exist a $v \in V$ satisfying $p_l << v << q_l$. Third, note that we clearly have $p_1 << p_l$ for all $2 \leq l \leq L$ and $q_l << q_L$ for all $1 \leq l \leq L-1$. Fourth, suppose for a contradiction that $2 \leq l < l' \leq L-1$ and yet $p_{l'} << p_l$. This implies that $\hat{\gamma}$ enters the coset $Y_{l'}$ before it enters the coset $Y_{l}$. \Cref{lem.sisto1} implies that $\hat{\gamma}$ first enters $Y_{l}$ close to $\pi_{Y_{l}}(X)$. But it also implies that $\hat{\gamma}$ enters $Y_{l}$ close to $\pi_{Y_{l}}(Y_{l'})$, since $\hat{\gamma}$ passes through $Y_{l'}$ beforehand. This implies that $\pi_{Y_l}(X)$ and $\pi_{Y_l}(Y_{l'})$ are close. But $d_{Y_{l}}(X, Y_{l'}) > K$ since $X < Y_{l} < Y_{l'}$ in $\G_K[X,Z]$, and so by choosing $K$ sufficiently large (i.e. much larger than the constant $M$ given by \Cref{lem.sisto1}) we have a contradiction. So if $l < l'$ then $p_{l} \leq \leq p_{l'}$. The claim follows from the four statements above.
\end{proof}

So $\hat{\gamma}$ moves through the cosets $Y_1, Y_2, \dots, Y_{L}$ in the natural order. \Cref{prop.sisto} tells us that a lift of $\hat{\gamma}$ in $G$ is a $(\lambda,\mu)$-quasigeodesic for some constants $\lambda \geq 1$ and $\mu \geq 0$. We denote some chosen lift by $\gamma$. The claim implies that $\length(\gamma) \geq \sum_{l = 1}^{L} d_G(p_l,q_l) + \sum_{l = 1}^{L-1} d_G(q_l,p_{l+1})$.

Further, \Cref{lem.sisto1} implies that there exists some uniform constant $M$ such that for $1 \leq l \leq L$ we have $d_G(p_{l},p(Y_{l},Y_{l - 1})) \leq M$ and $d_G(q_{l},p(Y_{l},Y_{l +1})) \leq M$. 

Since peripheral cosets are undistorted in $G$ (see \Cref{prop.ds}), and since standard paths \textit{within} cosets are quasigeodesics (see \Cref{prop.cosetsp}), there exist constants $\lambda' \geq 1$ and $\mu' \geq 0$ such that for $1 \leq l \leq L$ we have $\length(\zeta_l) \leq \lambda' d_G(p(Y_l,Y_{l-1}), p(Y_l,Y_{l+1})) + \mu'$. We may assume for simplicity that $\lambda' \geq \lambdaqg$ and $\mu' \geq \muqg$. It follows that 
\begin{align*}
&\length(\zeta) = \sum_{l = 1}^{L} \length(\zeta_l) + \sum_{l=1}^{L-1} \length(\gamma(Y_l,Y_{l+1})) \\
&\leq \lambda'\sum_{l=1}^{L} d_G(p(Y_l,Y_{l-1}), p(Y_l,Y_{l+1})) + \lambdaqg \sum_{l=1}^{L-1} d_G(p(Y_l,Y_{l+1}), p(Y_{l+1},Y_{l})) + L \mu' + (L-1)\muqg\\
&\leq \lambda' \sum_{l=1}^{L} (d_G(p_l,q_l) + 2M) + \lambdaqg \sum_{l=1}^{L-1} (d_G(q_l,p_{l+1}) + 2M) + L \mu' + (L-1)\muqg\\
&\leq \lambda'\length(\gamma) + 2M\lambda'L + 2M\lambda'(L - 1) + L \mu' + (L-1)\mu' \\
&= \lambda'\length(\gamma) + (2L - 1)(2M\lambda' + \mu')
\end{align*}
Since $\hat{\gamma}$ contains an edge in $Y_l$ for each $2 \leq l \leq L-1$, we have $\length(\hat{\gamma}) \geq L-2$. So certainly $\length(\gamma) \geq L-2$. Hence,
\[
\length(\zeta) \leq \lambda'\length(\gamma) + (2\length(\gamma) + 3)(2M\lambda' + \mu')
\]
and so we are done since $\length(\gamma) \leq \lambda d_G(x,z) + \mu$. 
\end{proof}

\section{Proof of the main theorem part one} \label{sec.proofpart1}

Let $G$ be relatively hyperbolic with respect to virtually abelian peripheral subgroups as in \Cref{thm.main}. Let us retain all the notation and assumptions described at the beginning of \Cref{sec.standardpaths}. 

However, in contrast to \Cref{sec.standardpaths} where the transverse quasigeodesics $\gamma(X,Y)$ were chosen arbitrarily, for this proof, we want to choose them so that the following two conditions hold. Let $X,Y \in \G$ and suppose $\{X,Y\}$ is an edge in $\pcg$. 
\begin{enumerate}
    \myitem{($\natural 1$)} \label{condition.natural1} 
    We choose $\gamma(X,Y)$ from $p(X,Y) \in G$ to $p(Y,X) \in G$ to be such that $\gamma(X,Y)$ is non-empty and not entirely contained in any coset subgraph $\C(W) \subset \Gamma(G,S)$.
    \myitem{($\natural 2$)} \label{condition.natural2} We choose $\gamma(X,Y)$ in such a way that we can deduce from the final $\chi$ letters of $\gamma(X,Y)$ which peripheral subgroup $H_1, H_2, \dots, H_T$ is such that $Y = gH_t$, where $\chi$ is a constant that only depends on $T$.
\end{enumerate}

It is always possible to choose $\gamma(X,Y)$ such that \ref{condition.natural1} holds (for uniform choices of $\lambdaqg \geq 1$, $\muqg \geq 0$) since we can force $\gamma(X,Y)$ to take a small diversion out of any coset subgraph $\C(W)$. It is possible to choose $\gamma(X,Y)$ such that \ref{condition.natural2} holds since there are only $T$ peripheral subgroups. One solution would be this: if $Y$ is a coset of $H_t$, then you could insist that $\gamma(X,Y)$ consists of a geodesic from $p(X,Y)$ to $p(Y,X)$ followed by the word $(a_{t,1} a_{t,1}^{-1})^t$ - in this case you could choose $\chi = 2T$. 

\begin{remark}
The necessity of \ref{condition.natural1} and \ref{condition.natural2} is an artifact of the fact that distinct peripheral cosets $X,Y$ can have non-empty intersection. If we have a single peripheral subgroup $H$ this cannot happen and we may choose the transverse quasigeodesics $\gamma(X,Y)$ to simply be geodesics. 

The purpose of \ref{condition.natural1} is so that we can recognise, just by looking at the letters of $\gamma(X,Y)$ whether it is a transverse quasigeodesic or contained in a peripheral coset. 

The purpose of \ref{condition.natural2} is so that we can recognise, just by looking at the final letters of $\gamma(X,Y)$, which peripheral subgroup $Y$ corresponds to. 
\end{remark}

\subsection{The map $F: G \rightarrow T_W \times T_W \times \dots \times T_W$}

Recall that $S$ is the finite generating set of $G$. Let $A = S \cup S^{-1}$ be our finite alphabet, let $W$ be the set of all finite (but non-empty) words on $A$ and let $T_W$ be the associated sentence-tree. Given $x \in G$, we will now define a sentence $F_r(x) \in T_W$ by partitioning the standard path from $e$ to $x$ into multiple words. 

\begin{definition} \label{def.fr}
We aim to define a function $F_r: G \rightarrow T_W$. Let $x \in G$. Write $X = xH_1$. Let $L = L[H_1, X]$ and suppose
\[\G_K[H_1,X] = \{H_1 = X_1 < X_2 < \dots < X_{L-1} < X_L = X\}\]
Let $1 \leq r \leq R$. Let $\zeta$ denote the $r$'th ordered standard path from $e$ to $x$ in $\Gamma(G,S)$. Recall from \Cref{def.sp} that we can write $\zeta$ as a concatenation of paths in the follow manner 
\[ \zeta = \zeta_1 \ \gamma(X_1, X_2) \ \zeta_2 \ \gamma(X_2, X_3) \ \zeta_3 \ \dots \ \zeta_{L-1} \ \gamma(X_{L-1}, X_L) \ \zeta_L\]

Let $1 \leq l \leq L$ and suppose that $X_l = gH_t$. We would like to define a sentence $\alpha_l \in T_W$. It has a different form depending on whether $r \leq R_t$ or $r > R_t$. Note that $\zeta_l$ has the form 
\[
\zeta_l = 
\begin{cases}
a_{t,r}^{q_r}a_{t,r+1}^{q_{r+1}} \dots a_{t,R_t}^{q_{R_t}} a_{t,1}^{q_1} a_{t,2}^{q_2} \dots a_{t,r-1}^{q_{r-1}} \psi_{t,j} & \textrm{if } r \leq R_t \\
a_{t,1}^{q_1}a_{t,2}^{q_{2}} \dots a_{t,R_t}^{q_{R_t}} \psi_{t,j} & \textrm{if } r > R_t
\end{cases}
\]
Then we define
\[
\alpha_l = 
\begin{cases}
(\ol{a_{t,r}})^{q_r} \ \ol{a_{t,r+1}^{q_{r+1}}} \dots \ol{a_{t,R_t}^{q_{R_t}}} \ \ol{a_{t,1}^{q_1}} \ \ol{a_{t,2}^{q_2}} \dots \ol{a_{t,r-1}^{q_{r-1}}} \  \ol{\psi_{t,j}} & \textrm{if } r \leq R_t \\
\ol{a_{t,1}^{q_1}} \ \ol{a_{t,2}^{q_{2}}} \dots \ol{a_{t,R_t}^{q_{R_t}}} \ \ol{\psi_{t,j}} & \textrm{if } r > R_t
\end{cases}
\]
where the meaning of $(\ol{a_{t,r}})^{q_r}$ when $q_r < 0$ is described in \Cref{not.sentencenegative}. Further, for $1 \leq l \leq L-1$, we define a sentence $\alpha_{l,l+1}$ by
\[\alpha_{l,l+1} = \ol{\gamma(X_l,X_{l+1})}\]
Finally, we define
\[F_r(x) = \alpha_{1} \alpha_{1,2} \alpha_{2} \alpha_{2,3} \dots \alpha_{L-1, L} \alpha_{L}\]
\end{definition}
Intuitively, $F_r(x)$ "crushes" everything in the $r$'th ordered standard path from $e$ to $x$ other than the directions corresponding to generators of the form $a_{t,r}$.

\begin{remark}
We don't allow empty words to appear in sentences. For example, if $\psi_{t,j}$ is empty, then it simply doesen't appear in the sentence $\alpha_l$ and we don't write something like $\ol{\emptyset}$. So the vertices of $T_W$ are in bijection with the free monoid over non-empty words on $S \cup S^{-1}$. 
\end{remark}

Recall that $R = \max_t R_t$.

\begin{definition} \label{def.f}
We define $F: G \rightarrow T_W \times T_W \times \dots \times T_W$ by $F(x) = (F_1(x),F_2(x), \dots, F_R(x))$. 
\end{definition}

\subsection{A quasiisometric embedding into a product of trees}

By \Cref{cor.relhypreg}, there exists a regular map $\phi: G \rightarrow \prod_{q=1}^Q \tb$ where $Q \leq \max(\asdim(G),R+1) + 1$. Recall as well that $\phi$ is the composition $G \hookrightarrow X(G) \rightarrow \prod_{q=1}^Q \tb$ where the second map comes from \Cref{cor.VAbinary}. 

\begin{theorem} \label{thm.fxphiqie}
There exist $\lambdaqie \geq 1$ and $\muqie \geq 0$ such that 
\[F \times \phi: G \rightarrow T_W \times T_W \times \dots \times T_W \times \prod_{q=1}^Q \tb\]
is a $(\lambdaqie,\muqie)$-quasiisometric embedding. 
\end{theorem}

The intuition for \Cref{thm.fxphiqie} comes from \Cref{thm.relhypdistanceformula} which tells you that the metric on a relatively hyperbolic group can be split into its transverse and peripheral parts. If $x,z \in G$ then $d(\phi(x),\phi(z))$ captures the \textit{transverse} distance in $G$ between $x$ and $z$ while $d_{T_W}(F_r(x),F_r(z))$ captures the $r$'th factor of the \textit{peripheral} distance in $G$ between $x$ and $z$.

\begin{proof}[Proof of \Cref{thm.fxphiqie}]

Let $x, z \in G$. By \Cref{cor.embedding2}, \cite[Theorem 4.3]{PART1} and \Cref{cor.X(G)intobinarytrees}, we have
\[d_G(x,z) \approx d_{\tomsg}(\iota_1(x),\iota_1(z)) + d_{\prod_{q=1}^Q \tb}(\phi(x),\phi(z))\]
It follows that, in order to prove \Cref{thm.fxphiqie}, we simply need to show that 
\[d_{\tomsg}(\iota_1(x), \iota_1(z)) \approx d(F(x),F(z))\]

First, we develop some notation. Write $X := xH_1$, $Z := zH_1$, $L_X := L[H_1,X]$, $L_Z := L[H_1,Z]$, $\G_K[H_1, X] = \{H_1 = X_1 < X_2 < \dots < X_{L_X} = X \}$, $\G_K[H_1, Z] = \{H_1 = Z_1 < Z_2 < \dots < Z_{L_Z} = Z \}$ and $l_\dagger = \max \{l : X_l = Z_l\}$. So $l_\dagger$ is the index at which the paths $\G_K[H_1, X]$ and $\G_K[H_1, Z]$ diverge in $\pcg$. Further, let $X = Y_1 \rightarrow Y_2 \rightarrow \dots \rightarrow Y_{L-1} \rightarrow Y_L = Z$ denote the unique geodesic from $X$ to $Z$ in $\T = \T_K(\G, H_1)$. Suppose $1 \leq l_* \leq L$ is the index such that $Y_{l_*} = X_{l_\dagger} = Z_{l_\dagger}$. See \Cref{fig.xyz}.

\begin{figure}
    \centering
    \[\begin{tikzcd}[cramped]
	{X = X_{L_X} = Y_1} && {Z = Z_{L_Z} = Y_L} \\
	{X_{L_X-1} = Y_2} && {Z_{L_Z-1} = Y_{L-1}} \\
	\vdots && \vdots \\
	{X_{l_\dagger + 1} = Y_{l_* - 1}} && {Z_{l_\dagger + 1} = Y_{l_* + 1}} \\
	& {X_{l_\dagger} = Y_{l_*} = Z_{l_\dagger}} \\
	& {X_{l_\dagger - 1} = Z_{l_\dagger - 1}} \\
	& \vdots \\
	& {H_1 = X_1 = Z_1}
	\arrow[from=5-2, to=4-3]
	\arrow[from=4-3, to=3-3]
	\arrow[from=5-2, to=4-1]
	\arrow[from=4-1, to=3-1]
	\arrow[from=3-1, to=2-1]
	\arrow[from=2-1, to=1-1]
	\arrow[from=3-3, to=2-3]
	\arrow[from=2-3, to=1-3]
	\arrow[from=6-2, to=5-2]
	\arrow[from=7-2, to=6-2]
	\arrow[from=8-2, to=7-2]
    \end{tikzcd}\]
    \caption{The paths $\G_K[H_1,X]$ and $\G_K[H_1,Z]$.}
    \label{fig.xyz}
\end{figure}
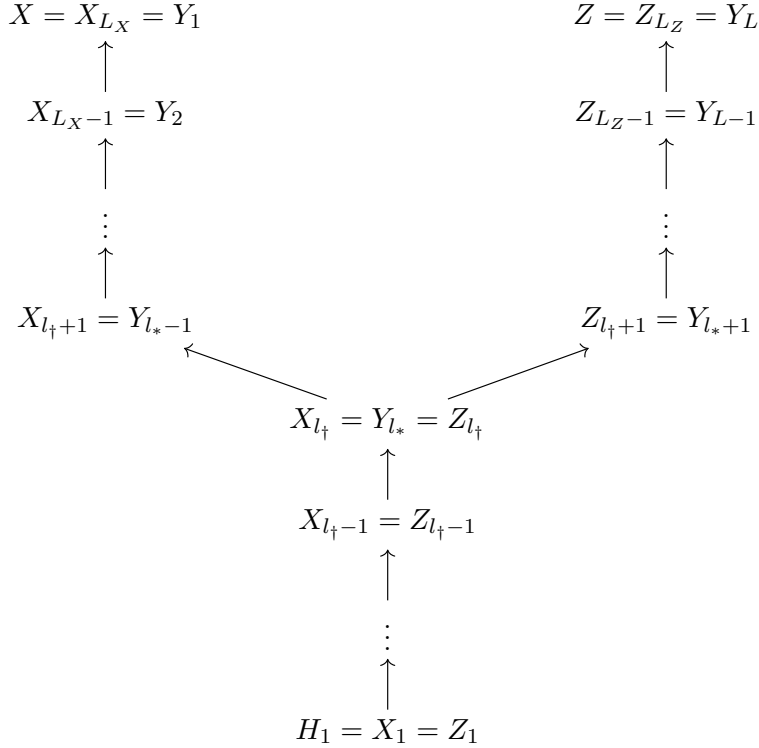

Let $\zeta$ be the composition of the following paths in $\tomsg$: 
\begin{align*}
&\textrm{a standard path from $x \in \C(Y_1)$ to $p(Y_1, Y_2) \in \C(Y_1)$ (in the sense of \Cref{def.spcoset})} \\ 
&\textrm{the transverse edge between $p(Y_1, Y_2)$ and $p(Y_2, Y_1)$} \\ 
&\textrm{a standard path from $p(Y_2, Y_1) \in \C(Y_2)$ to $p(Y_2, Y_3) \in \C(Y_2)$ (in the sense of \Cref{def.spcoset})} \\ 
&\textrm{the transverse edge between $p(Y_2, Y_3)$ and $p(Y_3, Y_2)$} \\ 
&\textrm{a standard path from $p(Y_3, Y_2) \in \C(Y_3)$ to $p(Y_3, Y_4) \in \C(Y_3)$ (in the sense of \Cref{def.spcoset})} \\ 
&\textrm{the transverse edge between $p(Y_3, Y_4)$ and $p(Y_4, Y_3)$} \\
&\vdots \\
&\textrm{the transverse edge between $p(Y_{L-1}, Y_L)$ and $p(Y_{L}, Y_{L-1})$} \\
&\textrm{a standard path from $p(Y_{L},Y_{L-1}) \in \C(Y_L)$ to $z \in \C(Y_L)$ (in the sense of \Cref{def.spcoset})}
\end{align*}
Let $\zeta_l$ indicate the intersection of $\zeta$ with the coset $\C(Y_l)$, in other words, the standard path from $p(Y_l,Y_{l-1}) \in \C(Y_l)$ to $p(Y_l,Y_{l+1}) \in \C(Y_l)$.

\begin{proofclaim} \label{claim1}
There exists $\lambda \geq 1$ and $\mu \geq 0$, not depending on $x$ or $z$, such that $\zeta$ is a $(\lambda,\mu)$-quasigeodesic in $\tomsg$. 
\end{proofclaim}

\begin{proof}[Proof of claim 1]
Let $\gamma$ denote a geodesic from $x$ to $z$ in $\tomsg$. For $1 \leq l \leq L$, let $\gamma_l$ indicate the intersection of $\gamma$ with the coset $Y_l$. We have that 
\[d_{\tomsg}(x,z) = \length(\gamma) = \sum_{l=1}^L \length(\gamma_l) + K(L-1) \]
Now, by \Cref{prop.cosetsp}, since $\gamma_l$ and $\zeta_l$ have the same initial and terminal vertices, we have that $\length(\zeta_l) \leq \lambdacosetsp \length(\gamma_l) + \mucosetsp$. Hence,
\begin{align*}
\length(\zeta) &= \sum_{l=1}^L \length(\zeta_l) + K(L-1) \\
&\leq \lambdacosetsp \sum_{l=1}^L \length(\gamma_l) + L\mucosetsp + K(L-1) \\
&\leq \lambdacosetsp d_{\tomsg}(x,z) + L\mucosetsp
\end{align*}
If we note that $L \leq d_{\tomsg}(x,z) + 1$ then we are done. 
\end{proof}

Given a coset $Y_l$, where $1 \leq l \leq L$, let $s_r(Y_l)$ denote the amount of times that the path $\zeta$ crosses a generator of the form $a_{t,r}^{\pm 1}$ in the coset subgraph $\C(Y_l) \subset \tomsg$. Further, let us define $s_{r} = \sum_{l=1}^L s_r(Y_l)$ and $s = \sum_{r = 1}^R s_r$. 

\begin{proofclaim} \label{claim2}
There exist $\lambda' \geq 1$ and $\mu' \geq 0$, not depending on $x$ or $z$, such that 
\[ s \leq \length(\zeta) \leq \lambda' s + \mu'\]
\end{proofclaim}

\begin{proof}[Proof of claim 2]
The lower bound of the inequality follows from the definition of $s$. 

Write $\zeta_l = \psi_l \nu_l \psi_l'$ where $\nu_l$ indicates the segment of $\zeta_l$ that consists entirely of the abelian generators $a_{t,r}$ and $\psi_l$, $\psi_l'$ indicate the standard non-abelian paths between cosets. For the upper bound of the inequality, note that we have
\begin{align*}
\length(\zeta) &= \sum_{l=1}^L \length(\zeta_l) + K(L-1) \\
&= s + \sum_{l=1}^L (\length(\psi_l) + \length(\psi_l')) + K(L-1) \\
&\leq s + 2LI + K(L-1)
\end{align*}

If $1 < l < l_*$, it follows from \Cref{lem.existsatrgenerator} that the intersection of $\zeta$ with the coset $Y_l$ contains at least one abelian generator $a_{t,r}^{\pm 1}$ since $d_{Y_l}(Y_{l-1}, Y_{l+1}) > K$. For the same reason, if $l_* < l < L$ then the intersection of $\zeta$ with the coset $Y_l$ also contains at least one abelian generator $a_{t,r}^{\pm 1}$. It follows that $s \geq L - 3$. Hence,
\[\length(\zeta) \leq s + 2(s+3)I + K(s+2) = (1 + 2I + K)s + 6I + 2K\]
and we are done. 
\end{proof}

As described in \Cref{def.fr}, we can write $F_r(x)$ and $F_r(z)$ as concatenations of sentences. Let us write
\[F_r(x) = \alpha_1 \alpha_{1,2} \alpha_2 \alpha_{2,3} \dots \alpha_{L_X}\]
and 
\[F_r(z) = \beta_1 \beta_{1,2} \beta_2 \beta_{2,3} \dots \beta_{L_Z}\]

\begin{proofclaim} \label{claim3}
We have 
\begin{align*}
d_{T_W}(F_r(x),F_r(z)) &= d_{T_W}(\alpha_1 \alpha_{1,2} \dots \alpha_{l_{\dagger}}, \beta_1 \beta_{1,2} \dots \beta_{l_{\dagger}}) \\
&+ \sum_{l = l_{\dagger}+1}^{L_X} \length(\alpha_{l}) + \sum_{l = l_\dagger+1}^{L_Z} \length(\beta_{l}) + L - 1
\end{align*}

\end{proofclaim}

\begin{proof}[Proof of claim 3]
Write
\begin{align*}
&u_1 = \alpha_1 \alpha_{1,2} \dots \alpha_{l_{\dagger}} \\
&u_2 = \alpha_1 \alpha_{1,2} \dots \alpha_{l_{\dagger}} \alpha_{l_\dagger,l_\dagger + 1} \\
&u_3 = \alpha_1 \alpha_{1,2} \dots \alpha_{l_{\dagger}} \alpha_{l_\dagger,l_\dagger + 1} \alpha_{l_{\dagger+1}} \\
&\vdots \\
&u_m = \alpha_1 \alpha_{1,2} \dots \alpha_{l_{\dagger}} \alpha_{l_\dagger,l_\dagger + 1} \alpha_{l_{\dagger+1}} \dots \alpha_{L_X} = F_r(x)
\end{align*}
Similarly, we choose
\begin{align*}
&v_1 = \beta_1 \beta_{1,2} \dots \beta_{l_{\dagger}} \\
&v_2 = \beta_1 \beta_{1,2} \dots \beta_{l_{\dagger}} \beta_{l_\dagger,l_\dagger + 1} \\
&v_3 = \beta_1 \beta_{1,2} \dots \beta_{l_{\dagger}} \beta_{l_\dagger,l_\dagger + 1} \beta_{l_{\dagger+1}} \\
&\vdots \\
&v_n = \beta_1 \beta_{1,2} \dots \beta_{l_{\dagger}} \beta_{l_\dagger,l_\dagger + 1} \beta_{l_{\dagger+1}} \dots \beta_{L_Z} = F_r(z)
\end{align*}

Suppose, for a contradiction, that $v_1$ descends from $u_2$. If $v_1$ were to descend from $u_2$, then $\alpha_1 \alpha_{1,2} \dots \alpha_{l_\dagger}\alpha_{l_\dagger, l_\dagger + 1}$ would be a subsentence at the start of $\beta_1 \beta_{1,2} \dots \beta_{l_\dagger}$. Noting that $\alpha_1 \alpha_{1,2} \dots \alpha_{l_\dagger -1} \alpha_{l_\dagger - 1, l_\dagger} = \beta_1 \beta_{1,2} \dots \beta_{l_\dagger -1} \beta_{l_\dagger - 1, l_\dagger}$, it follows that $\alpha_{l_\dagger} \alpha_{l_\dagger, l_\dagger + 1}$ is a subsentence of $\beta_{l_\dagger}$. However, due to our choice of transverse quasigeodesics described in \ref{condition.natural1}, the path corresponding to $\alpha_{l_\dagger} \alpha_{l_\dagger, l_\dagger + 1}$ leaves the coset $\C(X_{l_\dagger}) = \C(Z_{l_\dagger})$. In contrast, the path corresponding to $\beta_{l_\dagger}$ remains entirely within $\C(X_{l_\dagger}) = \C(Z_{l_\dagger})$. So we have a contradiction. 

For similar reasons, $u_1$ cannot descend from $v_2$. 

We claim that the path
\[u_m \rightarrow u_{m-1} \rightarrow \dots \rightarrow u_2 \rightarrow u_1 \rightarrow v_1 \rightarrow v_2 \rightarrow \dots \rightarrow v_{n-1} \rightarrow v_n\]
is a geodesic in $T_W$. If $v_1$ does not descend from $u_1$ and $u_1$ does not descend from $v_1$ then it is clearly a geodesic. If $v_1$ descends from $u_1$ then, since $d_{T_W}(u_1,u_2) = 1$ and since $v_1$ does not descend from $u_2$, we can also deduce that it is a geodesic. Similarly, if $u_1$ descends from $v_1$ it is a geodesic. Hence,
\begin{align*}
&d_{T_W}(F_r(x),F_r(z)) = d_{T_W}(\alpha_1 \alpha_{1,2} \dots \alpha_{L_X}, \beta_1 \beta_{1,2} \dots \beta_{L_Z}) \\
&= d_{T_W}(\alpha_1 \alpha_{1,2} \dots \alpha_{l_{\dagger}}, \beta_1 \beta_{1,2} \dots \beta_{l_{\dagger}}) \\
&+ \sum_{l = l_{\dagger}+1}^{L_X} d_{T_W}(\alpha_1 \alpha_{1,2} \dots \alpha_{l-1}, \alpha_1 \alpha_{1,2} \dots \alpha_{l} \alpha_{l-1,l}) + d_{T_W}(\alpha_1 \alpha_{1,2} \dots \alpha_{l-1,l}, \alpha_1 \alpha_{1,2} \dots \alpha_{l-1,l} \alpha_{l}) \\
&+ \sum_{l = l_\dagger+1}^{L_Z} d_{T_W}(\beta_1 \beta_{1,2} \dots \beta_{l-1}, \beta_1 \beta_{1,2} \dots \beta_{l-1} \beta_{l-1,l}) + d_{T_W}(\beta_1 \beta_{1,2} \dots \beta_{l-1,l}, \beta_1 \beta_{1,2} \dots \beta_{l-1,l} \beta_{l})
\end{align*}
By noting that 
\[d_{T_W}(\alpha_1 \alpha_{1,2} \dots \alpha_{l-1}, \alpha_1 \alpha_{1,2} \dots \alpha_{l-1} \alpha_{l-1,l}) = d_{T_W}(\beta_1 \beta_{1,2} \dots \beta_{l-1}, \beta_1 \beta_{1,2} \dots \beta_{l-1} \beta_{l-1,l}) = 1\]
and
\[d_{T_W}(\alpha_1 \alpha_{1,2} \dots \alpha_{l-1,l}, \alpha_1 \alpha_{1,2} \dots \alpha_{l-1,l} \alpha_{l}) = \length(\alpha_l)\]
and
\[d_{T_W}(\beta_1 \beta_{1,2} \dots \beta_{l-1,l}, \beta_1 \beta_{1,2} \dots \beta_{l-1,l} \beta_{l}) = \length(\beta_l)\]
the claim follows. 
\end{proof}

\begin{proofclaim} \label{claim4}
If $l_\dagger < l \leq L_X$ then $s_r(X_l) \leq \length(\alpha_l) \leq s_r(X_l) + R$
and if $l_\dagger < l \leq L_Z$ then $s_r(Z_l) \leq \length(\beta_l) \leq s_r(Z_l) + R$.
\end{proofclaim}

\begin{proof}[Proof of claim 4]
Let $l_\dagger < l \leq L_X$ and suppose $X_l$ is a coset of $H_t$. We have two cases: either $r \leq R_t$ or $r > R_t$. If $r \leq R_t$, then $\alpha_{l}$ has the form 
\[(\ol{a_{t,r}})^{q_r} \ol{a_{t,r+1}^{q_{r+1}}} \dots \ol{a_{t,r-1}^{q_{r-1}}} \ol{\psi_{t,j}}\]
where 
\[(a_{t,r})^{q_r} a_{t,r+1}^{q_{r+1}} \dots a_{t,r-1}^{q_{r-1}} \psi_{t,j}\]
is the $r$'th ordered standard path from $p(X_l, X_{l-1})$ to $p(X_l, X_{l+1})$. Hence, 
\[s_r(X_l) \leq \length(\alpha_l) \leq s_r(X_l) + R_t \leq s_r(X_l) + R\]
where we have used the fact that $s_r(X_l) = \absval{q_r}$. If, $r > R_t$, then $s_r(X_l) = 0$ and $\length(\alpha_l) \leq R_t + 1 \leq R$. So, in either case, we have $s_r(X_l) \leq \length(\alpha_l) \leq s_r(X_l) + R$.

For similar reasons, for $l_\dagger < l \leq L_Z$, we have $s_r(Z_l) \leq \length(\beta_l) \leq s_r(Z_l) + R$
and we are done.
\end{proof}

\begin{proofclaim} \label{claim5}
We have $s_r(Y_{l_*}) \leq d_{T_W}(\alpha_1 \alpha_{1,2} \dots \alpha_{l_{\dagger}}, \beta_1 \beta_{1,2} \dots \beta_{l_{\dagger}}) \leq s_r(Y_{l_*}) + 2(R+1)$.
\end{proofclaim}

\begin{proof}[Proof of claim 5]
Recall that $Y_{l_*} = X_{l_\dagger} = Z_{l_\dagger}$. Suppose that $Y_{l_*}$ is a coset of $H_t$. We have two cases: $r \leq R_t$ and $r > R_t$. 

Suppose first that $r > R_t$. Then $s_r(Y_{l_*}) = 0$. Further, $\length(\alpha_{l_\dagger}) \leq R_t + 1 \leq R + 1$ and $\length(\beta_{l_\dagger}) \leq R_t + 1 \leq R + 1$. Since $\alpha_1 \alpha_{1,2} \dots \alpha_{l_{\dagger} - 1,l_\dagger} = \beta_1 \beta_{1,2} \dots \beta_{l_{\dagger} - 1,l_\dagger}$, it follows that 
\[d_{T_W}(\alpha_1 \alpha_{1,2} \dots \alpha_{l_{\dagger}}, \beta_1 \beta_{1,2} \dots \beta_{l_{\dagger}}) = d_{T_W}(\alpha_{l_\dagger},\beta_{l_\dagger}) \leq 2(R+1)\] 
So $s_r(Y_{l_*}) \leq d_{T_W}(\alpha_1 \alpha_{1,2} \dots \alpha_{l_{\dagger}}, \beta_1 \beta_{1,2} \dots \beta_{l_{\dagger}}) \leq s_r(Y_{l_*}) + 2(R+1)$.

Now suppose that $r \leq R_t$. Then $\alpha_{l_\dagger}$ has the form 
\[\alpha_{l_\dagger} = (\ol{a_{t,r}})^{q_r} \ \ol{a_{t,r+1}^{q_{r+1}}} \dots \ol{a_{t,R_t}^{q_{R_t}}} \ \ol{a_{t,1}^{q_1}} \ \ol{a_{t,2}^{q_2}} \dots \ol{a_{t,r-1}^{q_{r-1}}} \  \ol{\psi_{t,j}}\]
and $\beta_{l_\dagger}$ has the form 
\[\beta_{l_\dagger} = (\ol{a_{t,r}})^{q_r'} \ \ol{a_{t,r+1}^{q_{r+1}'}} \dots \ol{a_{t,R_t}^{q_{R_t}'}} \ \ol{a_{t,1}^{q_1'}} \ \ol{a_{t,2}^{q_2'}} \dots \ol{a_{t,r-1}^{q_{r-1}'}} \  \ol{\psi_{t,j}'}\]
It follows that
\begin{equation} \label{eq.qrineq}
\absval{q_r - q_r'} \leq d_{T_W}(\alpha_{l_\dagger}, \beta_{l_\dagger}) \leq \absval{q_r - q_r'} + 2R_t
\end{equation}

Now consider the coset $p_0(Y_{l_*})A_t \subset Y_{l_*}$. We identify $p_0(Y_{l_*}) \in p_0(Y_{l_*})A_t$ with $(0,0,\dots,0) \in \Z^{R_t}$. More precisely, we identify $p_0(Y_{l_*}) a_{t,1}^{q_1} a_{t,2}^{q_2} \dots a_{t,R_t}^{q_{R_t}} \in p_0(Y_{l_*})A_t$ with $(q_1, q_2, \dots, q_{R_t}) \in \Z^{R_t}$. Under this identification, let $\pi_r: \Z^{R_t} \rightarrow \Z$ be the projection onto the $r$'th factor. Let us also write $p := p(Y_{l_*}, Y_{l_*-1})$ and $p' := p(Y_{l_*}, Y_{l_*+1})$. The intersection of $\zeta$ with $Y_{l_*}$, which we have denoted by $\zeta_{l_*}$, is precisely a standard path from $p$ to $p'$ in the coset $Y_{l_*}$. Now $\zeta_{l_*}$ has the form $\psi_{t,j}^{-1} \nu \psi_{t,j}'$ where $\nu$ corresponds to a geodesic in $\Z^{R_t}$ from $p\psi_{t,j}^{-1}$ to $p'(\psi_{t,j}')^{-1}$ and $\psi_{t,j}$ and $\psi_{t,j'}$ are the standard non-abelian paths arising from our description of $\alpha_{l_\dagger}$ and $\beta_{l_\dagger}$. It should be clear that 
\begin{equation} \label{eq.sr}
s_r(Y_{l_*}) = d_{\Z}(\pi_r(p\psi_{t,j}^{-1}), \pi_r(p'(\psi_{t,j}')^{-1})) = \length(\pi_r(\nu))
\end{equation}

Finally, observe that $\pi_r(p\psi_{t,j}^{-1}) = q_r$ and $\pi_r(p'(\psi_{t,j}')^{-1}) = q_r'$. See \Cref{fig.claim5}. 

\begin{figure}
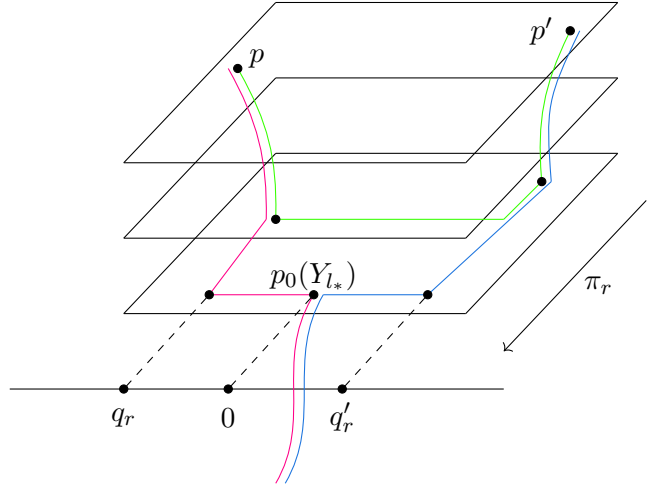

    \ctikzfig{claim5}
    \caption{The green line is the path $\zeta_{l_*}$. The red line is a section of the $r$'th ordered standard path from $e$ to $x$. Similarly, the blue line is a section of the $r$'th ordered standard path from $e$ to $z$.}
    \label{fig.claim5}
\end{figure}

So we conclude from \eqref{eq.qrineq} and \eqref{eq.sr} that
\[s_r(Y_{l_*}) \leq d_{T_W}(\alpha_{l_\dagger}, \beta_{l_\dagger}) \leq s_r(Y_{l_*}) + 2R_t \leq s_r(Y_{l_*}) + 2R\]
As before, since $\alpha_1 \alpha_{1,2} \dots \alpha_{l_{\dagger} - 1,l_\dagger} = \beta_1 \beta_{1,2} \dots \beta_{l_{\dagger} - 1,l_\dagger}$, we know that 
\[d_{T_W}(\alpha_1 \alpha_{1,2} \dots \alpha_{l_{\dagger}}, \beta_1 \beta_{1,2} \dots \beta_{l_{\dagger}}) = d_{T_W}(\alpha_{l_\dagger},\beta_{l_\dagger})\]
and so we are done. 
\end{proof}

\Cref{claim3}, \Cref{claim4} and \Cref{claim5} combine to give us the following claim. 

\begin{proofclaim} \label{claim6}
We have 
\[s_r + (L-1) \leq d_{T_W}(F_r(x),F_r(z)) \leq s_r + (R+1)(L+1)\]
\end{proofclaim}

\begin{proof}[Proof of claim 6]
We have the following chain of inequalities. 
\begin{align*}
d_{T_W}(F_r(x),F_r(z)) &= d_{T_W}(\alpha_1 \alpha_{1,2} \dots \alpha_{l_{\dagger}}, \beta_1 \beta_{1,2} \dots \beta_{l_{\dagger}}) \\
&+ \sum_{l = l_{\dagger}+1}^{L_X} \length(\alpha_{l}) + \sum_{l = l_\dagger+1}^{L_Z} \length(\beta_{l}) + L-1 \\
&\leq s_r(Y_{l_*}) + 2(R+1) + \sum_{l = l_{\dagger}+1}^{L_X} (s_r(X_l) + R) + \sum_{l = l_{\dagger}+1}^{L_Z} (s_r(Z_l) + R) + L-1  \\
&= s_r + 2(R+1) + (L-1)R + L-1 \\
&= s_r + (R+1)(L+1)
\end{align*}
Similarly,
\begin{align*}
d_{T_W}(F_r(x),F_r(z)) &= d_{T_W}(\alpha_1 \alpha_{1,2} \dots \alpha_{l_{\dagger}}, \beta_1 \beta_{1,2} \dots \beta_{l_{\dagger}}) \\
&+ \sum_{l = l_{\dagger}+1}^{L_X} \length(\alpha_{l}) + \sum_{l = l_\dagger+1}^{L_Z} \length(\beta_{l}) + L-1 \\
&\geq s_r(Y_{l_*}) + \sum_{l = l_{\dagger}+1}^{L_X} s_r(X_l) + \sum_{l = l_{\dagger}+1}^{L_Z} s_r(Z_l) + L-1  \\
&= s_r + L-1
\end{align*}
and so we have proved the claim.
\end{proof}

We will now combine \Cref{claim1}, \Cref{claim2} and \Cref{claim6} above in order to prove that 
\[d_{\tomsg}(\iota_1(x), \iota_1(z)) \approx d(F(x),F(z))\]

Recalling that $L-3 \leq s$ (by the argument given in the proof of \Cref{claim2}), we have
\begin{align*}
d(F(x),F(z)) &= \sum_{r=1}^R d_{T_W}(F_r(x),F_r(z)) \\ 
&\leq \sum_{r=1}^R (s_r + (R+1)(L+1)) &(\textrm{by \Cref{claim6}})\\ 
&= s + R(R+1)(L+1) \\
&\leq s + R(R+1)(s+4) \\
&= (R^2 + R + 1)s + 4R(R+1) \\
&\leq (R^2 + R + 1)\length(\zeta)+ 4R(R+1) &(\textrm{by \Cref{claim2}})\\
&\leq (R^2 + R + 1)(\lambda d_{\tomsg}(x,z) + \mu) + 4R(R+1) &(\textrm{by \Cref{claim1}})
\end{align*}
Similarly, 
\begin{align*}
d_{\tomsg}(x,z) &\leq \length(\zeta) &(\textrm{by \Cref{claim1}})\\
&\leq \lambda' s + \mu' &(\textrm{by \Cref{claim2}})\\
&= \lambda' \sum_{r = 1}^R s_r + \mu' \\
&\leq \lambda' \sum_{r=1}^R d_{T_W}(F_r(x),F_r(z)) + \mu' &(\textrm{by \Cref{claim6}})\\
&= \lambda' d(F(x),F(z)) + \mu'
\end{align*}
and so we have proved the theorem.
\end{proof}

\section{Proof of the main theorem part two} \label{sec.proofpart2}

A reader may find it helpful at this point to read the proof in \Cref{appchap.3} that the hyperbolic plane quasiisometrically embeds into a product of two binary trees. Several of the ideas in the proof of \Cref{thm.main} also appear in this simpler case, however it does nonetheless have its own quirks. 

We continue with our proof of \Cref{thm.main} and we use all the notation developed in \Cref{sec.standardpaths} and \Cref{sec.proofpart1}.

Let $\phi: G \rightarrow \prod_{q=1}^Q \tb$ be the regular map given by \Cref{cor.relhypreg}. Let $\Phi: G \rightarrow \prod_{q=1}^Q T_C$ be the injective and Lipschitz map given by \Cref{lem.reglem}. Let us say that $\Phi$ is $\lambdaPhi$-Lipschitz. 

\begin{notation} \label{not.wq}
A vertex of $T_C$ is naturally identified with a word on the alphabet $C$. Hence, every vertex of $\prod_{q = 1}^{Q} T_C$ naturally corresponds to a vector $(v_1, \dots, v_Q)$
where $v_q$ is a word on the alphabet $C$. Let $g \in G$. Then $\Phi(g) \in \prod_{q = 1}^{Q} T_C$ corresponds to a vector
\[
\begin{pmatrix}
v_1 \\ v_2 \\ \vdots \\ v_Q
\end{pmatrix}
\]
For $1 \leq q \leq Q$, we let $\Phi_q(g)$ denote the word $v_q \in T_C$.  
\end{notation}

\subsection{Choosing our statistics}

We would like to apply \Cref{cor.leotaurus} to the sentence-tree $T_W$ where $W$ is the set of finite words on $S \cup S^{-1}$. In order to do so, we need to make choices for $\scalf$, $\scall$, $\jf$, $\jl$, $N$ and $\epsilon$. We make the following choices.

Let $\jl = 4\lambdacosetsp K + 2\mucosetsp + 3$, $\jf = \max(\jl, R)$, $N = 12(R+1)\lambdaqie (\lambdasp^2 + \lambdasp \musp + \lambdasp K + \musp)$ and let $\epsilon = 1$. Here, $\lambdacosetsp$ and $\mucosetsp$ are constants arising from \Cref{prop.cosetsp}, $\lambdasp$ and $\musp$ are the constants arising from \Cref{prop.sp}, and $\lambdaqie$ is the constant arising from \Cref{thm.fxphiqie}. 

Let $\scalf$ be the following collection of finite statistics on $T_W$ (defined with respect to an arbitrary sentence $\ol{w_1} \ \ol{w_2} \dots \ol{w_i} \in T_W$). 
\begin{itemize}
    \item is the path $w_i$ entirely contained in one of the coset subgraphs $\C(H_1), \C(H_2), \dots, \C(H_T)$?
    \item the final $I+3K$ letters of $w_i$;
    \item $\length(w_i)$ modulo $(\lambdacosetsp K + \mucosetsp + 1)$;
    \item if $w_j$ is the last word of $\ol{w_1} \ \ol{w_2} \dots \ol{w_i}$ that is not entirely contained in one of the subgraphs $\C(H_1), \dots, \C(H_T)$, then what are the final $\chi$ letters of $w_i$?
\end{itemize}
So, in total, we have chosen four finite statistics. Note that a question like "is the path $w_i$ entirely contained in one of the coset subgraphs $\C(H_1), \C(H_2), \dots, \C(H_T)$?" is indeed a finite statistic as it has a yes or no answer. Recall that $\chi$ is the constant defined in \ref{condition.natural2}. 

Let $\omega$ be the smallest natural number that is greater than both 
\[\lambdaPhi(\jl (\lambdaqie + \muqie) + K + \lambdasp + \musp + \jl(\lambdaqie + \muqie) + K) \]
and
\[\lambdaPhi(\jl (\lambdaqie + \muqie) + K + \lambdasp + \musp + \lambdasp (\lambdasp + \musp + K) + \musp + K)\]
Let $\scall$ be the following collection of linear statistics on $T_W$, where we again let $\ol{w_1} \ \ol{w_2} \dots \ol{w_i}$ denote an arbitrary sentence in $T_W$. In the linear statistics below, note that $c$ is the \textit{variable} required in the definition of a linear statistic (see \Cref{def.lstat}) and not a constant like $R, \lambdaqie$ and $\omega$. 
\begin{itemize}
    \item the final $2(R+1) \lambdaqie \omega c$ letters of $w_{i-r+1}$ (for $1 \leq r \leq R$);
    \item the final $2(R+1) \lambdaqie \omega c$ letters of the base $10$ expansion of $\length(w_{i-r+1}) \in \N$ (for $1 \leq r \leq R$);
    \item the final $2(R+1) \lambdaqie \omega c$ letters of $\Phi_q(w_{i - j + 1} \dots w_i)$ (for $1 \leq j \leq \jl$ and $1 \leq q \leq Q$);
    \item the final $2(R+1)\lambdaqie \omega c$ letters of the base $10$ expansion of $\length(\Phi_q(w_{i - j + 1} \dots w_i)) \in \N$ (for $1 \leq j \leq \jl$ and $1 \leq q \leq Q$).
\end{itemize}
So, in total, we have chosen $2R + 2\jl Q$ linear statistics. To be clear: when we refer to $\Phi_q(w_{i - j + 1} \dots w_i)$, we mean the image of $w_{i - j + 1} w_{i-j+2} \dots w_i \in G$ under $\Phi_q: G \rightarrow T_C$.

\subsection{The associated diary}

Let $D = D(\scalf, \scall, \jl, N, \epsilon)$ be the diary given by \Cref{cor.leotaurus}. We claim that the composition
\[ G \xrightarrow{F \times \phi} T_W \times T_W \times \dots \times T_W \times \prod_{q=1}^Q \tb \xrightarrow{D \times D \times \dots \times D \times \textrm{id}} T_{\Omega} \times T_{\Omega} \times \dots \times T_\Omega \times \prod_{q=1}^Q \tb \]
is a quasiisometric embedding. Once this is shown, the proof of the main theorem (\Cref{thm.main}) is complete since $T_\Omega$ is quasiisometric to the rooted binary tree $\tb$. 

\begin{notation}
Let us refer to the composition $(D \times D \times \dots \times D \times \textrm{id}) \circ (F \times \phi)$ as $\mathcal{F}: G \rightarrow T_{\Omega} \times T_{\Omega} \times \dots \times T_\Omega \times \prod_{q=1}^Q \tb$. 
\end{notation}

Note that $\F$ is coarsely Lipschitz since $F \times \phi$ is a quasiisometric embedding and $D \times D \times \dots \times D \times \textrm{id}$ is $1$-Lipschitz. It follows that we only need to prove the lower bound of the quasiisometric inequality for $\F$. 

\subsection{Reductions}

Let $x,z \in G$. We may assume that 
\begin{equation} \label{eq.dlb}
d_G(x,z) \geq 2\lambdaqie\muqie + 12R(R+1)\lambdaqie \jl
\end{equation}
since we are only interested in the coarse geometry of $G$. Now,
\begin{itemize}
    \item define $d := d_G(x,z)$;
    \item define $d' := d((F \times \phi)(x), (F \times \phi)(z))$;
    \item define $d'' := d(\F(x),\F(z))$;
    \item let $d'_r$ be the distance between $(F \times \phi)(x)$ and $(F \times \phi)(z)$ in the $r$'th factor for $r = 1,2, \dots, R+1$;
    \item let $d''_r$ be the distance between $\F(x)$ and $\F(z)$ in the $r$'th factor for $r = 1,2, \dots, R+1$.
\end{itemize}
In the above, by the $(R+1)$'st factor, we mean the entirety of $\prod_{q=1}^Q \tb$. So $d' = d'_1 + d'_2 + \dots + d'_{R} + d'_{R+1}$ and $d'' = d''_1 + d''_2 + \dots + d''_{R} + d''_{R+1}$. Further, note that $d'_{R+1} = d''_{R+1}$. Since $d \geq 2 \lambdaqie \muqie$ by \eqref{eq.dlb}, and since $F \times \phi$ is a $(\lambdaqie,\muqie)$-quasiisometric embedding, we have $d' \geq d / 2\lambdaqie$.

\begin{claim*}
We are done if $d'_{R+1} \geq d'/(R+1)$.
\end{claim*}

\begin{proof}[Proof of claim]
If $d'_{R+1} \geq d'/(R+1)$ then 
\[d'' \geq d''_{R+1} = d'_{R+1} \geq d'/(R+1) \geq \frac{d}{2(R+1)\lambdaqie}\]
In other words, 
\[d(\F(x),\F(z)) \geq \frac{d}{2(R+1)\lambdaqie}\]
Since, as mentioned above, $\F$ is coarsely Lipschitz, we are done. 
\end{proof}

So we can assume that $d'_{R+1} \leq d'/(R+1)$. This implies that $d'_1 + d'_2 + \dots + d'_R \geq Rd'/(R+1)$. Suppose $1 \leq \r \leq R$ is such that $d'_{\r} = \max(d'_1, d'_2, \dots, d'_R)$. It follows that
\begin{equation} \label{eq.mplusnlb}
d'_{\r} \geq \frac{d'_1 + d'_2 + \dots + d'_R}{R} \geq d'/(R+1) \geq \frac{d}{2(R+1)\lambdaqie}
\end{equation}

Let $\alpha \in T_W$ be the $\r$'th factor of $(F \times \phi)(x)$ and let $\beta \in T_W$ be the $\r$'th factor of $(F \times \phi)(z)$. We can write
\begin{align*}
&\alpha = \ol{u_1} \dots \ol{u_p} \ \ol{u_{p+1}} \dots \ol{u_{p+m}} \\
&\beta = \ol{u_1} \dots \ol{u_p} \ \ol{u_{p+1}'} \dots \ol{u_{p+n}'} \\
&u_{p+1} \neq u_{p+1}'
\end{align*}
So $d_{T_W}(\alpha,\beta) = m+n = d'_{\r}$. 

\begin{claim*}
We are done if $m \leq (m+n)/3$ or $n \leq (m+n) / 3$.
\end{claim*}

\begin{proof}[Proof of claim]
If either of these hold, we have $\absval{m - n} \geq d'_{\r} / 3$ and hence 
\[d'' \geq d''_{\r} \geq \absval{m - n} \geq d'_{\r} / 3 \geq \frac{d}{6(R+1)\lambdaqie}\]
where the second inequality follows from the fact that $D$ is height-preserving.
\end{proof}

So we can assume that
\begin{equation} \label{eq.mn}
m,n \geq (m+n)/3 = d'_{\r} / 3 \geq \frac{d}{6(R+1)\lambdaqie}
\end{equation}
In particular, it follows that 
\begin{equation} \label{eq.mnlb}
m \geq \max(R,\jl) \quad \textrm{and} \quad n \geq \max(R,\jl)
\end{equation}
since $d \geq 6R(R+1)\lambdaqie$ and $d \geq 6 (R+1) \lambdaqie \jl$ by \eqref{eq.dlb}. 

\begin{claim*}
We are done if we manage to prove that $\alpha, \beta \in T_W$ satisfy either $\leo(\scalf, \jf)$ or $\taurus(\scall,\jl,N,\epsilon)$.
\end{claim*}

\begin{proof}
If $\alpha, \beta \in T_W$ satisfy either $\leo(\scalf, \jf)$ or $\taurus(\scall,\jl,N,\epsilon)$ then \Cref{cor.leotaurus} would imply that 
\[d''_{\r} =  d_{T_\Omega}(D\alpha, D\beta) \geq d_{T_W}(\alpha, \beta) / M = d'_{\r} / M\] 
where $M = M(\jf,\jl)$ is the constant given by \Cref{cor.leotaurus}. But then, using \eqref{eq.mplusnlb}, we would have 
\[d'' \geq \frac{d}{2(R+1)M\lambdaqie} \] 
and the proof of the theorem would be complete.
\end{proof}

So we have reduced the problem to proving that 
\begin{align*}
\alpha &= F_{\r}(x) = \ol{u_1} \dots \ol{u_p} \ \ol{u_{p+1}} \dots \ol{u_{p+m}} \in T_W \\ 
\beta &= F_{\r}(z) = \ol{u_1} \dots \ol{u_p} \ \ol{u_{p+1}'} \dots \ol{u_{p+n}'} \in T_W
\end{align*}
satisfy either $\leo(\scalf, \jf)$ or $\taurus(\scall,\jl,N,\epsilon)$. This will now be our only goal.

Recall the definition of $F_r: G \rightarrow T_W$ given in \Cref{def.fr}. We have several cases depending on the forms of the words $u_{p+1}$ and $u_{p+1}'$: they can be a single letter of the form $a_{t,\r}^{\pm 1}$, they can be a word of the form $a_{t,r}^q$ for some $r \neq \r$ and $q \neq 0$, they can have the form $\psi_{t,i}$ or they can correspond to a transverse quasigeodesic $\gamma(X,Y)$.

\subsection{Simple cases}

If $\length(u_{p+1}) \leq I + 3K$ or $\length(u_{p+1}') \leq I + 3K$ then since $u_{p+1} \neq u_{p+1}'$ we know that the finite statistic "the final $I + 3K$ letters of $w_i$" in $\scalf$ distinguishes $u_{p+1}$ from $u_{p+1}'$. So $\alpha$ and $\beta$ would satisfy $\leo(\scalf,\jf)$ (with a choice of $j = 1$). Thus $u_{p+1}$ and $u_{p+1}'$ can't have the forms $a_{t,\r}$ or $a_{t,\r}^{-1}$ or $\psi_{t,i}$. 

Recall \ref{condition.natural1} which implies that the transverse quasigeodesics $\gamma(X,Y)$ cannot be entirely contained in a single coset subgraph $\C(X)$. Thus if $u_{p+1}$ has the form $a_{t,r}^{q}$ (for $r \neq \r$) and $u_{p+1}'$ has the form $\gamma(X,Y)$ then the the finite statistic "is the path $w_i$ entirely contained in one of the coset subgraphs $\C(H_1), \C(H_2), \dots, \C(H_T)$?" distinguishes $u_{p+1}$ from $u_{p+1}'$. It follows that $\alpha$, $\beta$ satisfy $\leo(\scalf,\jf)$. Similarly, if $u_{p+1}'$ has the form $a_{t,r}^{q}$ ($r \neq \r$) and $u_{p+1}$ has the form $\gamma(X,Y)$ then we are done. 

Two cases remain:  
\begin{enumerate}
    \myitem{(C1)}\label{case.c1} $u_{p+1} = a_{t,r}^q$ and $u_{p+1}' = a_{t,r}^{q'}$ where $r \neq \r$;
    \myitem{(C2)}\label{case.c2} $u_{p+1} = \gamma(X,Y)$ and $u_{p+1}' = \gamma(X',Y')$ for some transverse quasigeodesics $\gamma(X,Y)$ and $\gamma(X',Y')$.
\end{enumerate}

\subsection{Some more notation}

We need some more notation. Write $X = xH_1 \in \G$ and $Z = zH_1 \in \G$. Set $L_X = L[H_1,X]$ and $L_Z = L[H_1,Z]$ and $L = L[X,Z]$. Write
\[\G_K[H_1,X] = \{H_1 = X_1 < X_2 < \dots < X_{L_X} = X\}\]
and 
\[\G_K[H_1,Z] = \{H_1 = Z_1 < Z_2 < \dots < Z_{L_Z} = Z\}\]
and
\[\G_K[X,Z] = \{X = Y_1 < Y_2 < \dots < Y_{L} = Z\}\]
Let $\zeta$ denote the $\r$'th ordered standard path from $e$ to $x$ and write $\zeta$ as the concatenation of paths 
\[\zeta = \zeta_1 \ \gamma(X_1,X_2) \ \zeta_2 \ \gamma(X_2, X_3) \ \dots \ \gamma(X_{L_X - 1}, X_{L_X}) \ \zeta_{L_X}\]
as described in \Cref{def.sp}. Similarly, Let $\zeta'$ denote the $\r$'th ordered standard path from $e$ to $z$ and write $\zeta'$ as the concatenation of paths 
\[\zeta' = \zeta_1' \ \gamma(Z_1,Z_2) \ \zeta_2' \ \gamma(Z_2, Z_3) \ \dots \ \gamma(Z_{L_Z - 1}, Z_{L_Z}) \ \zeta_{L_Z}'\]

\subsection{Degenerate cases}

\begin{proofclaim2} \label{2claim1}
If $\G_K[H_1,X]$ and $\G_K[H_1,Z]$ never diverge, equivalently, if one is contained in the other, then we are done.
\end{proofclaim2}

\begin{proof}[Proof of \Cref{2claim1}]
If, say, $\G_K[H_1,X]$ is contained in $\G_K[H_1,Z]$, then that would imply that the $\zeta$ and $\zeta'$ diverge in the coset subgraph $\C(X)$. This implies, since $x \in X$, that $u_{p+1}$ has the form $a_{t,\r}^{\pm 1}$ or  $\psi_{t,i}$ or $a_{t,r}^{q}$ (for some $r \neq \r$). As described above, we are done unless $u_{p+1}$ has the form $a_{t,r}^q$ ($r \neq \r$). If that occurs, by considering the structure of $F_{\r}(x)$, there must exist $2 \leq j \leq R$ such that $u_{p+j}$ has the form $\gamma(X,Y)$ for some $Y \in \G$ (note that $m \geq R$ by \eqref{eq.mnlb}). But this can't happen because $x \in X$.

We are done for similar reasons if $\G_K[H,Z]$ is contained in $\G_K[H,X]$.
\end{proof}

So we can assume there exists some $l_\dagger < \min(L_X,L_Z)$ for which $X_{l} = Z_{l}$ when $l \leq l_\dagger$ and $X_{l_\dagger+1} \neq Z_{l_\dagger+1}$. In other words $X_{l_\dagger} = Z_{l_\dagger}$ is the divergence point of the two paths $\G_K[H_1,X]$ and $\G_K[H_1,Z]$ in $\pcg$. 

\begin{proofclaim2} \label{2claim2}
We may assume that one of the two following situations occur:
\begin{enumerate}
    \item $\zeta_{l_\dagger} \neq \zeta_{l_\dagger}'$;
    \item $\zeta_{l_\dagger} = \zeta_{l_\dagger}'$  and $\gamma(X_{l_\dagger},X_{l_\dagger + 1}) \neq \gamma(Z_{l_\dagger},Z_{l_\dagger + 1})$.
\end{enumerate}
\end{proofclaim2}

\begin{proof}[Proof of \Cref{2claim2}]
If $\zeta_{l_\dagger} = \zeta_{l_\dagger}'$ and $\gamma(X_{l_\dagger},X_{l_\dagger + 1}) = \gamma(Z_{l_\dagger},Z_{l_\dagger + 1})$ then 
\[\zeta_1 \ \gamma(X_1,X_2) \ \zeta_2 \ \gamma(X_2, X_3) \ \dots \zeta_{l_\dagger} \gamma(X_{l_\dagger},X_{l_\dagger + 1}) = \zeta_1' \ \gamma(Z_1,Z_2) \ \zeta_2' \ \gamma(Z_2, Z_3) \ \dots \zeta_{l_\dagger}' \gamma(Z_{l_\dagger},Z_{l_\dagger + 1})\]
and so $p(X_{l_\dagger+1},X_{l_\dagger}) =_G p(Z_{l_\dagger + 1},Z_{l_\dagger})$. It follows that $X_{l_\dagger + 1} \cap Z_{l_\dagger + 1} \neq \emptyset$ and so they are cosets of different peripheral subgroups $H_1, H_2, \dots H_T$. But then, due to \ref{condition.natural2}, it is possible to distinguish $\gamma(X_{l_\dagger},X_{l_\dagger + 1})$ from $\gamma(Z_{l_\dagger},Z_{l_\dagger + 1})$ just by looking at the final $\chi$ letters. In particular, $\gamma(X_{l_\dagger},X_{l_\dagger + 1}) \neq \gamma(Z_{l_\dagger},Z_{l_\dagger + 1})$ which is a contradiction. 
\end{proof}

So, since either \ref{case.c1} or \ref{case.c2} occurs, and since one of the two cases from \Cref{2claim2} occurs, we have reduced the problem to the following two possibilities: 
\begin{enumerate}
    \myitem{(D1)}\label{case.d1} $\zeta_{l_\dagger} \neq \zeta_{l_\dagger}'$ and $u_{p+1} = a_{t,r}^q$ and $u_{p+1}' = a_{t,r}^{q'}$ where $r \neq \r$;
    \myitem{(D2)} \label{case.d2} $\zeta_{l_\dagger} = \zeta_{l_\dagger}'$ and $u_{p+1} = \gamma(X_{l_\dagger},X_{l_\dagger + 1})$ and $u_{p+1}' = \gamma(Z_{l_\dagger},Z_{l_\dagger + 1})$ and $p(X_{l_\dagger},X_{l_\dagger + 1}) = p(Z_{l_\dagger},Z_{l_\dagger + 1})$.
\end{enumerate}

\begin{proofclaim2} \label{2claim3}
When \ref{case.d1} holds, we have
\[\absval{q - q'} \leq \length(\eta) \leq \lambdacosetsp d_{\C(X_{l_\dagger})}(p(X_{l_\dagger}, X_{l_\dagger + 1}),p(Z_{l_\dagger}, Z_{l_\dagger + 1})) + \mucosetsp \]
where $\eta$ is a standard path from $p(X_{l_\dagger}, X_{l_\dagger + 1})$ to $p(Z_{l_\dagger}, Z_{l_\dagger + 1})$ in the coset $\C(X_{l_\dagger}) = \C(Z_{l_\dagger})$ in the sense of \Cref{def.spcoset}.
\end{proofclaim2}

\begin{proof}[Proof of \Cref{2claim3}]
If \ref{case.d1} holds then $u_{p+1} = a_{t,r}^q$ and $u_{p+1}' = a_{t,r}^{q'}$ correspond to subpaths of $\zeta_{l_\dagger}$ and $\zeta_{l_\dagger}'$ respectively. 

We know that $\zeta_{l_\dagger}$ travels from $p_0(X_{l_\dagger})$ to $p(X_{l_\dagger}, X_{l_\dagger + 1})$ and has the form $\nu \psi$ where $\nu$ is a geodesic in $p_0(X_{l_\dagger})A_t \cong \Z^{R_t}$ from $p_0(X_{l_\dagger})$ to $p(X_{l_\dagger}, X_{l_\dagger + 1})\psi^{-1}$ and $\psi$ is one of the non-abelian paths $\psi_{t,i}$. Analogously, we can write $\zeta_{l_\dagger}' = \nu' \psi'$. Now, identifying the abelian coset $p_0(X_{l_\dagger})A_t = p_0(Z_{l_\dagger})A_t$ with $\Z^{R_t}$ by identifying $p_0(X_{l_\dagger}) = p_0(Z_{l_\dagger})$ with $(0,0,\dots,0)$, we know that the $r$'th coordinate of $p(X_{l_\dagger}, X_{l_\dagger + 1}) \psi^{-1}$ is $q$ and the $r$'th coordinate of $p(Z_{l_\dagger}, Z_{l_\dagger + 1}) (\psi')^{-1}$ is $q'$. Hence,
\[\absval{q - q'} \leq d_{\Z^{R_t}}(p(X_{l_\dagger}, X_{l_\dagger + 1}) \psi^{-1}, p(Z_{l_\dagger}, Z_{l_\dagger + 1}) (\psi')^{-1})\]
But we also know that a standard path $\eta$ from $p(X_{l_\dagger}, X_{l_\dagger + 1})$ to $p(Z_{l_\dagger}, Z_{l_\dagger + 1})$ in the coset $X_{l_\dagger} = Z_{l_\dagger}$ contains a geodesic in $p_0(X_{l_\dagger})A_t \cong \Z^{R_t}$ from $p(X_{l_\dagger}, X_{l_\dagger + 1}) \psi^{-1}$ to $p(Z_{l_\dagger}, Z_{l_\dagger + 1}) (\psi')^{-1}$. It follows that
\[\absval{q-q'} \leq \length(\eta) \leq \lambdacosetsp d_{\C(X_{l_\dagger})}(p(X_{l_\dagger}, X_{l_\dagger + 1}),p(Z_{l_\dagger}, Z_{l_\dagger + 1})) + \mucosetsp\]
where the second inequality follows from \Cref{prop.cosetsp}.
\end{proof}

\begin{proofclaim2} \label{2claim4}
We are done if $X_{l_\dagger} = Z_{l_\dagger}$ is an element of $\G_K[X,Z]$.
\end{proofclaim2}

\begin{proof}[Proof of \Cref{2claim4}]
Suppose this is true. It follows from \cite[Lemma 3.3]{PART1} that 
\begin{equation} \label{eq.zxpath}
\G_K[Z,X] = \{Z_{L_Z} < \dots < Z_{l_\dagger + 1} < Z_{l_\dagger} = X_{l_\dagger} < X_{l_\dagger + 1} < \dots < X_{L_X}\}
\end{equation}

Now, if \ref{case.d1} occurs, then, by considering the structure of $F_{\r}(x)$ and $F_{\r}(z)$, there exists $2 \leq j \leq R$ such that $u_{p+j} = \gamma(X_{l_\dagger},X_{l_\dagger+1})$ and there exists $2 \leq j' \leq R$ be such that $u_{p+j'}' = \gamma(Z_{l_\dagger},Z_{l_{\dagger+1}})$. So if either \ref{case.d1} or \ref{case.d2} occurs, there exists $1 \leq j \leq R$ such that $u_{p+j} = \gamma(X_{l_\dagger},X_{l_\dagger+1})$ and there exists $1 \leq j' \leq R$ be such that $u_{p+j'}' = \gamma(Z_{l_\dagger},Z_{l_{\dagger+1}})$.

Suppose that $j < j'$. Then $u_{p+j} = \gamma(X_{l_\dagger},X_{l_\dagger+1})$ yet $u_{p+j}'$ has the form $a_{t,r}^q$ or $\psi_{t,i}$. Recalling that $\jf \geq R$, it follows that $\alpha$ and $\beta$ satisfy $\leo(\scalf,\jf)$. Similarly, we are done if $j' < j$. So we may assume that $j = j'$. 

Now, we can see from \eqref{eq.zxpath} that $X_{l_\dagger} < X_{l_\dagger+1} < \dots < X$ is a subpath of $\G_K[Z,X]$ and hence $u_{p+j+1} \dots u_{p+m}$ is a subpath of the $\r$'th ordered standard path from $z$ to $x$. It follows that $\length(u_{p+j+1} \dots u_{p+m}) \leq \lambdasp d + \musp$ by \Cref{prop.sp}. Hence $\ol{u_{p+j+1}} \dots \ol{u_{p+m}}$ has at most $(\lambdasp + \musp)d$ letters. Recall from \eqref{eq.mn} that $m,n \geq \frac{d}{6(R+1)\lambdaqie}$ and so $\ol{u_{p+j+1}} \dots \ol{u_{p+m}}$ has at least $\frac{d}{6(R+1)\lambdaqie} - R$ words. Now, thanks to \eqref{eq.dlb}, we have $d \geq 12 R(R+1)\lambdaqie$ and so $R \leq \frac{d}{12(R+1)\lambdaqie}$. Therefore $\frac{d}{6(R+1)\lambdaqie} - R \geq \frac{d}{12(R+1)\lambdaqie}$. It follows that
\[\awl(\ol{u_{p+j+1}}\dots\ol{u_{p+m}}) \leq \frac{(\lambdasp + \musp)d}{(d / 12(R+1)\lambdaqie)} \leq 12(R+1) \lambdaqie (\lambdasp + \musp) \leq N\]
For similar reasons,
\[\awl(\ol{u_{p+j+1}'}\dots\ol{u_{p+n}'}) \leq N\]
In other words, we have proved that \ref{condition.taurus1} holds for $\alpha$ and $\beta$. We now turn to proving that \ref{condition.taurus2} holds for $\alpha$ and $\beta$ (with respect to the choices for $\scall$, $\jl$, $N$, $\epsilon$ given above). 

Suppose first that \ref{case.d1} holds, i.e. $u_{p+1}$ has the form $a_{t,r}^q$ and $u_{p+1}'$ has the form $a_{t,r}^{q'}$ where $r \neq \r$. We want to bound $\absval{q - q'}$ in terms of $d = d_G(x,z)$. Let $\zeta$ denote a standard path from $x$ to $z$ in $G$. Then the intersection of $\zeta$ with $\C(X_{l_\dagger}) = \C(Z_{l_\dagger})$, which we denote by $\eta$, is a standard path from $p(X_{l_\dagger}, X_{l_\dagger + 1})$ to $p(Z_{l_\dagger}, Z_{l_\dagger + 1})$ in the coset $X_{l_\dagger} = Z_{l_\dagger}$. Hence, using \Cref{2claim3} and \Cref{prop.sp}, we have
\begin{equation} \label{eq.qshift}
\absval{q - q'} \leq \length(\eta) \leq \length(\zeta) \leq \lambdasp d + \musp \leq \omega d   
\end{equation}
So we can deduce, from \Cref{lem.nomt}, that either the final $\omega d$ letters of $a_{t,r}^q$ and $a_{t,r}^{q'}$ are distinct or the final $\omega d$ letters of the base $10$ expansions of $\absval{q} \in \N$ and $\absval{q'} \in \N$ are distinct. 

Now let $1 \leq j' \leq j$. Consider the linear statistic "the final $2(R+1) \lambdaqie \omega c$ letters of $w_{i-j'+1}$". When applied to the sentences $\ol{u_1} \dots \ol{u_{p+j'}}$ and $\ol{u_1} \dots \ol{u_{p+j'}'}$ with $c = m+n$ this becomes "the final $2(R+1) \lambdaqie \omega (m+n)$ letters of $a_{t,r}^q$" and "the final $2(R+1) \lambdaqie \omega (m+n)$ letters of $a_{t,r}^{q'}$" respectively. Note that $2(R+1) \lambdaqie \omega (m+n) \geq \omega d$ by \eqref{eq.mplusnlb}. Hence $\alpha$ and $\beta$ satisfy $\taurus(\scall,\jl,N,\epsilon)$ if the final $\omega d$ letters of $a_{t,r}^q$ and $a_{t,r}^{q'}$ are distinct. For similar reasons, by considering the linear statistic "the final $2(R+1) \lambdaqie \omega c$ letters of the base $10$ expansion of $\length(w_{i-j'+1}) \in \N$", we are done when the final $\omega d$ letters of the base $10$ expansions of $\absval{q}$ and $\absval{q'}$ are distinct. So we are done when \ref{case.d1} holds. 

Now suppose that \ref{case.d2} holds. In particular this implies that $p(X_{l_\dagger},X_{l_\dagger+1}) = p(Z_{l_\dagger},Z_{l_\dagger+1})$. By considering \eqref{eq.zxpath}, we see that $u_{p+1}$ is a subpath of a standard path from $z$ to $x$ and so $u_{p+1}$ has length at most $\lambdasp d + \musp \leq \omega d$ by \Cref{prop.sp}. Similarly $\length(u_{p+1}') \leq \omega d$. Since $u_{p+1} \neq u_{p+1}'$, it follows that they are distinguished by the linear statistic "the final $2(R+1)\lambdaqie \omega c$ letters of $w_{i}$" in $\scall$ and so $\alpha$ and $\beta$ satisfy $\taurus(\scall,\jl,N,\epsilon)$.
\end{proof}

So we can assume that $X_{l_\dagger} = Z_{l_\dagger} \not \in \G_K[X,Z]$.

\begin{proofclaim2} \label{2claim5}
We are done when \ref{case.d1} holds.
\end{proofclaim2}

\begin{proof}[Proof of \Cref{2claim5}]
Suppose \ref{case.d1} holds. If $q > 0$ and $q' < 0$ then of course $\alpha$, $\beta$ satisfy $\leo(\scalf,\jf)$ as $\ol{u_1} \dots \ol{u_p} \ \ol{u_{p+1}}$ and $\ol{u_1} \dots \ol{u_p} \ \ol{u_{p+1}'}$ are distinguished by the finite statistic "the final $I + 3K$ letters of $w_i$". Similarly we are done if $q < 0$ and $q' > 0$. So $q,q' > 0$ or $q,q' < 0$ and so $\length(u_{p+1}) \neq \length(u_{p+1}')$.

Since $X_{l_\dagger} = Z_{l_\dagger} \not \in \G_K[X,Z]$, we know that $d_{X_{l_\dagger}}(X,Z) \leq K$. Recall that $d_{X_{l_\dagger}}(X,Z)$ is by definition the diameter of $\pi_{X_{l_\dagger}}(X) \cup \pi_{X_{l_\dagger}}(Z)$ in $\C(X_{l_\dagger})$. Hence, using \Cref{2claim3} we have
\[\absval{q - q'} \leq \lambdacosetsp d_{\C(X_{l_\dagger})}(p(X_{l_\dagger}, X_{l_\dagger + 1}),p(Z_{l_\dagger}, Z_{l_\dagger + 1})) + \mucosetsp \leq \lambdacosetsp d_{X_{l_\dagger}}(X,Z) + \mucosetsp \leq \lambdacosetsp K + \mucosetsp\]
So the sentences $\ol{u_1} \dots \ol{u_p} \ \ol{u_{p+1}}$ and $\ol{u_1} \dots \ol{u_p} \ \ol{u_{p+1}'}$ are distinguished by the finite statistic "$\length(w_i)$ modulo $(\lambdacosetsp K + \mucosetsp + 1)$".
\end{proof}

By \Cref{2claim5} we can assume that \ref{case.d2} holds.

\begin{proofclaim2} \label{2claim6}
We may assume that for all $1 \leq k \leq \jf$ the group elements $u_1 \dots u_p u_{p+1} \dots u_{p+k}$ and $u_1 \dots u_p u_{p+1}' \dots u_{p+k}'$ are distinct (as elements of $G$).
\end{proofclaim2}

\begin{proof}[Proof of \Cref{2claim6}]
Suppose $1 \leq k \leq \jf$ is such that $u_1 \dots u_p u_{p+1} \dots u_{p+k} =_G u_1 \dots u_p u_{p+1}' \dots u_{p+k}'$.

Firstly, note that $u_1 \dots u_p u_{p+1} \dots u_{p+k}$ is an element of the coset $X_{l}$ for some $l \geq l_\dagger + 1$ and similarly $u_1 \dots u_p u_{p+1}' \dots u_{p+k}'$ is an element of the coset $Z_{l'}$ for some $l' \geq l_\dagger + 1$. We know that $X_l \neq Z_{l'}$ since $X_l \in \G_K[X_{l_\dagger + 1},X]$ and $Z_l \in \G_K[Z_{l_\dagger + 1},Z]$. 

Due to \ref{condition.natural1}, we know that the last word of $\ol{u_1} \dots \ol{u_p} \ \ol{u_{p+1}} \dots \ol{u_{p+k}}$ not contained in one of the cosets $\C(H_t)$ has the form $\gamma(X_{l-1},X_l)$. Similarly, the last word of $\ol{u_1} \dots \ol{u_p} \ \ol{u_{p+1}'} \dots \ol{u_{p+k}'}$ not contained in one of the cosets $\C(H_t)$ has the form $\gamma(Z_{l'-1},Z_l')$.

But since $u_1 \dots u_p u_{p+1} \dots u_{p+k} =_G u_1 \dots u_p u_{p+1}' \dots u_{p+k}'$ we know that $X_{l} \cap Z_{l'} \neq \emptyset$ and so, since $X_l \neq Z_{l'}$, we deduce that $X_l$ and $Z_{l'}$ are cosets of different peripheral subgroups $H_1, \dots, H_T$. It follows from \ref{condition.natural2} that the final $\chi$ letters of $\gamma(X_{l-1},X_l)$ are distinct from the final $\chi$ letters of $\gamma(Z_{l'-1},Z_l')$. By considering the fourth finite statistic in $\scalf$, we deduce that $\alpha$ and $\beta$ satisfy $\leo(\scalf,\jf)$.
\end{proof}

\subsection{The non-degenerate case}

We have now reached the heart of the proof.

Consider the path $\G_K[X,Z]$. By \cite[Lemma 3.6]{BBFS}, we know that $\G_K[X,Z]$ diverges from $\G_K[X,H_1]$ at some vertex $X_{l_\dagger + l_x}$ where $1 \leq l_x \leq 3$ and that it joins $\G_K[H_1,Z]$ at a vertex $Z_{l_\dagger + l_z}$ for some $1 \leq l_z \leq 3$. See \Cref{fig.bigdiagram} for a visual depiction of the situation at which we have arrived. 

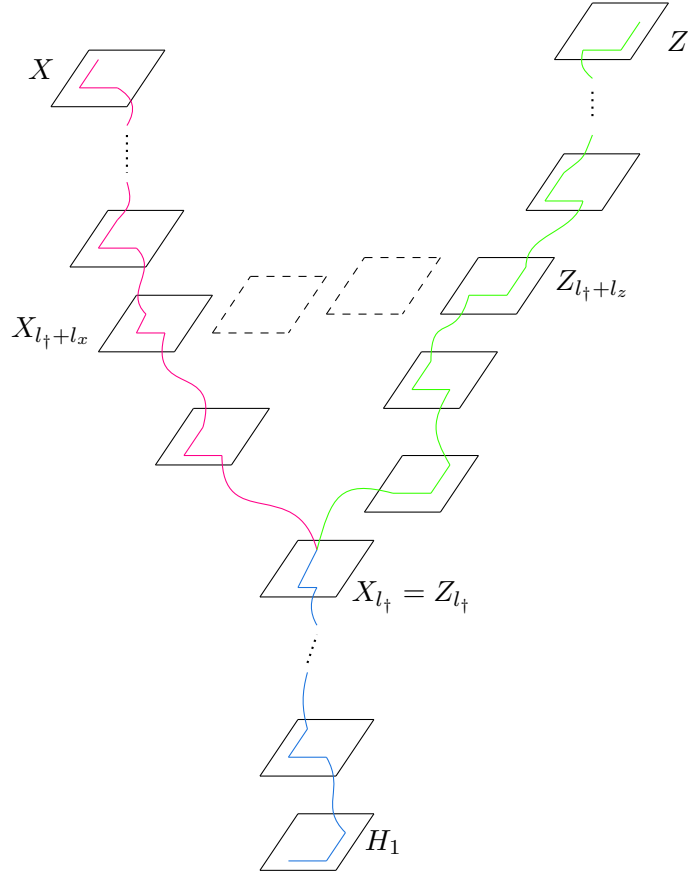
\begin{figure}
    \vspace{10mm}
    \centering
    \scalebox{1}{
    \begin{tikzpicture}
	\begin{pgfonlayer}{nodelayer}
		\node [style=none] (0) at (-1, -6) {};
		\node [style=none] (1) at (0, -4.5) {};
		\node [style=none] (2) at (2, -4.5) {};
		\node [style=none] (3) at (1, -6) {};
		\node [style=none] (4) at (-1, -3.5) {};
		\node [style=none] (5) at (0, -2) {};
		\node [style=none] (6) at (2, -2) {};
		\node [style=none] (7) at (1, -3.5) {};
		\node [style=none] (8) at (-0.25, -5.75) {};
		\node [style=none] (9) at (1.25, -5) {};
		\node [style=none] (10) at (0.75, -5.75) {};
		\node [style=none] (11) at (0.75, -3) {};
		\node [style=none] (12) at (-0.25, -3) {};
		\node [style=none] (13) at (0.25, -2.25) {};
		\node [style=none] (14) at (0.25, -0.75) {};
		\node [style=none] (15) at (0.25, -0.5) {};
		\node [style=none] (16) at (0.5, 0.25) {};
		\node [style=none] (17) at (0.5, 0.5) {};
		\node [style=none] (18) at (-1, 1.25) {};
		\node [style=none] (19) at (0, 2.75) {};
		\node [style=none] (20) at (2, 2.75) {};
		\node [style=none] (21) at (1, 1.25) {};
		\node [style=none] (22) at (0.5, 1.5) {};
		\node [style=none] (23) at (0, 1.5) {};
		\node [style=none] (24) at (0.5, 2.5) {};
		\node [style=none] (25) at (-3.75, 4.75) {};
		\node [style=none] (26) at (-2.75, 6.25) {};
		\node [style=none] (27) at (-0.75, 6.25) {};
		\node [style=none] (28) at (-1.75, 4.75) {};
		\node [style=none] (29) at (1.75, 3.5) {};
		\node [style=none] (30) at (2.75, 5) {};
		\node [style=none] (31) at (4.75, 5) {};
		\node [style=none] (32) at (3.75, 3.5) {};
		\node [style=none] (33) at (2.25, 6.25) {};
		\node [style=none] (34) at (3.25, 7.75) {};
		\node [style=none] (35) at (5.25, 7.75) {};
		\node [style=none] (36) at (4.25, 6.25) {};
		\node [style=none] (37) at (3.75, 8.75) {};
		\node [style=none] (38) at (4.75, 10.25) {};
		\node [style=none] (39) at (6.75, 10.25) {};
		\node [style=none] (40) at (5.75, 8.75) {};
		\node [style=none] (41) at (-5.25, 7.75) {};
		\node [style=none] (42) at (-4.25, 9.25) {};
		\node [style=none] (43) at (-2.25, 9.25) {};
		\node [style=none] (44) at (-3.25, 7.75) {};
		\node [style=none] (45) at (-2.25, 8.25) {};
		\node [style=none] (46) at (-1.25, 9.75) {};
		\node [style=none] (47) at (0.75, 9.75) {};
		\node [style=none] (48) at (-0.25, 8.25) {};
		\node [style=none] (49) at (6, 11.5) {};
		\node [style=none] (50) at (7, 13) {};
		\node [style=none] (51) at (9, 13) {};
		\node [style=none] (52) at (8, 11.5) {};
		\node [style=none] (53) at (6.75, 15.5) {};
		\node [style=none] (54) at (7.75, 17) {};
		\node [style=none] (55) at (9.75, 17) {};
		\node [style=none] (56) at (8.75, 15.5) {};
		\node [style=none] (57) at (-6, 10) {};
		\node [style=none] (58) at (-5, 11.5) {};
		\node [style=none] (59) at (-3, 11.5) {};
		\node [style=none] (60) at (-4, 10) {};
		\node [style=none] (61) at (-6.5, 14.25) {};
		\node [style=none] (62) at (-5.5, 15.75) {};
		\node [style=none] (63) at (-3.5, 15.75) {};
		\node [style=none] (64) at (-4.5, 14.25) {};
		\node [style=none] (65) at (2.5, 4) {};
		\node [style=none] (66) at (-2, 5) {};
		\node [style=none] (67) at (3.5, 4) {};
		\node [style=none] (68) at (4, 4.75) {};
		\node [style=none] (69) at (-3, 5) {};
		\node [style=none] (70) at (-2.5, 5.75) {};
		\node [style=none] (71) at (-3.5, 8.25) {};
		\node [style=none] (72) at (-4.25, 8.25) {};
		\node [style=none] (73) at (-4, 8.75) {};
		\node [style=none] (74) at (-4.25, 10.5) {};
		\node [style=none] (75) at (-5.25, 10.5) {};
		\node [style=none] (76) at (-4.75, 11.25) {};
		\node [style=none] (77) at (-4.75, 14.75) {};
		\node [style=none] (78) at (-5.75, 14.75) {};
		\node [style=none] (79) at (-5.25, 15.5) {};
		\node [style=none] (82) at (-4.5, 13.5) {};
		\node [style=none] (85) at (-4.5, 12.5) {};
		\node [style=none] (86) at (0.75, 8.75) {};
		\node [style=none] (87) at (1.75, 10.25) {};
		\node [style=none] (88) at (3.75, 10.25) {};
		\node [style=none] (89) at (2.75, 8.75) {};
		\node [style=none] (90) at (-4.5, 12.25) {};
		\node [style=none] (91) at (-4.5, 13.75) {};
		\node [style=none] (92) at (4, 6.75) {};
		\node [style=none] (93) at (3, 6.75) {};
		\node [style=none] (94) at (3.5, 7.5) {};
		\node [style=none] (95) at (4.5, 9.25) {};
		\node [style=none] (96) at (5.5, 9.25) {};
		\node [style=none] (97) at (6, 10) {};
		\node [style=none] (98) at (7.5, 11.75) {};
		\node [style=none] (99) at (6.5, 11.75) {};
		\node [style=none] (100) at (7, 12.5) {};
		\node [style=none] (101) at (7.75, 13.5) {};
		\node [style=none] (102) at (7.75, 14) {};
		\node [style=none] (103) at (7.75, 14.75) {};
		\node [style=none] (104) at (7.75, 15) {};
		\node [style=none] (105) at (7.5, 15.75) {};
		\node [style=none] (106) at (8.5, 15.75) {};
		\node [style=none] (107) at (9, 16.5) {};
		\node [style=none] (109) at (-6.75, 15.25) {$X$};
		\node [style=none] (110) at (10, 16) {$Z$};
		\node [style=none] (112) at (3, 1.25) {$X_{l_\dagger} = Z_{l_\dagger}$};
		\node [style=none] (113) at (7.75, 9.5) {$Z_{l_\dagger + l_z}$};
		\node [style=none] (115) at (-6.5, 8.25) {$X_{l_\dagger + l_x}$};
		\node [style=none] (116) at (2.25, -5.25) {$H_1$};
	\end{pgfonlayer}
	\begin{pgfonlayer}{edgelayer}
		\draw (0.center) to (3.center);
		\draw (0.center) to (1.center);
		\draw (1.center) to (2.center);
		\draw (2.center) to (3.center);
		\draw (4.center) to (7.center);
		\draw (4.center) to (5.center);
		\draw (5.center) to (6.center);
		\draw (6.center) to (7.center);
		\draw [style=blueline] (8.center) to (10.center);
		\draw [style=blueline] (10.center) to (9.center);
		\draw [style=blueline, in=-60, out=135, looseness=1.25] (9.center) to (11.center);
		\draw [style=blueline] (11.center) to (12.center);
		\draw [style=blueline] (12.center) to (13.center);
		\draw [style=blueline, bend left=15] (13.center) to (14.center);
		\draw [style={dotted?}] (15.center) to (16.center);
		\draw (18.center) to (21.center);
		\draw (18.center) to (19.center);
		\draw (19.center) to (20.center);
		\draw (20.center) to (21.center);
		\draw [style=blueline, bend left] (17.center) to (22.center);
		\draw [style=blueline] (22.center) to (23.center);
		\draw [style=blueline] (23.center) to (24.center);
		\draw (25.center) to (28.center);
		\draw (25.center) to (26.center);
		\draw (26.center) to (27.center);
		\draw (27.center) to (28.center);
		\draw (29.center) to (32.center);
		\draw (29.center) to (30.center);
		\draw (30.center) to (31.center);
		\draw (31.center) to (32.center);
		\draw (33.center) to (36.center);
		\draw (33.center) to (34.center);
		\draw (34.center) to (35.center);
		\draw (35.center) to (36.center);
		\draw (37.center) to (40.center);
		\draw (37.center) to (38.center);
		\draw (38.center) to (39.center);
		\draw (39.center) to (40.center);
		\draw (41.center) to (44.center);
		\draw (41.center) to (42.center);
		\draw (42.center) to (43.center);
		\draw (43.center) to (44.center);
		\draw [style={dashed?}] (45.center) to (48.center);
		\draw [style={dashed?}] (45.center) to (46.center);
		\draw [style={dashed?}] (46.center) to (47.center);
		\draw [style={dashed?}] (47.center) to (48.center);
		\draw (49.center) to (52.center);
		\draw (49.center) to (50.center);
		\draw (50.center) to (51.center);
		\draw (51.center) to (52.center);
		\draw (53.center) to (56.center);
		\draw (53.center) to (54.center);
		\draw (54.center) to (55.center);
		\draw (55.center) to (56.center);
		\draw (57.center) to (60.center);
		\draw (57.center) to (58.center);
		\draw (58.center) to (59.center);
		\draw (59.center) to (60.center);
		\draw (61.center) to (64.center);
		\draw (61.center) to (62.center);
		\draw (62.center) to (63.center);
		\draw (63.center) to (64.center);
		\draw [style=pinkline, in=270, out=105, looseness=1.50] (24.center) to (66.center);
		\draw [style=pinkline] (66.center) to (69.center);
		\draw [style=pinkline] (69.center) to (70.center);
		\draw [style=pinkline, in=-105, out=75, looseness=1.75] (70.center) to (71.center);
		\draw [style=pinkline] (71.center) to (72.center);
		\draw [style=pinkline] (72.center) to (73.center);
		\draw [style=pinkline] (74.center) to (75.center);
		\draw [style=pinkline] (75.center) to (76.center);
		\draw [style={dotted?}] (82.center) to (85.center);
		\draw [style={dashed?}] (86.center) to (89.center);
		\draw [style={dashed?}] (86.center) to (87.center);
		\draw [style={dashed?}] (87.center) to (88.center);
		\draw [style={dashed?}] (88.center) to (89.center);
		\draw [style=pinkline, in=315, out=135, looseness=1.25] (73.center) to (74.center);
		\draw [style=pinkline, bend right, looseness=1.25] (76.center) to (90.center);
		\draw [style=pinkline, bend right=45, looseness=1.25] (91.center) to (77.center);
		\draw [style=pinkline] (77.center) to (78.center);
		\draw [style=pinkline] (78.center) to (79.center);
		\draw [style=greenline] (65.center) to (67.center);
		\draw [style=greenline] (67.center) to (68.center);
		\draw [style=greenline, in=165, out=75, looseness=1.50] (24.center) to (65.center);
		\draw [style=greenline, bend left, looseness=1.25] (68.center) to (92.center);
		\draw [style=greenline] (92.center) to (93.center);
		\draw [style=greenline] (93.center) to (94.center);
		\draw [style=greenline, in=255, out=90, looseness=1.75] (94.center) to (95.center);
		\draw [style=greenline] (95.center) to (96.center);
		\draw [style=greenline] (96.center) to (97.center);
		\draw [style=greenline, in=-105, out=90] (97.center) to (98.center);
		\draw [style=greenline] (98.center) to (99.center);
		\draw [style=greenline] (99.center) to (100.center);
		\draw [style=greenline, bend right=15, looseness=1.50] (100.center) to (101.center);
		\draw [style=greenline, bend left, looseness=1.25] (104.center) to (105.center);
		\draw [style=greenline] (105.center) to (106.center);
		\draw [style=greenline] (106.center) to (107.center);
		\draw [style={dotted?}] (102.center) to (103.center);
	\end{pgfonlayer}
\end{tikzpicture}
    }
    \caption{This diagram is drawn as if $H_t = A_t = \Z^2$ for all $1 \leq t \leq T$. It is also drawn for $l_x = 2$ and $l_z = 3$. The blue line indicates where the $\r$'th ordered standard path from $e$ to $x$ (i.e. $\zeta$) and the $\r$'th ordered standard path from $e$ to $z$ (i.e. $\zeta'$) agree. The pink line is the rest of $\zeta$ from $p(X_{l_\dagger},X_{l_\dagger + 1}) = p(Z_{l_\dagger},Z_{l_\dagger + 1})$ to $x$. The green line is the rest of $\zeta'$ from $p(X_{l_\dagger},X_{l_\dagger + 1}) = p(Z_{l_\dagger},Z_{l_\dagger + 1})$ to $x$. The dashed squares are those which are in $\G_K[X,Z]$ but not in $\G_K[H_1,X]$ or $\G_H[H_1,Z]$.}
    \label{fig.bigdiagram}
\end{figure}

Suppose that $j_x \in \N$ is such that $u_{p+j_x} = \gamma(X_{l_\dagger+l_x-1},X_{l_\dagger+l_x})$. Suppose that $j_z \in \N$ is such that $u_{p+j_z}' = \gamma(Z_{l_\dagger+l_z-1},Z_{l_\dagger+l_z})$.

\begin{proofclaim2} \label{2claim7}
We have $j_x \leq \jl$ and $j_z \leq \jl$.
\end{proofclaim2}

\begin{proof}[Proof of \Cref{2claim7}]
We begin by proving that $j_x \leq \jl$. Now, $\ol{u_{p+1}} \ \ol{u_{p+2}} \dots \ol{u_{p+j_x}}$ has the form 
\[\ol{\gamma(X_{l_\dagger},X_{l_\dagger+1})}\]
or the form
\[\ol{\gamma(X_{l_\dagger},X_{l_\dagger+1})} \alpha_{l_\dagger+1} \ol{\gamma(X_{l_\dagger+1}, X_{l_\dagger+2})} \] 
or the form
\[\ol{\gamma(X_{l_\dagger},X_{l_\dagger+1})} \alpha_{l_\dagger+1} \ol{\gamma(X_{l_\dagger+1}, X_{l_\dagger+2})} \alpha_{l_\dagger+2}\ol{\gamma(X_{l_\dagger+2},X_{l_\dagger+3})}\] 
depending on whether $l_x = 1$, $2$ or $3$. In the above, we are using the notation from \Cref{def.fr}. 

We will prove that (in the second and third cases above) we have $\length(\alpha_{l_\dagger + 1}) \leq 2 \lambdacosetsp K + \mucosetsp$. First, note that $\length(\alpha_{l_\dagger + 1}) \leq \length(\zeta_{l_\dagger + 1})$. But we have
\[\length(\zeta_{l_\dagger + 1}) \leq \lambdacosetsp d_{\C(X_{l_\dagger + 1})}(p(X_{l_\dagger + 1},X_{l_\dagger}), p(X_{l_\dagger + 1},X_{l_\dagger + 2})) + \mucosetsp \leq \lambdacosetsp d_{X_{l_\dagger + 1}}(H_1, X) + \mucosetsp\]
where the first inequality follows from \Cref{prop.cosetsp}. However, and this is really the crucial observation, since $X_{l_\dagger + 1} \not \in \G_K[H_1, Z]$ and $X_{l_\dagger + 1} \not \in \G_K[X, Z]$ we have
\[d_{X_{l_\dagger + 1}}(H_1, X) \leq d_{X_{l_\dagger + 1}}(H_1, Z) + d_{X_{l_\dagger + 1}}(Z, X) \leq 2K\]
and so 
\[\length(\alpha_{l_\dagger + 1}) \leq \length(\zeta_{l_\dagger + 1}) \leq 2 \lambdacosetsp K + \mucosetsp\]
For identical reasons, in the third case above, we have $\length(\alpha_{l_\dagger + 2}) \leq 2\lambdacosetsp K + \mucosetsp$. It follows that, in all three cases, we have $j_x \leq 2(2\lambdacosetsp K + \mucosetsp) + 3 = \jl$.

We have $j_z \leq \jl$ for analogous reasons. 
\end{proof}

We define $j = \max(j_x, j_z)$. Without loss of generality we can assume that in fact $j = j_x$. Let $p_x$ denote the group element $u_1 \dots u_p u_{p+1} \dots u_{p+j_x} = p_0(X_{l_\dagger + l_x}) = p(X_{l_\dagger + l_x}, X_{l_\dagger + l_x - 1})$ and let $p_z$ denote the group element $u_1 \dots u_p u_{p+1}' \dots u_{p+j_z}' = p_0(Z_{l_\dagger + l_z}) = p(Z_{l_\dagger + l_z}, Z_{l_\dagger + l_z - 1})$. Then $u_{p+j_x+1} u_{p+j_x+2} \dots u_{p+m}$ is a path in $\Gamma(G,S)$ from $p_x$ to $x$ and $u_{p+j_z+1}' u_{p+j_z+2}' \dots u_{p+n}'$ is a path in $\Gamma(G,S)$ from $p_z$ to $z$. 

Recall that $\G_K[X,Z] = \{X = Y_1 < Y_2 < \dots < Y_{L} = Z\}$. Let $1 \leq l \leq L$ be such that $Y_l = X_{l_\dagger + l_x}$. Let $p = p(Y_{l}, Y_{l+1}) \in X_{l_\dagger + l_x}$. So $p \in \pi_{X_{l_\dagger + l_x}}(Z)$. Now, since $X_{l_\dagger + l_x} \not \in \G_K[H_1,Z]$, we have $d_{X_{l_\dagger + l_x}}(H_1, Z) \leq K$ and so 
\begin{equation} \label{eq.px}
d_G(p_x, p) \leq K
\end{equation}

Similarly, let $1 \leq l' \leq L$ be such that $Y_{l'} = Z_{l_\dagger + l_z}$ and let $p'$ be the vertex $p' = p(Y_{l'}, Y_{l'-1})$. So $p' \in \pi_{Z_{l_\dagger + l_z}}(X)$. Since $Z_{l_\dagger + l_z} \not \in \G_K[H_1,X]$, we have $d_{Z_{l_\dagger + l_z}}(H_1, X) \leq K$ and so \begin{equation} \label{eq.pz}
d_G(p_z, p') \leq K
\end{equation}

\begin{proofclaim2} \label{2claim8}
$\awl(\ol{u_{p+j+1}} \ \ol{u_{p+j+2}} \dots \ol{u_{p+m}})$ and $\awl(\ol{u_{p+j+1}'} \ \ol{u_{p+j+2}'} \dots \ol{u_{p+n}'})$ are both at most $N$.
\end{proofclaim2}

\begin{proof}[Proof of \Cref{2claim8}]
We begin by counting the number of letters in $u_{p+j_x+1} u_{p+j_x+2} \dots u_{p+m}$. Now, $u_{p+j_x+1} u_{p+j_x+2} \dots u_{p+m}$ is a subpath of the $\r$'th ordered standard path from $e$ to $x$. By \Cref{prop.sp}, this is $(\lambdasp,\musp)$-quasigeodesic and hence
\[\length(u_{p+j_x+1} u_{p+j_x+2} \dots u_{p+m}) \leq \lambdasp d_G(p_x,x) + \musp \leq \lambdasp (d_G(p,x) + K) + \musp\]
where we have used \eqref{eq.px}. Now, consider a standard path in $G$ from $z$ to $x$. By definition, this standard path must go through the coset $X_{l_\dagger + l_x}$ since $X_{l_\dagger + l_x} \in \G_K[Z,X]$ and also through the point $p \in X_{l_\dagger + l_x}$. By \Cref{prop.sp}, we know that this standard path has length at most $\lambdasp d + \musp$. Hence,
\[d_G(p,x) \leq \lambdasp d + \musp\]
So we have 
\[\length(u_{p+j_x+1} u_{p+j_x+2} \dots u_{p+m}) \leq \lambdasp^2 d + \lambdasp \musp + \lambdasp K + \musp\]
For similar reasons, we have $\length(u_{p+j_z+1}' u_{p+j_z+2}' \dots u_{p+n}') \leq \lambdasp^2 d + \lambdasp \musp + \lambdasp K + \musp$. 

Since $j = j_x$, we therefore have
\[\length(u_{p+j+1} u_{p+j+2} \dots u_{p+m}) \leq \lambdasp^2 d + \lambdasp \musp + \lambdasp K + \musp\]
Further, we know that $u_{p+j+1}' u_{p+j+2}' \dots u_{p+n}'$ is a subword of $u_{p+j_z+1}' u_{p+j_z+2}' \dots u_{p+n}'$ and so 
\[\length(u_{p+j+1}' u_{p+j+2}' \dots u_{p+n}') \leq \lambdasp^2 d + \lambdasp \musp + \lambdasp K + \musp\]

By \eqref{eq.dlb} we have $d \geq 12(R+1) \lambdaqie \jl$ and hence $\jl \leq \frac{d}{12(R+1) \lambdaqie}$. Therefore, using \Cref{2claim7} and \eqref{eq.mn}, we have
\[m - j \geq m - \jl \geq \frac{d}{6(R+1)\lambdaqie} - \jl \geq \frac{d}{12(R+1) \lambdaqie}\]
Therefore
\begin{align*}
\awl(\ol{u_{p+j+1}} \ \ol{u_{p+j+2}} \dots \ol{u_{p+m}})) &\leq \frac{\lambdasp^2 d + \lambdasp \musp + \lambdasp K + \musp}{d / 12(R+1)\lambdaqie} \\
&\leq 12(R+1)\lambdaqie (\lambdasp^2 + \lambdasp \musp + \lambdasp K + \musp) \\
&= N
\end{align*}
For similar reasons, we also have
\[\awl(\ol{u_{p+j+1}'} \ \ol{u_{p+j+2}'} \dots \ol{u_{p+n}'})) \leq N\]
and we have proved the claim. 
\end{proof}
So \ref{condition.taurus1} of $\taurus(\scall,J,N,\epsilon)$ holds. We now turn to proving \ref{condition.taurus2}. With this in mind, and recalling that $\epsilon = 1$, let us suppose that $1 \leq j' \leq j$ is such that $\length(u_{p+j''}) \leq m+n$ and $\length(u_{p+j''}') \leq m+n$ for all $j' < j'' \leq j$.

\begin{proofclaim2} \label{2claim9}
The distance between $\Phi(u_{p+1}\dots u_{p+j'})$ and $\Phi(u_{p+1}'\dots u_{p+j'}')$ in $\prod_{q=1}^Q T_C$ is at most $\omega d$.
\end{proofclaim2}

\begin{proof}[Proof of \Cref{2claim9}]
Since $\length(u_{p+j''}) \leq m+n$ for all $j' < j'' \leq j$, we have 
\begin{align*}
d_G(u_1 \dots u_p u_{p+1}\dots u_{p+j'}, p_x) &= d_G(u_1 \dots u_p u_{p+1}\dots u_{p+j'}, u_1 \dots u_p u_{p+1}\dots u_{p+j}) \\
&\leq \jl (m+n)
\end{align*}
So, using \eqref{eq.px}, we have
\[d_G(u_1 \dots u_p u_{p+1}\dots u_{p+j'}, p) \leq \jl (m+n) + K\]
Recall that we have $m + n = d'_{\r} \leq d' \leq \lambdaqie d + \muqie$. Hence, 
\begin{equation} \label{eq.dj'p}
d_G(u_1 \dots u_p u_{p+1}\dots u_{p+j'}, p) \leq \jl (\lambdaqie d + \muqie) + K
\end{equation}

Consider a standard path from $x$ to $z$ to $\Gamma(G,S)$. By \Cref{prop.sp}, the length of this path is at most $\lambdasp d + \musp$. Further, note that this standard path goes through the vertices $p$ and $p'$. Therefore,
\begin{equation} \label{eq.dpp'}
d_G(p,p') \leq \lambdasp d + \musp 
\end{equation}

We must now divide into two cases: we can have $j' \leq j_z$ or $j' > j_z$. 

In the first case, since $\length(u_{p+j''}') \leq m+n$ for all $j' < j'' \leq j$, and since $j_z \leq j$, we have 
\begin{align*}
d_G(u_1 \dots u_p u_{p+1}'\dots u_{p+j'}', p_z) &= d_G(u_1 \dots u_p u_{p+1}'\dots u_{p+j'}', u_1 \dots u_p u_{p+1}'\dots u_{p+j_z}') \\
&\leq \sum_{j' < j'' \leq j_z} \length(u_{p+j''}') \\
&\leq \jl (m+n)
\end{align*}
and so, using \eqref{eq.pz}, we have
\begin{equation} \label{eq.dj'p'}
d_G(u_1 \dots u_p u_{p+1}'\dots u_{p+j'}', p') \leq \jl (m+n) + K \leq \jl(\lambdaqie d + \muqie) + K
\end{equation}
Therefore, combining \eqref{eq.dj'p}, \eqref{eq.dpp'} and \eqref{eq.dj'p'}, and recalling that $\Phi: G \rightarrow \prod_{q=1}^Q T_C$ is $\lambdaPhi$-Lipschitz, we have
\begin{align*}
&d(\Phi(u_{p+1}\dots u_{p+j'}), \Phi(u_{p+1}'\dots u_{p+j'}')) \leq \lambdaPhi d_{G}(u_{p+1}\dots u_{p+j'}, u_{p+1}'\dots u_{p+j'}')\\
&= \lambdaPhi d_{G}(u_1 \dots u_p u_{p+1}\dots u_{p+j'}, u_1 \dots u_p u_{p+1}'\dots u_{p+j'}')) \\
&\leq \lambdaPhi(\jl (\lambdaqie d + \muqie) + K + \lambdasp d + \musp + \jl(\lambdaqie d + \muqie) + K) \\
&\leq \omega d
\end{align*}
as desired.

In the second case, when $j' > j_z$, we know that $u_{p+ j_z + 1}'u_{p+ j_z + 2}'\dots u_{p+n}'$ is a subpath of the $\r$'th ordered standard path from $e$ to $z$. By \Cref{prop.sp}, this is a $(\lambdasp,\musp)$-quasigeodesic, hence
\begin{align*}
&\length(u_{p+ j_z + 1}'u_{p+ j_z + 2}'\dots u_{p+j'}') \leq \length(u_{p+ j_z + 1}'u_{p+ j_z + 2}'\dots u_{p+n}') \leq \lambdasp d_G(p_z, z) + \musp \\
&\leq \lambdasp (d_G(p', z) + K) + \musp \leq \lambdasp (\lambdasp d + \musp + K) + \musp
\end{align*}
where we have used \eqref{eq.pz} and the fact that a standard path in $\Gamma(G,S)$ from $x$ to $z$ must pass through $p'$. Therefore
\begin{align}
d_G(u_1 \dots u_p u_{p+1}'\dots u_{p+j'}', p') &\leq d_G(u_1 \dots u_p u_{p+1}'\dots u_{p+j'}', p_z) + K \nonumber \\
&\leq \lambdasp (\lambdasp d + \musp + K) + \musp + K \label{eq.c2}
\end{align}
where we have again used \eqref{eq.pz}. Hence, combining \eqref{eq.dj'p}, \eqref{eq.dpp'} and \eqref{eq.c2}, and using the fact that $\Phi$ is $\lambdaPhi$-Lipschitz, we have
\begin{align*}
&d(\Phi(u_{p+1}\dots u_{p+j'}), \Phi(u_{p+1}'\dots u_{p+j'}')) \leq \lambdaPhi d_{G}(u_{p+1}\dots u_{p+j'}, u_{p+1}'\dots u_{p+j'}')\\
&= \lambdaPhi d_{G}(u_1 \dots u_p u_{p+1}\dots u_{p+j'}, u_1 \dots u_p u_{p+1}'\dots u_{p+j'}')) \\
&\leq \lambdaPhi(\jl (\lambdaqie + \muqie) + K + \lambdasp d + \musp + \lambdasp (\lambdasp d + \musp + K) + \musp + K) \leq \omega d
\end{align*}
and we are done.
\end{proof}

Since $j' \leq \jf$, it follows from \Cref{2claim6} that $u_{p+1}\dots u_{p+j'} \neq u_{p+1}'\dots u_{p+j'}'$. Then, since $\Phi$ is injective by \Cref{lem.reglem}, we have $\Phi(u_{p+1}\dots u_{p+j'}) \neq \Phi(u_{p+1}'\dots u_{p+j'}')$, and so there exists some $1 \leq q \leq Q$ such that $\Phi_q(u_{p+1}\dots u_{p+j'}) \neq \Phi_q(u_{p+1}'\dots u_{p+j'}')$. For simplicity of notation, let us write $v_q = \Phi_q(u_{p+1}\dots u_{p+j'})$ and $v_q' = \Phi_q(u_{p+1}'\dots u_{p+j'}')$. We have $d_{T_C}(v_q,v_q') \leq \omega d$ by \Cref{2claim9}. So we deduce from \Cref{lem.nomt} that either the final $\omega d$ letters of $v_q$ and $v_q'$ are distinct or the final $\omega d$ letters of the base $10$ expansions of $\length(v_q) \in \N$ and $\length(v_q') \in \N$ are distinct. 

Consider the linear statistic "the final $2(R+1) \lambdaqie \omega c$ letters of $\Phi_q(w_{i-j'+1} \dots w_i)$". When applied to the sentences $\ol{u_1} \dots \ol{u_{p+j'}}$ and $\ol{u_1} \dots \ol{u_{p+j'}'}$ with $c = m+n$ this becomes "the final $2(R+1) \lambdaqie \omega (m+n)$ letters of $v_q$" and "the final $2(R+1) \lambdaqie \omega (m+n)$ letters of $v_q'$" respectively. Note that $2(R+1) \lambdaqie \omega (m+n) \geq \omega d$ by \eqref{eq.mplusnlb}. Hence $\alpha$ and $\beta$ satisfy $\taurus(\scall,\jl,N,\epsilon)$ if the final $\omega d$ letters of $v_q$ and $v_q'$ are distinct. For similar reasons, by considering the linear statistic "the final $2(R+1) \lambdaqie \omega c$ letters of the base $10$ expansion of $\length(w_{i-j'+1}) \in \N$", we are done when the final $\omega d$ letters of the base $10$ expansions of $\length(v_q)$ and $\length(v_q')$ are distinct. So $\alpha$ and $\beta$ satisfy $\taurus(\scall,\jl,N,\epsilon)$. \qed

\appendix

\section{A quasiisometric embedding of $\H^2$ into a product of two binary trees} \label{appchap.3}

In this appendix we will prove that there is a quasiisometric embedding of the hyperbolic plane into a product of two binary trees. The proof roughly follows the outline of the proof of the same result given in \cite[Section 6]{BDS}, however I use the terminology of diaries and statistics in order to prove it. I hope this example helps to clarify the usage of diaries and statistics in a simple case. 

One can prove that the hyperbolic plane $\hyp^2$ is quasiisometric to the hexagonal hyperbolic Coxeter group 
\[G = \langle a_1, a_2, a_3, b_1, b_2, b_3 \ | \ a_i^2 = b_i ^2 = e \textrm{ for } i \in \{1,2,3\} \textrm{ and } [a_k, b_l] = e \textrm{ for } k \neq l \rangle\]
To see this, one considers the cocompact action of $G$ on $\hyp^2$, where the generators of $G$ correspond to reflections in the sides of a regular hexagon in $\hyp^2$. Write $S = \{a_1, a_2, a_3, b_1, b_2, b_3\}$. We consider the word metric on $G$ arising from this generating set. 

Let $A = S \cup S^{-1} = S$ and let $W$ be the set of finite words on $A$. We write $\A = \{a_1, a_2, a_3\}$ and $\B = \{b_1, b_2, b_3\}$. 

\begin{definition}
Let $g \in G$. The \textit{$a$-left representation} of $g$ is the unique geodesic word $w \in W$ which represents $g$ such that all the $\A$-letters in $w$ have been commuted as far to the left as possible. For example, if $g = b_1 a_2 a_3 b_2 a_1 b_1$ then the $a$-left representation of $g$ is $w = a_2 a_3 b_1 a_1 b_2 b_1$. The \textit{$b$-left representation} of $g$ is defined analogously. 
\end{definition}

We can map $G$ into $T_W$ via a map $F_\A: G \rightarrow T_W$ as follows. Let $g \in G$ and let $w \in W$ be the $a$-left representation of $g$. We may write 
\[w = u_1 a_1 u_2 a_2 u_3 a_3 \dots u_m a_m u_{m+1}\]
where $a_i \in \A$ and $u_i$ is a word on the alphabet $\B$. Then we define
\[F_\A(g) = \ol{u_1 a_1} \ \ol{u_2 a_2} \ \ol{u_3 a_3} \dots \ol{u_m a_m}\]
So, for example $F_\A(b_1 a_2 a_3 b_2 a_1 b_1) = \ol{a_2} \ \ol{a_3} \ \ol{b_1 a_1}$. 

We define $F_\B: G \rightarrow T_W$ analogously. Let $F: G \rightarrow T_W \times T_W$ be $F = F_1 \times F_2$. 

\begin{lemma} \label{lem.aleftrep}
Suppose $g,g' \in G$ and suppose they have $a$-left representations (respectively)
\[w = u_1 a_1 u_2 a_2 \dots u_p a_p u_{p+1} a_{p+1} \dots u_{p+m} a_{p+m} u_{p+m+1}\]
and
\[w' = u_1 a_1 u_2 a_2 \dots u_p a_p u_{p+1}' a_{p+1}' \dots u_{p+n}' a_{p+n}' u_{p+n+1}'\]
where $u_{p+1} a_{p+1} \neq u_{p+1}' a_{p+1}'$. Then one can prove that the letters $a_{p+1}, a_{p+2}, \dots, a_{p+m}$ and $a_{p+1}', a_{p+2}', \dots, a_{p+n}'$ are \textit{exactly} the letters in $\A$ which survive when you fully cancel the word $w^{-1}w'$. A similar statement can be said for the $b$-left representations.
\end{lemma}

\begin{proof}
This is is left to the reader.
\end{proof}

\begin{corollary}
$F: G \rightarrow T_W \times T_W$ is an isometric embedding. 
\end{corollary}

\begin{proof}
Let $g,g' \in G$. \Cref{lem.aleftrep} tells us that $F_\A$ counts the numbers of letters in $\A$ that are in a reduced word representing $g^{-1}g'$ and $F_\B$ counts the numbers of letters in $\B$ that are in a reduced word representing $g^{-1}g'$.
\end{proof}

Now, as discussed in \Cref{remark.criteria}, we can combine $(\leo)$ and $(\virgo)$, analogously to how $(\leo)$ and $(\taurus)$ are combined in \Cref{cor.leotaurus}, to create a $D: T_W \rightarrow T_\Omega$ associated to the data
\begin{itemize}
    \item $\scalf = \{\textrm{the final letter of } w_i\}$;
    \item $\jf = 2$;
    \item $\scall = 
    \begin{Bmatrix}
    \textrm{the final $12c$ letters of } w_i \\
    \textrm{the final $12c$ letters of } \mt(w_i) \\
    \end{Bmatrix}$;
    \item $\delta = 0$; 
    \item $\jl = 2$;
    \item $N = 18$;
    \item $\epsilon = 1$.
\end{itemize}
We imagine the single finite statistic and the two linear statistics as being applied to an arbitrary sentence $\ol{w_1} \ \ol{w_2} \dots \ol{w_i} \in T_W$. Further, in the linear statistics above note that $c$ is the \textit{variable} required in the definition of a linear statistic (see \Cref{def.lstat}) and not a constant. Let $M \geq 1$ be the associated constant so that $d_{T_\Omega}(D\alpha,D\beta)) \geq d_{T_W}(\alpha,\beta)/M$ for all $\alpha, \beta$ satisfying $\leo(\scalf,\jf)$ or $\virgo(\scall,\delta,\jl,N,\epsilon)$.  

\begin{theorem} \label{thm.h2example}
The composition
\[G \xrightarrow{F} T_W\times T_W \xrightarrow{D \times D} T_\Omega \times T_\Omega \]
is a quasiisometric embedding. 
\end{theorem}

\begin{proof}
Since $D$ is $1$-Lipschitz, the composition is coarsely Lipschitz. So we only need to worry about the lower bound of the quasiisometric inequality. 

Let $g, g' \in G$. Since we only care about the coarse geometry of $G$, we may assume that $d_G(g,g') \geq 12$. We may write $F(g) = (\alpha_1, \alpha_2)$ and $F(g') = (\beta_1, \beta_2)$ so that $d(F(g),F(g')) = d_{T_W}(\alpha_1,\beta_1) + d_{T_W}(\alpha_2,\beta_2)$. Without loss of generality, we may assume that $d_{T_W}(\alpha_1,\beta_1) \geq d_{T_W}(\alpha_2,\beta_2)$. So $d_{T_W}(\alpha_1,\beta_1) \geq \frac{1}{2} d(F(g),F(g'))$. We may write 
\[\alpha_1 = \ol{u_1 a_1} \ \ol{u_2 a_2} \dots \ol{u_p a_p} \ \ol{u_{p+1} a_{p+1}} \dots \ol{u_{p+m} a_{p+m}} \]
and
\[\beta_1 = \ol{u_1 a_1} \ \ol{u_2 a_2} \dots \ol{u_p a_p} \ \ol{u_{p+1}' a_{p+1}'} \dots \ol{u_{p+n}' a_{p+n}'} \]
where $u_{p+1} a_{p+1} \neq u_{p+1}' a_{p+1}'$ and where the $a_i$ and $a_i'$ are letters in $\A$ and the $u_i$ and $u_i'$ are words on $\B$. Note that $d_{T_W}(\alpha_1,\beta_1) = m+n$. 

\begin{proofclaim3} \label{3claim1}
We may assume that $m,n \geq (m+n) / 3$.
\end{proofclaim3}

\begin{proof}[Proof of \Cref{3claim1}]
Otherwise, we have $\absval{m - n} \geq (m+n)/3$ and so 
\[d_{T_\Omega}(D\alpha_1,D\beta_1) \geq (m+n)/3 = d_{T_W}(\alpha_1,\beta_1) / 3 \geq \frac{1}{6} d(F(g),F(g')) = \frac{1}{6} d_G(g,g')\]
where the first inequality follows from the fact that $D$ is height-preserving. So we are done in this case. 
\end{proof}

It follows that 
\begin{equation} \label{eq.mnlbhyp}
m,n \geq (m+n) / 3 \geq \frac{1}{6} d(F(g),F(g')) = \frac{1}{6} d_G(g,g')
\end{equation}
In particular, since $d_G(g,g') \geq 12$, we know that $m,n \geq 2$. 

\begin{proofclaim3} \label{3claim2}
We are done if the sentences $\alpha_1$, $\beta_1 \in T_W$ satisfy $\leo(\scalf,\jf)$ or $\virgo(\scall,\delta,\jl,N,\epsilon)$.
\end{proofclaim3}

\begin{proof}[Proof of \Cref{3claim2}]
If this occurs then we have
\[d_{T_\Omega}(D\alpha_1,D\beta_1) \geq  d_{T_W}(\alpha_1,\beta_1) / M \geq \frac{1}{2M} d(F(g),F(g')) = \frac{1}{2M} d_G(g,g')\]
and we are done. 
\end{proof}

If $a_{p+1} \neq a_{p+1}'$ or $a_{p+2} \neq a_{p+2}'$ then $\alpha_1$ and $\beta_1$ satisfy $\leo(\scalf,\jf)$ and we are done. So we may assume that $a_{p+1} = a_{p+1}'$ and $a_{p+2} = a_{p+2}'$. 

A crucial observation is that when we reduce (i.e. cancel as much as possible by commuting elements and deleting pairs of the form $a_i^2$ or $b_i^2$)
\begin{equation} \label{eq.reducer}
a_{p+m}^{-1} u_{p+m}^{-1} \dots a_{p+1}^{-1} u_{p+1}^{-1} u_{p+1}' a_{p+1}' \dots u_{p+n}' a_{p+n}'
\end{equation}
we are left with a geodesic from $g$ to $g'$. So any collection of letters in \eqref{eq.reducer} which do not cancel have cardinality bounded above by $d_G(g,g')$. 

\begin{proofclaim3} \label{3claim3}
At most one $\B$-letter in $u_{p+3}a_{p+3}\dots u_{p+m}a_{p+m}$, and at most one $\B$-letter in $u_{p+3}'a_{p+3}'\dots u_{p+n}'a_{p+n}'$, can cancel when we reduce \eqref{eq.reducer}. 
\end{proofclaim3}

\begin{proof}[Proof of \Cref{3claim3}]
This is left to the reader but the idea is that it is difficult for a $\B$-letter in $u_{p+3}a_{p+3}\dots u_{p+m}a_{p+m}$ to commute with both $a_{p+1}$ and $a_{p+2}$ (yet it needs to). 
\end{proof} 

\begin{proofclaim3} \label{3claim4}
When we reduce \eqref{eq.reducer}, there can only be cancellation in at most one of the words $u_{p+2}$ and $u_{p+2}'$. 
\end{proofclaim3}

\begin{proof}[Proof of \Cref{3claim4}]
For there to be cancellation in $u_{p+2}$, we must have that $u_{p+1}$ is a subword of $u_{p+1}'u_{p+2}'u_{p+3}'\dots u_{p+n}'$. If the final letter of $u_{p+1}$ is within $u_{p+2}'u_{p+3}'\dots u_{p+n}'$, then, for there to be any cancellation within $u_{p+2}$, this final letter must commute with $a_{p+1}'$. But then the final letter of $u_{p+1}$ would commute with $a_{p+1} = a_{p+1}'$ which contradicts the $a$-left representation. Thus $u_{p+1}$ is actually a subword of $u_{p+1}'$. In exactly the same way for there to be any cancellation in $u_{p+2}'$, we must have that $u_{p+1}'$ is a subword of $u_{p+1}$. Both these conditions cannot happen, and hence there can only be cancellation in at most one of $u_{p+2}$ and $u_{p+2}'$.
\end{proof}

\begin{proofclaim3} \label{3claim5}
We are done when $u_{p+2} a_{p+2} \neq u_{p+2}' a_{p+2}'$.
\end{proofclaim3}

\begin{proof}[Proof of \Cref{3claim5}]
We claim that $\virgo(\scall,\delta,\jl,N,\epsilon)$ holds with $j = 2$ as the choice of index. Obviously \ref{condition.virgo2} holds. 

Let us prove that \ref{condition.virgo1} holds. Since, by \Cref{3claim3}, at most one $\B$-letter in $u_{p+3}a_{p+3}\dots u_{p+m}a_{p+m}$ can reduce when we cancel \eqref{eq.reducer}, it follows that 
\[\length(u_{p+3}a_{p+3}\dots u_{p+m}a_{p+m}) - 1 \leq d_G(g,g')\]
In particular, $\length(u_{p+3}a_{p+3}\dots u_{p+m}a_{p+m}) \leq 2 d_G(g,g')$. Hence, by \eqref{eq.mnlbhyp}, we have
\[\awl(\ol{u_{p+3} a_{p+3}} \dots \ol{u_{p+m} a_{p+m}}) \leq \frac{2 d_G(g,g')}{d_G(g,g')/6} = 12 \leq N\]
Similarly, we have $\awl(\ol{u_{p+3}' a_{p+3}'} \dots \ol{u_{p+n}' a_{p+n}'}) \leq N$.

Finally, we need to prove that \ref{condition.virgo3} holds. It follows from \Cref{3claim4} and \Cref{lem.aleftrep} that $\length(u_{p+2}a_{p+2}) \leq d_G(g,g')$ or $\length(u_{p+2}'a_{p+2}') \leq d_G(g,g')$. Since $d_G(g,g') \leq 2(m+n)$, it follows that the linear statistic "the final $12c$ letters of $w_i$" with $c = m+n$ can distinguish between the sentences $\ol{u_1 a_1} \dots \ol{u_p a_p} \ \ol{u_{p+1} a_{p+1}} \ \ol{u_{p+2} a_{p+2}}$ and $\ol{u_1 a_1} \dots \ol{u_p a_p} \ \ol{u_{p+1}' a_{p+1}'} \ \ol{u_{p+2}' a_{p+2}'}$. 
\end{proof}

So we may assume that $u_{p+2} a_{p+2} = u_{p+2}' a_{p+2}'$. It follows from \Cref{3claim4} and \Cref{lem.aleftrep} that $\length(u_{p+2} a_{p+2}) \leq d_G(g,g')$ and $\length(u_{p+2}' a_{p+2}') \leq d_G(g,g')$. 

\begin{proofclaim3} \label{3claim6}
We are done when $u_{p+2} a_{p+2} = u_{p+2}' a_{p+2}'$.
\end{proofclaim3}

\begin{proof}[Proof of \Cref{3claim6}]
We claim that $\virgo(\scall,\delta,\jl,N,\epsilon)$ holds with $j = 1$ as the choice of index. Obviously \ref{condition.virgo2} holds. 

Let us prove that \ref{condition.virgo1} holds. As in the proof of \Cref{3claim5}, we have 
\[\length(u_{p+3}a_{p+3}\dots u_{p+m}a_{p+m}) \leq 2 d_G(g,g')\]
Further, we have $\length(u_{p+2} a_{p+2}) \leq d_G(g,g')$ and $\length(u_{p+2}' a_{p+2}') \leq d_G(g,g')$. Hence, by \eqref{eq.mnlbhyp}, we have
\[\awl(\ol{u_{p+2} a_{p+2}} \dots \ol{u_{p+m} a_{p+m}}) \leq \frac{3 d_G(g,g')}{d_G(g,g')/6} = 18 = N\]
Similarly, we have $\awl(\ol{u_{p+2}' a_{p+2}'} \dots \ol{u_{p+n}' a_{p+n}'}) \leq N$.

Finally, we need to prove that \ref{condition.virgo3} holds. Let $u$ be the largest common initial part of $u_{p+1} a_{p+1}$ and $u_{p+1}'a_{p+1}'$ and write $u_{p+1}a_{p+1} = uv$ and $u_{p+1}'a_{p+1}' = uv'$. 

We claim that $\length(v) \leq 3d_G(g,g')$ and $\length(v') \leq 3d_G(g,g')$. We have three cases: $v,v'$ can both be non-empty, $v$ can be non-empty and $v'$ can be empty, or $v'$ can be non-empty and $v$ can be empty. In the first case, it is not hard to see that both $v$ and $v'$ survive in the cancellation of \eqref{eq.reducer}. So $\length(v) \leq d_G(g,g')$ and $\length(v') \leq d_G(g,g')$. In the second case, when we reduce \eqref{eq.reducer}, we know that $v$ can only cancel with letters in $u_{p+2}'a_{p+2}'$ and $u_{p+3}'a_{p+3}'\dots u_{p+n}'a_{p+n}'$. So at most $d_G(g,g') + 1$ letters in $v$ can cancel in \eqref{eq.reducer}. So $\length(v) \leq 2 d_G(g,g') + 1 \leq 3d_G(g,g')$. We are done in the third case for similar reasons. 

Hence $d_{T_A}(u_{p+1} a_{p+1}, u_{p+1} a_{p+1}) \leq 6 d_G(g,g')$. It follows from \Cref{lem.nomt} that either the final $6d_G(g,g')$ letters of $u_{p+1} a_{p+1}$ and $u_{p+1}' a_{p+1}'$ are distinct or the final $6d_G(g,g')$ letters of the base $10$ expansions of $\length(u_{p+1} a_{p+1}) \in \N$ and $\length(u_{p+1}' a_{p+1}') \in \N$ are distinct. So the two linear statistics in $\scall$ with $c = m+n$ successfully distinguish $\ol{u_1 a_1} \dots \ol{u_p a_p} \ \ol{u_{p+1} a_{p+1}}$ from $\ol{u_1 a_1} \dots \ol{u_p a_p} \ \ol{u_{p+1}' a_{p+1}'}$ since $12(m+n) \geq 6 d_G(g,g')$.
\end{proof}
The proof of \Cref{thm.h2example} is complete.
\end{proof}

\bibliographystyle{alpha}
\bibliography{ref.bib}

\end{document}